\documentclass[12pt]{article}
\usepackage{amsmath, amsthm, amssymb, stmaryrd, fullpage}
\usepackage[all]{xy}

\newtheorem{theorem}{Theorem}[subsection]
\numberwithin{equation}{theorem}
\newtheorem{lemma}[theorem]{Lemma}
\newtheorem{cor}[theorem]{Corollary}
\newtheorem{prop}[theorem]{Proposition}
\theoremstyle{definition}
\newtheorem{defn}[theorem]{Definition}
\newtheorem{remark}[theorem]{Remark}
\newtheorem{hypothesis}[theorem]{Hypothesis}
\newtheorem{notation}[theorem]{Notation}
\newtheorem{construction}[theorem]{Construction}

\def\AAA{\mathbb{A}}
\def\CC{\mathbb{C}}
\def\DD{\mathbb{D}}
\def\FF{\mathbb{F}}
\def\PP{\mathbb{P}}
\def\QQ{\mathbb{Q}}
\def\RR{\mathbb{R}}
\def\ZZ{\mathbb{Z}}
\newcommand{\calC}{\mathcal{C}}
\newcommand{\calE}{\mathcal{E}}
\newcommand{\calF}{\mathcal{F}}
\newcommand{\calG}{\mathcal{G}}
\newcommand{\calO}{\mathcal{O}}
\newcommand{\gothm}{\mathfrak{m}}
\newcommand{\gotho}{\mathfrak{o}}

\newcommand{\del}{\partial}

\def\alg{\mathrm{alg}}

\def\dual{\vee}
\def\be{\mathbf{e}}
\def\bv{\mathbf{v}}
\def\cEnd{\mathit{End}}

\DeclareMathOperator{\Aut}{Aut}

\DeclareMathOperator{\Frac}{Frac}
\DeclareMathOperator{\Gal}{Gal}
\DeclareMathOperator{\GL}{GL}
\DeclareMathOperator{\height}{height}

\DeclareMathOperator{\perf}{perf}
\DeclareMathOperator{\rank}{rank}
\DeclareMathOperator{\ratrank}{ratrank}
\DeclareMathOperator{\semis}{ss}
\DeclareMathOperator{\sep}{sep}
\DeclareMathOperator{\Spec}{Spec}

\DeclareMathOperator{\Spf}{Spf}
\DeclareMathOperator{\trdefect}{trdefect}
\DeclareMathOperator{\trdeg}{trdeg}

\begin{document}

\title{Semistable reduction for overconvergent $F$-isocrystals, IV:
Local semistable reduction at nonmonomial valuations}
\author{Kiran S. Kedlaya \\ Department of Mathematics, 
Room 2-165 \\ Massachusetts
Institute of Technology \\ 77 Massachusetts Avenue \\
Cambridge, MA 02139 \\
\texttt{kedlaya@mit.edu}}
\date{July 22, 2010}

\maketitle

\begin{abstract}
We complete our proof that given an overconvergent $F$-isocrystal
on a variety over a field of positive characteristic, one can pull back
along a suitable generically finite cover to obtain an isocrystal
which extends, with logarithmic singularities and nilpotent residues,
to some complete variety.
We also establish an analogue for $F$-isocrystals overconvergent inside
a partial compactification.
By previous results, this reduces to solving a local problem
in a neighborhood of a valuation of 
height 1 and residual transcendence degree 0.
We do this by studying the variation of some numerical invariants
attached to $p$-adic differential modules, analogous to the irregularity
of a complex meromorphic connection. This allows for an induction on
the transcendence defect of the valuation, i.e., the discrepancy between
the dimension of the variety and the rational rank of the valuation.
\end{abstract}

\tableofcontents

\section{Introduction}

This paper is the fourth and last of a series, preceded by 
\cite{kedlaya-part1, kedlaya-part2, kedlaya-part3}.
The goal of the series is to prove a 
``semistable reduction'' theorem for overconvergent $F$-isocrystals,
a class of $p$-adic analytic objects associated to schemes of finite type
over a field of characteristic $p>0$.
Such a theorem is expected\footnote{This 
expectation has been 
confirmed by Caro and Tsuzuki \cite{caro-tsuzuki}, who use the 
semistable reduction
theorem to show that Caro's category
of overholonomic arithmetic $\mathcal{D}^\dagger$-modules both
contains the overconvergent $F$-isocrystals and is stable under
standard cohomological operations. See also \cite{caro}.} 
to have consequences for the theory of rigid
cohomology, in which overconvergent $F$-isocrystals play the role of
coefficient objects.

In \cite{kedlaya-part1}, it was shown that the problem of extending
an overconvergent isocrystal on a variety $X$ to a log-isocrystal
on a larger variety $\overline{X}$ is governed by the triviality of a sort of
local monodromy along components of the complement of $X$.
In \cite{kedlaya-part2}, it was shown that the problem can be localized on 
the space of valuations on the function field of the given variety,
and that it suffices to work in the neighborhood of a valuation which is
minimal (height 1, residual transcendence degree 0).
In \cite{kedlaya-part3}, this resulting local problem was solved at monomial
valuations (minimal Abhyankar valuations).
In this paper, we solve the local problem at nonmonomial minimal valuations;
this completes the proof of semistable reduction for overconvergent
$F$-isocrystals, and also yields an analogous theorem for partially
overconvergent $F$-isocrystals.

The introductions of \cite{kedlaya-part1, kedlaya-part2, kedlaya-part3}
provide context
for what is proved in this paper; we will not repeat that context here.
Instead, we devote
the remainder of this introduction to an overview of the results
specific to this paper, and a survey of the structure of the
various chapters of the paper.

\subsection{Local problems}

The problem of (global) semistable reduction is to show that
an overconvergent $F$-isocrystal on a nonproper $k$-variety
can be extended to a log-$F$-isocrystal on a proper $k$-variety
after pulling back along a suitable generically finite cover. 
The results of \cite{kedlaya-part1, kedlaya-part2} imply that this
problem can be localized within
the Riemann-Zariski space of valuations of the function field of the
variety, and that it suffices to consider neighborhoods of valuations
which are minimal (of height 1 and residual transcendence degree 0).

In \cite{kedlaya-part3}, this problem was solved for monomial (Abhyankar)
minimal valuations using a generalization of
the $p$-adic local monodromy theorem to so-called ``fake annuli''.
However, already for infinitely singular valuations on surfaces this 
approach is unsuitable, because of 
of the lack of convenient lifts from characteristic $p$ to characteristic $0$
in neighborhoods of such valuations.

\subsection{Valuation-theoretic induction}

Our strategy for proving local semistable reduction at nonmonomial minimal
valuations is to proceed by
induction on the transcendence defect 
(the discrepancy in Abhyankar's inequality)
of the valuation, using the monomial case as a base. 
Namely, given an irreducible variety $X$ and a minimal
nonmonomial valuation $v$ on the function field $k(X)$, after altering
(in the sense of de Jong) we can construct a dominant morphism
$X \to X^0$ of relative dimension 1, such that the restriction $v^0$ of $v$
to $k(X^0)$ has the same rational rank as $v$. Consequently, the 
transcendence defect of 
$v^0$ is one less than that of $v$. Moreover, we can construct a path in
valuation space terminating at $v$ and otherwise passing only through 
valuations of lower transcendence defect. Using this path, we can
reduce local semistable reduction at $v$ to local semistable reduction
at a valuation of lower transcendence defect; this  uses an analysis
of variation of differential invariants, described below.

Before continuing, we should add some remarks about this 
valuation-theoretic induction argument. While it seems quite natural,
we did not find any analogous argument in the literature before
preparing this paper. However, after this paper was largely completed,
we learned of Temkin's proof of local uniformization for any function field
in positive characteristic after a purely inseparable extension
\cite{temkin}, in which a very similar inductive argument is used.
We subsequently used the inductive method to describe the formal
structure of flat meromorphic connections on complex algebraic
and analytic varieties \cite{kedlaya-goodformal2}.
We suspect that there are additional problems susceptible to this 
strategy,
e.g., in the valuation-theoretic study of plurisubharmonic singularities
\cite{boucksom-favre-jonsson}.

\subsection{Variation of differential invariants}

Assuming local semistable reduction for all valuations of 
transcendence defect less than that
of $v$, we can then establish local semistable reduction at $v$
by analyzing the variation of a certain quantity associated to the
derivation in the fibral direction. This is similar to the analysis of the
differential Swan conductor introduced in \cite{kedlaya-swan1},
except that there one considers all derivations. Measuring only in one 
direction 
creates some initial difficulties, but ultimately allows to make
the analysis at a deep enough level to avoid getting entangled in
valuation-theoretic complications.

To simplify the analysis, we split the discussion into two parts.
We first consider a highly abstracted situation, using rings
with one power series variable over a field admitting a family of distinct
norms. Here we obtain a partial analogue of the usual
$p$-adic local monodromy theorem, using detailed results
from \cite{kedlaya-course} concerning
ordinary $p$-adic differential equations. We then introduce the
additional geometry, which allows us to obtain a closer analogue of the
monodromy theorem, and to carry out the induction on transcendence defect.

\subsection{Structure of the paper}

We conclude this introduction with a summary of the structure of the
paper.

In Section~\ref{sec:prelim}, we install some preliminary definitions
in valuation theory, give a brief discussion of the Berkovich unit disc,
state the local and global semistable reduction theorems,
and formulate a notion of the (semisimplified) local monodromy representation
associated to an $F$-isocrystal.

In Section~\ref{sec:diffmod}, we recall a number of facts from the theory
of differential modules on a one-dimensional disc or annulus. These are mostly
taken from the book \cite{kedlaya-course}, which is based on
a course given by the author in the fall of 2007.

In Section~\ref{sec:relative}, we make the abstracted relative analysis
described above, concerning
differential modules over a ring with one power series
variable and a coefficient field equipped with a family of norms.

In Section~\ref{sec:corank}, we specialize the relative analysis to the
case coming from an $F$-isocrystal, and use it to prove local semistable
reduction by induction on transcendence defect.

In the Appendix, we record some corrections to
\cite{kedlaya-part1} and \cite{kedlaya-part2}, and reflect upon the
necessity of the discreteness hypothesis on the coefficient field.

\setcounter{theorem}{0}
\begin{notation}
We retain the basic notations of \cite{kedlaya-part1, kedlaya-part2, kedlaya-part3}. 
In particular,
$k$ will always denote a field of characteristic $p>0$, $K$ will denote a complete discretely valued
field of characteristic zero with residue field $k$, equipped with an
endomorphism $\sigma_K$ lifting the $q$-power Frobenius for some power
$q$ of $p$, $\gotho_K$ will denote the ring of integers of $K$,
and $\gothm_K$ will denote the maximal ideal of $K$.
\end{notation}

\subsection*{Acknowledgments}
Thanks to Atsushi Shiho for finding and helping to correct the error in 
\cite{kedlaya-part1} reported in the appendix.
Thanks also to the referee for detailed comments,
particularly on Proposition~\ref{P:identically zero}
and on \S~\ref{sec:corank} as a whole,
and for pointing out the error in \cite{kedlaya-part2} reported in the
appendix.
The author was supported by 
NSF CAREER grant DMS-0545904, a Sloan Research Fellowship,
and the NEC Research Support Fund.

\section{Preliminaries}
\label{sec:prelim}

In this section, we introduce some valuation theory, then
recall the local approach to proving semistable reduction.
This culminates
with the statement of the global semistable reduction theorem
we will be proving.

\subsection{Valuations}

In order to introduce local analogues of the problem of semistable reduction,
we recall some terminology concerning valuations.

\begin{hypothesis}
Throughout this subsection, let
$F$ be a finitely generated field extension of $k$, and let
$v: F^* \to \Gamma$ be a surjective Krull valuation over $k$
(i.e., with $k^* \subseteq \ker(v)$). Let $\gotho_v$ and $\kappa_v$
denote the valuation ring and residue field of $v$.
\end{hypothesis}

These definitions are standard; they were first recalled in this series
in \cite[\S 2]{kedlaya-part2}.
\begin{defn}
The \emph{Riemann-Zariski space} $S_{F/k}$ 
is the set consisting of the equivalence classes of
Krull valuations on $F$ over $k$. We topologize $S_{F/k}$
with the 
\emph{patch topology} (or \emph{Zariski-Hausdorff topology}), in which bases
are given by sets of the form
\[
\{
w \in S_{F/k}: w(x_1) \geq 0, \dots, w(x_m) \geq 0; \quad w(y_1) > 0,
\dots, w(y_n) > 0
\}
\]
for $x_1, \dots, x_m, y_1, \dots, y_n \in F$.
If $E/F$ is an arbitrary field extension, then the restriction map
$S_{E/k} \to S_{F/k}$ is continuous and surjective.
\end{defn}

\begin{defn}
The \emph{height} of $v$, denoted $\height(v)$,
is the rank of the totally ordered group $\Gamma$,
i.e., the number of isolated subgroups of $\Gamma$.
The height is 0 only for $v$ trivial; for $v$ nontrivial, the height is 1
if and only if $\Gamma$ is isomorphic as an ordered
group to a subgroup of $\RR$. We will regularly make such an identification
implicitly, even though it is only well-defined up to a rescaling.
The \emph{rational rank} of $v$,
denoted $\ratrank(v)$, is the
rational rank of $\Gamma$, i.e., the dimension of $\Gamma \otimes_{\ZZ} \QQ$
as a vector space over $\QQ$. This is always greater than or equal to
the height of $v$.
\end{defn}

The following is a consequence of
Abhyankar's inequality \cite[Theorem~2.5.2]{kedlaya-part2},
or an older result of Zariski
\cite[Appendix~2, Corollary, p.\ 334]{zariski-samuel}.
\begin{prop}
We have
\[
\ratrank(v) + \trdeg(\kappa_v/k) \leq \trdeg(F/k),
\]
where $\trdeg(A/B)$ denote the transcendence degree of the field extension
$A$ of $B$. In particular, both quantities on the left are finite.
Moreover, if equality holds, then the group $\Gamma$ is finitely generated.
\end{prop}

This definition was introduced in \cite[Definition~4.3.2]{kedlaya-part2}.
\begin{defn}
We say $v$ is \emph{minimal} if $\height(v) = 1$ and $\trdeg(\kappa_v/k) = 0$.
\end{defn}

This definition is not quite standard; see 
Remark~\ref{R:corank} below.
\begin{defn} \label{D:corank}
The \emph{transcendence defect} of $v$ is the difference
\[
\trdefect(v) = \trdeg(F/k) - \ratrank(v) - \trdeg(\kappa_v/k),
\]
which is nonnegative by Abhyankar's inequality. 
We say $v$ is an 
\emph{Abhyankar valuation} if $\trdefect(v) = 0$; if $v$ is both minimal
and Abhyankar, we say that $v$ is \emph{monomial}.
\end{defn}

\begin{remark} \label{R:corank}
While the notion of transcendence defect is quite natural given
Abhyankar's inequality, it does not seem to have a generally
accepted name. 
For instance, we used the term \emph{corank} in early drafts of this
paper, while the term \emph{defect rank} was used in early drafts of
Temkin's paper on inseparable local uniformization \cite{temkin}.
Our present terminology is based on the fact that
those valuations for which Abhyankar's inequality becomes
an equality, while commonly known as \emph{Abhyankar valuations}
in the context of algebraic geometry, are often referred to as
\emph{valuations without transcendence defect} in the context
of the model theory of valued fields. For instance, this 
terminology occurs in Kuhlmann's
generalization of the Grauert-Remmert stability theorem
\cite{kuhlmann}.
The reference to defect also makes a link with other numerical 
quantities that measure pathological behavior of extensions of valued fields,
such as the classical \emph{henselian defect}
of Ostrowski, and the \emph{vector space defect}
of Green, Matignon, and Pop \cite{gmp}.
\end{remark}

For proofs of the following remarks, see \cite[Chapter~6]{ribenboim}.
\begin{remark} \label{R:extension}
For any extension $F'$ of $F$, there exists at least one extension
$v'$ of $v$ to a Krull valuation on $F'$
\cite[Definition~2.2.8]{kedlaya-part2}.
In case $F'$ is an algebraic extension of $F$,
so that $\trdeg(F'/k) = \trdeg(F/k)$, the extension preserves
numerical invariants in the following sense.
Since $\Gamma_{v'}$ is contained in the divisible closure of $\Gamma_v$,
we have $\height(v') = \height(v)$ and
$\ratrank(v') = \ratrank(v)$. Since $\kappa_{v'}$ is
algebraic over $\kappa_v$, we have $\trdeg(\kappa_{v'}/k) 
= \trdeg(\kappa_v/k)$. Putting these together,
we deduce $\trdefect(v') = \trdefect(v)$.
\end{remark}

\subsection{The Berkovich closed unit disc}

Since our strategy for proving local semistable reduction will be
induction on transcendence defect, 
we will need a method for comparing a valuation
on a field with the restriction to a subfield, in case the transcendence
defects differ
by 1. For this, it is most convenient to use Berkovich's 
notion of the closed unit disc over
a complete (but not necessarily algebraically closed) nonarchimedean field.

\begin{hypothesis}
Throughout this 
subsection, let $F$ be a field complete for a nonarchimedean norm $|\cdot|_F$,
corresponding to a real valuation $v_F$. Let $\CC$ denote the completed
algebraic closure of $F$.
Let $\Aut(\CC/F)$ denote the group of 
\emph{continuous} automorphisms of $\CC$ over $F$.
\end{hypothesis}

\begin{defn}
The Berkovich closed unit disc $\DD = \DD_F$ consists of the 
multiplicative seminorms $|\cdot|$ on $F[x]$ which are compatible with the
given norm on $F$ and bounded above by
the 1-Gauss norm. 
In particular, the 1-Gauss norm itself is a point 
$\alpha_\DD$ of $\DD$,
called the \emph{Gauss point}.
(We will often use Greek letters to refer to points of $\DD$.
When we want to emphasize the fact that $\alpha \in \DD$ represents
a function $F[x] \to \RR$, we notate it as $|\cdot|_\alpha$.)
The \emph{weak topology} (or \emph{Gel'fand topology})
on $\DD$ is the weakest topology 
under which evaluation at any element of $F[x]$ gives a continuous
function $\DD \to [0,+\infty)$.
\end{defn}

\begin{lemma} \label{L:base change1}
For any complete extension $F'$ of $F$, the restriction
map $\DD_{F'} \to \DD_F$ is surjective.
Moreover, $\DD_\CC \to \DD_F$ induces a bijection
$\DD_\CC / \Aut(\CC/F) \to \DD_F$.
\end{lemma}
\begin{proof}
See \cite[Corollary~1.3.6]{berkovich}.
\end{proof}

\begin{defn}
For $z \in \gotho_\CC$ and $r \in [0,1]$, the function $|\cdot|_{z,r}: 
F[x] \to [0, +\infty)$ given by taking $P(x)$ to the $r$-Gauss norm
of $P(x+z)$ is a seminorm; let $\alpha_{z,r}$ be the associated point of $\DD$.
In particular, the map $z \mapsto \alpha_{z,0}$ induces an injection
$\gotho_{\CC} / \Aut(\CC/F) \hookrightarrow \DD$.
Define the closed disc $D_{z,r} = \{z' \in \CC: |z'-z| \leq r\}$. 
\end{defn}

\begin{lemma} \label{L:disjoint discs}
Let $P \in F[x]$ be irreducible, and suppose $z_1 \in \gotho_\CC$ 
is a root of $P$.
Suppose $r_1 \in [0,1]$ and $z_2 \in \gotho_\CC$ are such that
$|z-z_2| > r_1$ for each root $z$ of $P$. Then 
\[
|P(z_2)| > |P|_{z_1,r_1}.
\]
\end{lemma}
\begin{proof}
By applying an element of $\Aut(\CC/F)$, we may assume that the minimum of
$|z-z_2|$, for  $z$ running over roots of $P$, is achieved by $z = z_1$.
By hypothesis, we have $|z_2 - z| > r_1$. On the other hand,
since $|z_2 - z| \geq |z_2 - z_1|$, we have also $|z_2-z| =
\max\{|z_2-z|, |z_1 - z_2|\} \geq |z_1 - z|$.
If we factor $P(x) = c \prod_z (x - z)$,
we obtain
\begin{align*}
|P(z_2)| &= |c| \prod_z |z_2-z| \\
&= |c| \prod_z \max\{r_1, |z_2-z|\} \\
&> |c| \prod_z \max\{r_1, |z_1-z|\} = |P|_{z_1,r_1};
\end{align*}
here the strict inequality occurs because the inequality
$\max\{r_1, |z_2-z|\} \geq \max\{r_1, |z_1-z|\}$
is strict for $z = z_1$.
\end{proof}

\begin{defn}
For $\alpha, \beta \in \DD$, we say that \emph{$\alpha$ dominates
$\beta$}, denoted $\alpha \geq \beta$,
if for all $P \in F[x]$, we have $|P|_\alpha \geq |P|_\beta$.
The relation of dominance is transitive, and the
Gauss point $\alpha_\DD$ is maximal.
Dominance is also stable under base change in an appropriate sense; see
Lemma~\ref{L:base change2}
below. (This assertion is made without proof in 
\cite[Remark~4.2.3]{berkovich}.) If $F = \CC$, then 
$\alpha_{z,r} \geq \alpha_{z',r'}$ if and only if 
$D_{z',r'} \subseteq D_{z,r}$:
namely, 
$\alpha_{z,r} \geq \alpha_{z',r'}$ 
implies $r = |x-z|_{z,r} \geq |x-z|_{z',r'} = \max\{r', |z-z'|\}$.
\end{defn}

Here is Berkovich's classification of points of $\DD$; we will use it repeatedly
in what follows without explicit citation.
\begin{prop} \label{P:classify points}
Each element of $\DD$ is of exactly one of the following four types.
\begin{enumerate}
\item[(i)]
A point of the form $\alpha_{z,0}$ for some $z \in \gotho_\CC$.
\item[(ii)]
A point of the form $\alpha_{z,r}$ for some $z \in \gotho_\CC$
and $r \in (0,1] \cap |\CC^*|$.
\item[(iii)]
A point of the form $\alpha_{z,r}$ for some $z \in \gotho_\CC$
and $r \in (0,1] \setminus
 |\CC^*|$.
\item[(iv)]
The infimum of a sequence $\alpha_{z_i,r_i}$ in which the discs $D_{z_i,r_i}$
form a decreasing sequence
with empty intersection and positive limiting radius. (This type does not
occur if $\CC$ is spherically complete.)
\end{enumerate}
Moreover, the points which are minimal under domination in $\DD_\CC$
are precisely those of type (i) and (iv).
\end{prop}
\begin{proof}
For $F = \CC$, this is included in \cite[1.4.4]{berkovich} except for
the following points.
\begin{itemize}
\item
We must check that 
the limiting radius in (iv) must be positive. This holds because
$\CC$ is complete, so the intersection of any decreasing sequence
of closed discs with zero limiting radius is nonempty, and hence
defines a point of type (i) and not (iv). In a similar vein, if
$\CC$ is spherically complete, then any decreasing sequence of closed
discs has nonempty intersection, so there are no points of type (iv) at all.
\item
We must also check that points of type (iv) are minimal.
(By contrast, it is obvious that points of type (i) are minimal 
whereas points of type (ii) and (iii) are not.)
Let $\alpha, \beta \in \DD$ be such that $\alpha$ is of type (iv) and
$\alpha \geq \beta$.
Let $\CC'$ be a spherically complete algebraically closed extension of $\CC$,
which exists by a result of Krull \cite[Satz~24]{krull}.
By Lemma~\ref{L:base change1}, we can find $\beta' \in \DD_{\CC'}$
extending $\beta$; since $\DD_{\CC'}$ has no points of type (iv),
we can write $\beta' = \alpha_{z',r'}$ for some $z' \in \gotho_{\CC'}$
and some $r' \in [0,1]$. 

Let $\alpha_{z_i,r_i} \in \DD$ be a decreasing sequence of points with
infimum $\alpha$, such that the $D_{z_i,r_i}$ have empty intersection.
Then for any $z \in \CC$, we can find $i_0$ such that
$z \notin D_{z_i,r_i}$ for $i \geq i_0$. For such $i$, we have
\[
|x-z|_{z_i,r_i} = \max\{r_i, |z_i-z|\} = |z_i - z|.
\]
For $j \geq i \geq i_0$, we have $z_j \in D_{z_i,r_i}$, so 
$|z_i-z_j| \leq r_i$ and $|z_j-z| = |z_j - z_i + z_i - z| = |z_i - z|$. 
In particular, $|z_i-z| = |z_{i_0} - z|$, so
\[
|x-z|_\alpha = \inf_{i \geq i_0} \{|x-z|_{z_i,r_i}\} = |z_{i_0} - z|.
\]
On the other hand, we have 
\begin{align*}
r_i &= |x-z_i|_{z_i,r_i} \\
&\geq |x-z_i|_\beta = |x-z_i|_{\beta'} \\
&= |x-z_i|_{z',r'} \\
&= \max\{r', |z'-z_i|\}.
\end{align*}
Consequently, $|z'-z_i| \leq r_i$. For $i \geq i_0$,
we have $|z-z_i| > r_i \geq r'$ and
\[
|x-z|_\beta = \max\{r', |z-z_i|\} = |z-z_i| = |x-z|_\alpha.
\]
Since $\CC$
is algebraically closed, this implies that $|P|_\beta = |P|_\alpha$
for all $P \in \CC[x]$, so $\alpha = \beta$.
\end{itemize}

For general $F$, note that the action of $\tau \in \Aut(\CC/F)$ on $\CC$
carries $\alpha_{z,r}$ to $\alpha_{\tau(z),r}$.
This implies that the action of $\Aut(\CC/F)$ on $\DD$
preserves types in the classification.
\end{proof}

\begin{remark}
For $\alpha \in \DD$ of type (ii) or (iii), we can always write
$\alpha = \alpha_{z,r}$ with $z \in F^{\alg}$. In fact, we can even
choose $z \in F^{\sep}$; this follows from the fact that if $P(x) \in F[x]$
is an inseparable polynomial with $z$ as a root, then for $c \in F$
sufficiently small, $P(x) + cx$ is a separable polynomial having a root
$z'$ with $|z-z'| < r$.
\end{remark}

\begin{lemma} \label{L:base change2}
Suppose $\alpha, \beta \in \DD$ satisfy $\alpha \geq \beta$. Then
there exist $\alpha', \beta' \in \DD_\CC$ restricting to $\alpha, \beta$,
respectively, such that $\alpha' \geq \beta'$. (Note that by applying
$\Aut(\CC/F)$, we may prescribe either $\alpha'$ or $\beta'$, but not
both.)
\end{lemma}
\begin{proof}
Suppose first that $\alpha = \alpha_{z_1,r_1}$ and $\beta = \alpha_{z_2,r_2}$
for some $z_1,z_2 \in \gotho_{\CC}$ and $r_1,r_2 \in [0,1]$ with $r_1 > 0$.
We may choose $z_1$ to be integral over $F$, with minimal polynomial $P$.
If there exist no extensions $\alpha', \beta' \in \DD_\CC$ of $\alpha,\beta$
satisfying $\alpha' \geq \beta'$, then $|z_2-z| > r_1$ for each root $z$ of $P$.
By Lemma~\ref{L:disjoint discs},
\[
|P|_{z_2,r_2} \geq |P|_{z_2,0} = |P(z_2)| > |P|_{z_1,r_1},
\]
contradicting the hypothesis $\alpha \geq \beta$.

Suppose next that $\alpha = \alpha_{z_1,r_1}$ with $r_1 > 0$,
but that $\beta$ is of type (iv).
We may repeat the previous argument by constructing 
a spherically complete and algebraically closed extension $\CC'$ of $F$
(as in Proposition~\ref{P:classify points}),
applying Lemma~\ref{L:base change1} to construct $\beta' \in \DD_{\CC'}$
restricting to $\beta$,
choosing $z_2 \in \CC'$ with $\beta' = \alpha_{z_2,r_2}$ 
for some $r_2 \in [0,1]$,
proceeding as above, then restricting back to $\CC$.

Finally, suppose that $\alpha$ is of type (i) or
(iv).
Write $\alpha$ as the infimum of a sequence $\alpha_{z_i,r_i}$ with 
each $z_i \in F^{\alg}$ and each $r_i > 0$; then
$\alpha_{z_i,r_i} \geq \beta$ for each $i$. Let $\beta'$ be any extension of 
$\beta$ to $\DD_\CC$. By the above argument, for each $l$, we can choose
conjugates $z_{1,l},\dots,z_{l,l}$ of $z_1,\dots,z_l$
such that 
\[
\alpha'_{z_{1,l},r_1}\geq \cdots \geq \alpha'_{z_{l,l},r_l} \geq \beta'.
\]
Since $z_1$ has only finitely many conjugates, we can choose a conjugate
$z_{1,\infty}$ that occurs
as $z_{1,l}$ for infinitely many $l$. We can then
choose a conjugate $z_{2,\infty}$ such that $z_{1,\infty},z_{2,\infty}$
occur as $z_{1,l},z_{2,l}$ for infinitely many $l$, and so on.
We thus build a sequence of $z_{i,\infty}$ such that
$\alpha'_{z_{i,\infty},r_i} \geq \alpha'_{z_{i+1,\infty},r_{i+1}} \geq \beta'$
for all $i$. The infimum $\alpha'$ of the $\alpha'_{z_{i,\infty},r_i}$
is then an extension of $\alpha$ dominating $\beta'$.
\end{proof}
\begin{cor}
If $\alpha,\beta,\gamma \in \DD$ satisfy
$\alpha > \gamma$ and $\beta > \gamma$, then 
there exists a single $z \in \gotho_\CC$ and some $r_1, r_2 \in (0,1]$
for which $\alpha = \alpha_{z,r_1}$ and $\beta = \alpha_{z,r_2}$.
In particular,
either $\alpha \geq \beta$ or $\beta \geq \alpha$ (so $\DD$ forms a tree
under domination, rooted at $\alpha_\DD$).
\end{cor}
\begin{proof}
The classification in Proposition~\ref{P:classify points}
(plus Lemma~\ref{L:base change2}, which implies that domination can be checked
over $\CC$) implies that $\alpha = \alpha_{z_1,r_1}$
and $\beta = \alpha_{z_2,r_2}$ for some $z_1, z_2 \in \gotho_\CC$. 
If $z_1 \in D_{z_2,r_2}$, then we can replace $z_2$ by $z_1$,
and conversely if $z_2 \in D_{z_1,r_1}$. If neither occurs, then
$|z_1-z_2| > \max\{r_1,r_2\}$, so
the two discs $D_{z_1,r_1}$ and $D_{z_2,r_2}$ are disjoint.
In this case, $\alpha$ and $\beta$ cannot both dominate $\gamma$.
\end{proof}

\begin{defn}
For $\alpha \in \DD$, we define the \emph{radius} of $\alpha$, denoted
$r(\alpha)$, as 
\[
r(\alpha) = \inf\{r \in [0,1]: \alpha_{z,r} \geq \alpha \mbox{ for some
$z \in \CC$}\}.
\]
For $z \in \gotho_\CC$ and $r \in [0,1]$, evidently $r(\alpha_{z,r}) = r$.
For $\alpha$ of type (iv), $r$ is the infimum of the $r_i$
for any sequence $\alpha_{z_i,r_i}$ as in Proposition~\ref{P:classify points}.
Note that if $\alpha \geq \beta$, then $r(\alpha) \geq r(\beta)$.
\end{defn}

\begin{lemma} \label{L:unique dominate}
For any $\alpha \in \DD$, 
for each $r \in [r(\alpha),1]$, there is a unique point $\alpha_r$
with $r(\alpha) = r$ and $\alpha_r \geq \alpha$. 
In particular, if $\alpha \geq \beta$ and $r(\alpha) = r(\beta)$, then 
$\alpha = \beta$.
\end{lemma}
\begin{proof}
By Lemma~\ref{L:base change2}, we may reduce to the case $F = \CC$.
We first prove existence; for this, we need only consider the case $r > r(\alpha)$,
as the case $r = r(\alpha)$ is obvious. Let $\alpha_{z_i,r_i} \in \DD$
be a decreasing sequence of points with infimum $\alpha$. Then
$\lim_{i \to \infty} r_i = r(\alpha)$, so we can find $i$ for which
$r_i < r$. For such $i$, we have $\alpha_{z_i,r} \geq \alpha_{z_i,r_i} \geq \alpha$.

We next prove uniqueness. We first note that if $\alpha_r \geq \alpha$
and $\alpha_r \neq \alpha$, then $\alpha_r$ is nonminimal and so by
Proposition~\ref{P:classify points} cannot be of type (iv). It thus suffices
to check that:
\begin{enumerate}
\item[(a)]
there is at most one $\alpha_r$ of a given radius of type other than (iv)
with $\alpha_r \geq \alpha$;
\item[(b)]
if $r = r(\alpha)$ and $\alpha$ is of type (iv), then there is no
$\alpha_r$ of radius $r$ and type other than (iv) with $\alpha_r \geq \alpha$.
\end{enumerate}

We first prove (a). Suppose $z,z' \in \gotho_\CC$ satisfy
$\alpha_{z,r} \geq \alpha$ and $\alpha_{z',r} \geq \alpha$.
If $\alpha$ is not of type (iv), then 
$\alpha = \alpha_{z'',r''}$ for some $z'', r''$,
so the intersection $D_{z,r} \cap D_{z',r}$
contains $z''$. In particular, this intersection is nonempty,
so $D_{z,r} = D_{z',r}$ and $\alpha_{z,r} = \alpha_{z',r}$.
If $\alpha$ is of type (iv), let $\alpha_{z_i,r_i} \in \DD$
be a decreasing sequence of points with infimum $\alpha$ such that the
$D_{z_i,r_i}$ have empty intersection.
Choose $i_0$ large enough so that
$z,z' \notin D_{z_i,r_i}$ for $i \geq i_0$. 
As in the proof of Proposition~\ref{P:classify points}, we then find that
$|z_i - z| = |x-z|_\alpha \leq |x-z|_{z,r} = r$,
and similarly $|z_i-z'| \leq r$. Again, $D_{z,r} \cap D_{z',r}$ is nonempty,
so $D_{z,r} = D_{z',r}$ and $\alpha_{z,r} = \alpha_{z',r}$.

We next prove (b). Suppose $\alpha$ is of type (iv),
$r = r(\alpha)$, and $\alpha_{z,r} \geq \alpha$ for some $z \in \gotho_\CC$.
Let $\alpha_{z_i,r_i} \in \DD$
be a decreasing sequence of points with infimum $\alpha$ such that the
$D_{z_i,r_i}$ have empty intersection.
Choose $i_0$ large enough so that
$z \notin D_{z_i,r_i}$ for $i \geq i_0$. 
As in the proof of Proposition~\ref{P:classify points},
we have $|z_i-z| = |x-z|_\alpha \leq |x-z|_{z,r} = r \leq r_i$,
so $z \in D_{z_i,r_i}$. But then the $D_{z_i,r_i}$ have nonempty intersection,
giving a contradiction.
\end{proof}

\begin{defn}
By Lemma~\ref{L:unique dominate},
for any $\alpha \in \DD$, 
for each $r \in [r(\alpha),1]$, there is a unique point $\alpha_r$
with $r(\alpha) = r$ and $\alpha_r \geq \alpha$. 
We define the
\emph{generic path} as the subset of $\DD$ consisting of these $\alpha_r$;
it is homeomorphic to an interval with endpoints $\alpha_\DD,\alpha$.
\end{defn}

\begin{lemma} \label{L:complete type 3}
Let $\alpha \in \DD$ be a point of type (iii) corresponding to a disc
containing an $F$-rational point $z$. Then $F[x-z, (x-z)^{-1}]$ is dense
in $F(x)$ under $|\cdot|_\alpha$.
\end{lemma}
\begin{proof}
We may write $\alpha = \alpha_{z,r}$
for some $r \in (0,1] \setminus |\CC^*|$. Then the $r$-Gauss norm of
any nonzero element $P = \sum_i P_i (x-z)^i \in F[x]$ is always 
achieved by exactly
one monomial $P_j (x-z)^j$, so the series
\[
\sum_{h=0}^\infty (P_j (x-z)^j - P)^h (P_j (x-z)^j)^{-h-1}
\]
converges in the completion of $F[x-z,(x-z)^{-1}]$ under $|\cdot|_\alpha$.
We can thus approximate the reciprocal of $P$ using
elements of $F[x-z,(x-z)^{-1}]$, so $F[x-z,(x-z)^{-1}]$ is dense in $F(x)$
under $|\cdot|_\alpha$.
\end{proof}
\begin{cor} \label{C:complete type 3}
With notation as in Lemma~\ref{L:complete type 3}, the completion of $F(x)$
under $|\cdot|_\alpha$ may be represented as the ring of formal sums
$\sum_{i \in \ZZ} c_i (x-z)^i$ with $c_i \in F$, such that
$|c_i| r^i \to 0$ as $i \to \pm \infty$.
\end{cor}
\begin{cor} \label{C:complete type 3a}
With notation as in Lemma~\ref{L:complete type 3}, 
for any interval $J \subseteq (0, 1]$ containing $r = r(\alpha)$,
let $R_J$ be the Fr\'echet completion of $F(x)$ under
$|\cdot|_{z,s}$ for $s \in J$.
Then for any open interval $J$ containing $r$ and any
nonzero $f \in R_J$, there exist $c_i \in F$, $i \in \ZZ$,
$\epsilon > 0$, and an open subinterval $J'$ of $J$ containing $r$,
such that for all $s \in J'$,
\[
|f - c_i (x-z)^i|_{z,s} \leq (1-\epsilon) |c_i (x-z)^i|_{z,s}.
\]
In particular, the union of $R_J$ over all open intervals $J$ containing $r$
is a field.
\end{cor}
\begin{proof}
By Corollary~\ref{C:complete type 3}, we can write
$f = \sum_{i \in \ZZ} c_i (x-z)^i$ for some $c_i \in F$
such that $|c_i| r^i \to 0$ as $i \to \pm \infty$.
As in Lemma~\ref{L:complete type 3},
there is a unique index $i$ maximizing $|c_i| r^i$.
Since the ratio $|f-c_i(x-z)^i|_{z,s} / |c_i (x-z)^i|_{z,s}$
is continuous in $s$, and it is
strictly less than 1 for $s=r$, it is also strictly
less than $1$ for $s$ in some open neighborhood of $r$.
This proves the claim.
\end{proof}
\begin{cor} \label{C:count components}
With notation as in Lemma~\ref{L:complete type 3}
and Corollary~\ref{C:complete type 3a},
for any open interval $J$ containing $r$ and any
finite \'etale extension $S$ of $R_J$, there exists
an open subinterval $J' \subseteq J$ containing $r$
such that the map
\[
\pi_0(S \otimes_{R_J} R_{J'}) \to \pi_0(S \otimes_{R_J} R_{[r,r]})
\]
(which is necessarily injective) is bijective.
\end{cor}
\begin{proof}
We first check that any factorization of a monic polynomial 
$P(T) \in R_J[T]$ into monic coprime factors
$Q_1, Q_2 \in R_{[r,r]}[T]$ lifts to a factorization in
$R_{J'}[T]$ for some $J'$.
Note that for $n$ a positive integer, $f_1,\dots,f_n \in
R_J[T_1,\dots,T_n]$, and $a \in (a_1,\dots,a_n) \in R_{[r,r]}^n$
such that $f_i(a) = 0$ for $i=1,\dots,n$ and $\det(\partial f_i/\partial T_j(a))
\neq 0$,
there exist an open interval $J' \subseteq J$ containing $r$ and
some $b \in R_{J'}^n$ such that $f_i(b) = 0$ for $i=1,\dots,n$.
(Namely, we construct $b$ using the standard
multivariate Newton-Raphson iteration.)
We then obtain the factorization assertion by following the proof
of the implication from (d$'$) to (e) in
\cite[Theorem~I.4.2]{milne}.

To deduce the claim from this, apply the primitive element theorem
over $\cup_{J'} R_{J'}$, which is a field
by Corollary~\ref{C:complete type 3a}. For some $J'$, this gives an isomorphism
$S \otimes_{R_J} R_{J'} = R_{J'}[T]/(P(T))$ for some monic separable
$P(T) \in R_{J'}[T]$. By applying the previous paragraph to
lift all factorizations of $P(T)$ from
$R_{[r,r]}(T)$ to $R_{J'}[T]$ for some $J'$, we get the desired result.
\end{proof}

\begin{lemma} \label{L:same corank}
Let $\alpha$ be a point of $\DD$ of type (ii) or (iii). Let $v = - \log \alpha$
be the corresponding real valuation on $F(x)$.
\begin{enumerate}
\item[(a)]
If $\alpha$ is of type (ii), then 
\[
\trdeg(\kappa_v / \kappa_{v_F}) = 1, \qquad \dim_\QQ ((\Gamma_v/\Gamma_{v_F}) \otimes \QQ) = 0.
\]
\item[(b)]
If $\alpha$ is of type (iii), then 
\[
\trdeg(\kappa_v / \kappa_{v_F}) = 0, \qquad \dim_\QQ ((\Gamma_v/\Gamma_{v_F}) \otimes \QQ) = 1.
\]
\end{enumerate}
\end{lemma}
\begin{proof}
Write $\alpha = \alpha_{z,r}$ with $z \in F^{\alg}$. 
If $\alpha$ is of type (ii), pick $\lambda \in F^{\alg}$ with
$|\lambda| = r$, otherwise put $\lambda = 1$.
There is no harm in replacing
$F,F(x)$ with $F(z,\lambda),F(z,\lambda,x)$, respectively,
as by Remark~\ref{R:extension},
this only replaces each residue field by a finite extension and
each value group by a group containing it with finite index. 

Suppose then that $z \in F$. If $\alpha$ is of type (ii), then 
\[
\kappa_v = \kappa_{v_F}(x/\lambda), \qquad \Gamma_v = \Gamma_{v_F}.
\]
If $\alpha$ is of type (iii), then by Lemma~\ref{L:complete type 3},
\[
\kappa_v = \kappa_{v_F}, \qquad \Gamma_v = \Gamma_{v_F} \oplus \ZZ \log r.
\]
This gives the desired assertions.
\end{proof}

\subsection{Geometry of valuations}

We next introduce some constructions concerning the
interplay between valuations and birational geometry of varieties.
\begin{hypothesis}
Throughout this subsection (except in the statement of
Theorem~\ref{T:de jong}), assume that $k$ is algebraically closed,
let $X$ be an irreducible variety over $k$, put $F = k(X)$, and let
$v: F^* \to \Gamma$ be a surjective Krull valuation.
\end{hypothesis}

This definition is standard; it was first recalled in this series
in \cite[\S 2.4]{kedlaya-part2}.
\begin{defn}
The \emph{center} of $v$ on $X$ is the subset of $x \in X$ for which
$\calO_{X,x} \subseteq \gotho_v$. If the center is nonempty, we say that
$v$ is \emph{centered on $X$}; in this case, the center is an irreducible closed
subset of $X$, which we will identify with the corresponding reduced closed
subscheme. If $X$ is proper, then $v$ is necessarily centered on $X$.
If $Z$ is a closed subvariety of $X$, we will say that $v$ is \emph{centered
on $Z$} if $v$ is centered on $X$ and the center is contained in $Z$.
\end{defn}

The following definitions are new.
\begin{defn}
By a \emph{local alteration} of $X$ around $v$,
we will mean a morphism
$f: X_1 \to X$ which is quasi-projective and generically finite,
with $X_1$ irreducible and some extension of $v$ to $k(X_1)$ centered
on $X_1$.
By contrast, an \emph{alteration} is supposed to be dominant, proper,
and generically finite. We say a (local) alteration is \emph{separable}
if it induces a separable extension of function fields; this is equivalent
to asking it to be generically finite \'etale. (The separability was
built into the definition we used earlier in the series, but is not built
into de Jong's original definition.)
\end{defn}

\begin{defn} \label{D:expose valuation}
Let $X$ be an irreducible variety, let $Z$ be a closed 
subvariety of $X$, and let $v$ be a minimal valuation on $k(X)$
centered on $Z$.
We call the pair $(X,Z)$ an \emph{exposure} of $v$ (or say that
$v$ is \emph{exposed} by the pair $(X,Z)$) if the following conditions hold.
\begin{enumerate}
\item[(a)]
The center $z$ of $v$ is a smooth
point of $X$ lying on exactly $r$ components of $Z$, for $r = \ratrank(v)$.
\item[(b)]
There exists a system of parameters $a_1, \dots, a_n$ for $X$ at $z$  
such that $v(a_1), \dots, v(a_r)$ are linearly independent over $\QQ$,
and the zero loci of $a_1, \dots, a_r$ at $z$ are the components of $Z$
passing through $z$.
\end{enumerate}
If $f: X_1 \to X$ is a local alteration, we say that
$f$ is an \emph{exposing alteration} of the pair $(X,Z)$ for $v$
if $(X_1, f^{-1}(Z))$ is an exposure
of some extension of $v$ to $k(X_1)$.
\end{defn}

The following is an analogue of \cite[Proposition~6.1.5]{kedlaya-part3}
for not necessarily monomial valuations.
\begin{lemma} \label{L:exposing alteration}
Assume that $k$ is algebraically closed. 
Then with notation as in
Definition~\ref{D:expose valuation}, there exists a
closed subscheme $Z'$ containing $Z$ and a separable
exposing alteration of $(X,Z')$ for $v$.
\end{lemma}
This is probably also true without enlarging $Z$, but that enlargement is
harmless for our purposes.
\begin{proof}
Choose $h_1,\dots,h_r \in k(X)^*$ such that $v(h_1),\dots,v(h_r)$
are positive and linearly independent over $\QQ$. 
We enlarge $Z$ so that each $h_i$, viewed as a rational map
$X \dashrightarrow \PP^1_k$, restricts to a regular map $X \setminus Z \to 
\mathbb{G}_{m,k}$. For the rest of the argument, we will repeat the operation
of replacing $X$ by a local alteration around $v$, $Z$ by its inverse image,
and $v$ by some extension to the new function field; we describe  this
for short as ``replacing''.

We first replace to force $h_1,\dots,h_r$ to become regular on $X$.
We then apply de Jong's alterations theorem 
(see Theorem~\ref{T:de jong} below)
to replace so that $X$ becomes smooth and 
$Z$ becomes a strict normal crossings divisor.
Suppose that
$D_1, \dots, D_s$ are the components 
of $Z$ passing through $z$. Let $a_1, \dots, a_n$ be a 
system of local parameters of $X$ at $z$, with
$a_i$ vanishing along $D_i$ for $i=1, \dots, s$. 
Since the Cartier divisor defined by each $h_i$ in a neighborhood of $z$
is a nonnegative
integer linear combination of $D_1,\dots,D_s$, and $v(h_1),\dots,v(h_r)$
generate a subgroup of $v(k(X)^*)$ of rank $r$, so do
$v(a_1), \dots, v(a_s)$. 
In particular, we must have $s \geq r$, and some $r$-element subset
of $v(a_1),\dots,v(a_s)$ must be linearly independent over $\QQ$; we
may assume without loss of generality that it is 
$v(a_1),\dots,v(a_r)$.

Construct the $s \times s$ matrix $A$ given by
Lemma~\ref{L:combin} below with $c_i = v(a_i)$,
and put 
\[
b_i = \prod_{j=1}^s a_j^{A_{ij}} \qquad (i=1,\dots,s);
\]
the nonnegativity of $A^{-1}$ means that each $a_i$ is a product
of nonnegative powers of the $b_j$.
We may then replace once more (by blowing up at $z$) 
to produce a new variety with local coordinates
$b_1, \dots, b_s, a_{s+1}, \dots, a_n$. 

At this point, we observe that 
$v(b_i) = 0$ for $i=r+1,\dots,s$.
Under the inclusion $\calO_{X,z} \subseteq \gotho_v$ of local rings,
the residue fields $\kappa_z$ and $\kappa_v$ are identified (since both
equal $k$), so the maximal ideal $\gothm_v$ of $\gotho_v$ must 
restrict to the maximal ideal of $\calO_{X,z}$.
Since $b_i \notin \gothm_v$, it follows that
$b_i$ does not vanish at $z$ for $i=r+1,\dots,s$.
Consequently, the components of the inverse image of 
$Z$ passing through $z$ are parametrized by 
$b_1, \dots, b_r$, so $v$ is exposed.
\end{proof}

\begin{lemma} \label{L:combin} 
Let $r \leq s$ be positive integers.
Let $c_1, \dots, c_s$ be positive real numbers such that
$c_1, \dots, c_r$ form a basis for the $\QQ$-span of 
$c_1, \dots, c_s$.
Then there exists a matrix $A \in \GL_s(\ZZ)$ such that $A^{-1}$ has 
nonnegative entries, and 
\[
\sum_{j=1}^s A_{ij} c_j > 0 \qquad (i=1, \dots, r),
\qquad
\sum_{j=1}^s A_{ij} c_j = 0 \qquad (i=r+1, \dots, s).
\]
\end{lemma}
\begin{proof}
The general case follows from repeated application of the case $r=s-1$,
which is \cite[Theorem~1]{zariski}.
\end{proof}

\begin{defn}
With notation as in Definition~\ref{D:expose valuation}, for $i=1,2,\dots$,
let $f_i: X_i \to X$ be an exposing alteration of $(X,Z)$ for $v$.
Let $S_i$ be the set of valuations on $k(X)$ admitting extensions
centered on $f_i^{-1}(Z)$.
We say that the $f_i$ form an \emph{exposing sequence} for $v$ if
the $S_i$ form a neighborhood basis of $v$ in the patch topology
on $S_{k(X)/k}$. (It is convenient in some cases to allow an exposing
sequence indexed by an arbitrary countable set.)
\end{defn}

\begin{lemma} \label{L:unramified open}
Let $f: Y \to X$ be an alteration with $Y$ irreducible.
Let $v$ be a valuation centered on $X$, which admits an extension
$w$ to $k(Y)$ which is unramified over $k(X)$. Then there exists
a commutative diagram
\[
\xymatrix{
Y' \ar[r] \ar^{f'}[d] & Y \ar^{f}[d] \\
X' \ar[r] & X
}
\]
in which $f'$ is an alteration
and the horizontal arrows are blowups, such that $w$
is centered on $Y'$ at a point at which $f'$ is \'etale.
(In particular, the property of admitting an unramified extension to $k(Y)$
is an open condition in $S_{k(X)/k}$.)
\end{lemma}
\begin{proof}
By the primitive element theorem, we have $\gotho_w = \gotho_v[t]/(P(t))$
for some monic polynomial $P(t) = t^d + \sum_{i=0}^{d-1} 
P_i t^i \in \gotho_v[t]$ whose discriminant
$\Delta$ is a unit in $\gotho_v$.
Choose $X'$ so that 
$v$ is centered on $X'$ at a point $z'$ for which $\Delta, \Delta^{-1}, P_0,
\dots, P_{d-1}
 \in \calO_{X',z'}$, then take $Y'$ to be the irreducible component of 
$X' \times_X Y$ which dominates $X'$. This has the desired
effect.
\end{proof}

Since we have used de Jong's alterations theorem once already, and 
will use it several times more, we formally recall its statement.
\begin{theorem}[de Jong] \label{T:de jong}
For $k$ an arbitrary field, 
let $X$ be a reduced separated scheme of finite type over $k$,
and let $Z$ be a reduced closed subscheme.
Then there exists a projective alteration $f: X_1 \to X$ such that
$X_1$ is regular and the reduced subscheme of
$f^{-1}(Z)$ is a strict normal crossings divisor.
Moreover, if $k$ is perfect, we may force $f$ to be separable.
\end{theorem}
\begin{proof}
See \cite[Theorem~4.1]{dejong}.
\end{proof}

It appears difficult to ensure that the alteration $f$ in Theorem~\ref{T:de jong}
is generically Galois. We are thus forced to consider the following construction.
\begin{defn}
Let $X, X_1$ be irreducible varieties over a field $k$, and let
$f: X_1 \to X$ be a separable alteration of generic degree $d$.
Let $F$ be the Galois closure of $k(X_1)$ over $k(X)$, and put $G = \Gal(F/k(X))$.
Let $\iota_1,\dots,\iota_d$ be the distinct embeddings of $k(X_1)$
into $F$; these carry a transitive action of $G$, with which we identify
$G$ with a subgroup of the symmetric group $S_d$.

Let $Y$ be the $d$-fold fibre product of $X_1$ over $X$, 
equipped with the natural action of $S_d$.
The embeddings $\iota_1,\dots,\iota_d$ define a geometric point $\eta$ of
the $d$-fold fibre product of $X_1$ over $X$, lying over the generic point of $X$.
Note that the stabilizer of $\eta$ in $S_d$ is exactly $G$.
Let $X_2$ be the Zariski closure of $\eta$ in $Y$, so that $X_2$ also carries an action
of $G$. Then $g: X_2 \to X$ is a separable alteration factoring through $X_1$,
which we call the \emph{Galois closure} of $f$.
\end{defn}

\subsection{Semistable reduction, local and global}
\label{subsec:semistable}

We now state the local and global semistable reduction theorems,
which constitute the principal results of this paper,
and explain how they may be derived from previous results in this series
plus the other results of this paper. We start with the
local version. (Beware that some definitions from
\cite{kedlaya-part2} are not correctly formulated 
in case $k$ is not perfect; see the appendix for relevant
corrections, which we also incorporate here.)

\begin{defn} \label{D:local semi}
Let $Z \hookrightarrow X$ be a closed immersion of $k$-varieties,
with $X$ irreducible and $X \setminus Z$ smooth.
Let $\calE$ be an $F$-isocrystal on $X \setminus Z$ overconvergent
along $Z$. 
Let $v$ be a valuation on $k(X)$ centered on $Z$.
We say that $\calE$ admits \emph{local semistable reduction} at $v$
if after replacing $k$ with $k^{q^{-n}}$ for some nonnegative integer $n$,
there exists an alteration
$f: X_1 \to X$ with $X_1$ irreducible
(which we require to be separable in case $k$ is perfect),
and an open subscheme $U$ of $X_1$ on which every extension $w$
of $v$ to $k(X_1)$ is centered, such that for
$Z_1 = f^{-1}(Z)$,
$(U, Z_1 \cap U)$ is a \emph{smooth pair}
(i.e., $U$ is smooth over $k$ and $Z_1 \cap U$ is a strict normal crossings
divisor)
and $f^* \calE$
extends to a convergent log-$F$-isocrystal with nilpotent residues on 
$(U, Z_1 \cap U)$.
\end{defn}

Before proceeding, it will be helpful to insert a remark omitted from
\cite{kedlaya-part2},
although it is partly implicit in the proof of 
\cite[Lemma~4.3.1]{kedlaya-part2}.
\begin{remark} \label{R:choose extension}
With $X,Z,v,\calE$ as in Definition~\ref{D:local semi}, 
suppose instead that $f: U \to X$ is a local alteration with $U$
irreducible, such that $(U, f^{-1}(Z))$ is a smooth pair,
$f^* \calE$ extends to a convergent log-isocrystal with
nilpotent residues on $(U, f^{-1}(Z))$, and at least one extension of $v$
to $k(U)$ is centered on $U$. Then $\calE$ again admits local semistable 
reduction, by the following argument.
From the form of the definition of local semistable reduction,
we may assume $k$ is perfect.
By Theorem~\ref{T:de jong}, we can pick a separable alteration
$g: X_2 \to X$ such that $k(X_2)$ contains the maximal separable subextension
of the normal closure of 
$k(U)$ over $k(X)$, and $(X_2, Z_2)$ is a smooth pair for $Z_2 = g^{-1}(Z)$.
Then for any extension $w$ of $v$ to $k(X_2)$ and any component of $Z_2$
passing through the center of $w$ on $X_2$, $g^* \calE$ is forced to
be unipotent along this component. 
(In case $f$ was not separable, the Frobenius structure on $\calE$
can be used to eliminate the inseparable
extension of function fields.) By \cite[Theorem~6.4.5]{kedlaya-part1},
$\calE$ is log-extendable to a neighborhood of the center of $w$,
and the Frobenius structure extends uniquely.
\end{remark}

\begin{theorem}[Local semistable reduction] \label{T:local semi}
With notation as in Definition~\ref{D:local semi},
$\calE$  always admits local semistable reduction at $v$.
\end{theorem}
\begin{proof}
By \cite[Theorem~4.3.4]{kedlaya-part2}, it suffices to prove this result
in case $k$ is algebraically closed and $v$ is minimal. If $\trdefect(v) = 0$,
then $v$ is a monomial valuation, so the claim holds by
\cite[Theorem~6.3.1]{kedlaya-part3}. This allows us to begin an induction
on $\trdefect(v)$, which we may conclude by showing that if $k$ is algebraically
closed and Theorem~\ref{T:local semi} holds for all $X, Z,
\calE, v$ with $v$ minimal of transcendence defect at most $n$,
then Theorem~\ref{T:local semi} also
holds when $v$ is minimal of transcendence defect $n+1$. This claim is the content of
Theorem~\ref{T:induct local semi}.
\end{proof}

We next deduce the global semistable reduction theorem.
The case of this theorem in which $X$ is proper
answers a conjecture of Shiho \cite[Conjecture~3.1.8]{shiho2}
which was introduced in this series as \cite[Conjecture~7.1.2]{kedlaya-part1}.

\begin{theorem}[Global semistable reduction] \label{T:global semi}
Let $Z \hookrightarrow X$ be a closed immersion of $k$-varieties,
with $X \setminus Z$ smooth.
Let $\calE$ be an $F$-isocrystal on $X \setminus Z$ overconvergent
along $Z$. Then 
after replacing $k$ with $k^{q^{-n}}$ for some nonnegative integer $n$,
there exists an alteration
$f: X_1 \to X$ (which we require to be separable in case $k$ is perfect)
such that for $Z_1 = f^{-1}(Z)$,
$(X_1, Z_1)$ is a smooth pair and $f^* \calE$
extends to a convergent log-$F$-isocrystal with nilpotent residues on
$(X_1, Z_1)$.
\end{theorem}
\begin{proof}
This follows from Theorem~\ref{T:local semi} via
\cite[Propositions~3.3.4 and~3.4.5]{kedlaya-part2}.
\end{proof}

\begin{remark}
Remember that the usual $p$-adic local monodromy theorem 
of Andr\'e \cite{andre}, Mebkhout \cite{mebkhout}, and the author
\cite{kedlaya-local}
applies to arbitrary
modules with Frobenius and connection structures over the Robba ring, not
just those that arise from overconvergent $F$-isocrystals on curves.
It should similarly be possible to formulate and prove Theorem~\ref{T:global semi} for 
``modules with Frobenius and connection concentrated on $Z$'',
with the same proof. 
\end{remark}

\subsection{Local monodromy representations}
\label{subsec:local mono}

Just as one may use the $p$-adic local monodromy theorem to construct
local monodromy representations for overconvergent $F$-isocrystals
on curves \cite[\S 4]{kedlaya-mono-over}, 
we may use local semistable reduction to construct
a local monodromy representation associated to an overconvergent $F$-isocrystal
and a valuation. We must assume local semistable reduction at that valuation,
which for external use is harmless in light of Theorem~\ref{T:local semi}.
However,
we cannot omit the hypothesis here, because we use the construction
at one valuation in the course of proving Theorem~\ref{T:local semi} for 
another valuation (as part of an induction on transcendence defect).

We will start with a bit of elementary group theory. This will be useful
in conjunction with the properties of inertia groups (Definition~\ref{D:inertia group}).
\begin{lemma} \label{L:wild structure2}
Let $G$ be a finite $p$-group. Let
$F$ be an algebraically closed field of characteristic zero. 
Let $\tau: G \to \GL(V)$ be a linear representation of $G$ on
a finite-dimensional $F$-vector space $V$.
\begin{enumerate}
\item[(a)]
If $\tau$ is irreducible, then $\dim(V)$ is a power of $p$.
\item[(b)]
If $\dim(V) = 1$, then there exists a nonnegative integer $h$ such that
$\tau^{\otimes p^h}$ is trivial.
\item[(c)]
If $\tau$ is nontrivial, then either $\tau$ or 
$\tau^\dual \otimes \tau$ contains a nontrivial subrepresentation of 
dimension $1$.
\end{enumerate}
\end{lemma}
\begin{proof}
\begin{enumerate}
\item[(a)]
The dimension of any irreducible representation of
any finite group on a finite-dimensional vector space over an algebraically
closed field divides the order of the group. Hence $\dim(V)$ divides a power
of $p$, and so must be a power of $p$ itself.
\item[(b)]
The action of $G$ via $\tau$ factors through an abelian quotient of $G$, 
necessarily of $p$-power order.
\item[(c)]
Since $\tau$ is nontrivial, it is nonzero and admits a nontrivial
irreducible subrepresentation $\psi$. By (a), the dimension of $\psi$ 
equals $p^m$ for some nonnegative integer $m$. If $m = 0$, then 
$\tau$ contains a nontrivial subrepresentation of dimension 1, and we
are done. Otherwise, we may split $\psi^\dual \otimes \psi$ into its
trace component and its trace-zero component. The latter has
dimension $p^{2m}-1 \not\equiv 0 \pmod{p}$, so its irreducible 
components cannot all have dimensions divisible by $p$. Hence 
the trace-zero component of $\psi^\dual \otimes \psi$ contains a 
subrepresentation of dimension 1; such a subrepresentation cannot be trivial,
or else $\psi$ would be forced to be reducible by Schur's lemma.
Hence $\psi^\dual \otimes \psi$ contains a nontrivial 
subrepresentation of dimension 1, as does 
$\tau^\dual \otimes \tau$ since $\tau$ is completely reducible 
(Maschke's theorem).
\end{enumerate}
\end{proof}

We recall some basic theory of Tannakian categories, as introduced in
\cite{saavedra}.
\begin{defn}
A \emph{Tannakian category} $\calC$ over $K$ is a associative, unital,
commutative $K$-linear tensor category with duals, admitting an exact faithful functor
to the category of finite-dimensional vector spaces over some field containing $K$.
Functors between Tannakian categories are required to respect the tensor and dual
structures.
Given an object $X \in \calC$, we will write $[X]$ for the
smallest subcategory of $\calC$
containing $X$ and closed under formation of isomorphic objects, 
direct sums, tensor products, duals, and subquotients.

Let $K'$ be a field extension of $K$.
A \emph{fibre functor} on $\calC$ with values in $K'$ is an exact functor
$\omega$ of Tannakian categories from $\calC$ to the category of finite-dimensional
$K'$-vector spaces. If $K' = K$, the fibre functor is said to be \emph{neutral}.

A \emph{$\otimes$-generator} of $\calC$ is an element $X$
such that the smallest subcategory of $\calC$ 
containing $X$ and closed under formation of isomorphic objects, 
direct sums, tensor products, and subquotients (but not duals) is equal to $\calC$.
A Tannakian category admitting an $\otimes$-generator is said to be \emph{algebraic}. For example, for any finite group $G$,
the category of representations of $G$ on
finite-dimensional $K$-vector spaces is algebraic.
\end{defn}

\begin{lemma} \label{L:Tannaka1}
Let $\calC$ be an algebraic Tannakian category over $K$.
\begin{enumerate}
 \item[(a)] There exists a fibre functor on $\calC$ with values in a finite
extension $K'$ of $K$.
\item[(b)] The automorphism group $G$ of any such fibre functor is finite.
\item[(c)] In case $K' = K$, $\calC$ is equivalent to the category of
representations of $G$ on finite dimensional $K$-vector spaces.
\end{enumerate}
\end{lemma}
\begin{proof}
See \cite[III, Scholie~3.3.1.1]{saavedra} for (a), 
\cite[III, 3.3.3(a)]{saavedra} for (b),
and \cite[III, Th\'eor\`eme~3.2.2]{saavedra} for (c).
\end{proof}

\begin{lemma} \label{L:Tannaka2}
Let $F: \calC' \to \calC$ be an exact functor of Tannakian categories.
Let $\omega$ be a fibre functor on $\calC$.
Let $G, G'$ be the automorphism groups of $\omega, \omega \circ F$, respectively.
\begin{enumerate}
 \item[(a)] The functor $F$ canonically induces a map $u:G \to G'$.
\item[(b)] Suppose that $F$ is fully faithful, and that for every $X \in \calC'$,
every subobject of $F(X)$ is the image of a subobject of $X$. Then $u$ is surjective.
\item[(c)] Suppose that every object of $\calC$ is a subquotient of an object
in the image of $F$. Then $u$ is injective.
\end{enumerate}
\end{lemma}
\begin{proof}
 See \cite[III, 3.3.3]{saavedra}.
\end{proof}

For the rest of the discussion of local monodromy representations,
it will be convenient to set some running hypotheses.
\begin{hypothesis}
For the remainder of \S~\ref{subsec:local mono},
assume that $k$ is algebraically closed.
Let $Z \hookrightarrow X$ be a closed immersion of $k$-varieties,
with $X$ irreducible affine and smooth.
Let $P$ be a smooth affine formal scheme over $\Spf \gotho_K$ with
$P_k \cong X$.
Let $v$ be a minimal valuation on $k(X)$ 
with center $z \in Z$.
\end{hypothesis}

\begin{defn} \label{D:inertia group}
Fix a geometric point $\overline{x}$ of $\Spec k(X)_v$.
We may then identify $\pi_1(\Spec k(X)_v, \overline{x})$ 
with a subgroup $I_v$ of $\pi_1(X, \overline{x})$, called the
\emph{inertia subgroup} associated to $v$
for the choice of the basepoint $\overline{x}$. 
In concrete terms, choosing $\overline{x}$ amounts to fixing
an algebraic closure $k(X)_v$, and $I_v$ may be identified with
the Galois group of that algebraic closure.

Suppose that $g: X_2 \to X$ is the normalization of the Galois closure of a separable alteration
$f: X_1 \to X$ with irreducible. Then $I_v$ surjects
onto the subgroup of $\Aut(X_2/X)$ fixing the center on $X_2$ of one of the extensions
of $v$ to $k(X_2)$. 
(Which center gets fixed depends on how we choose to embed $k(X_2)$
into our fixed algebraic closure of $k(X)_v$.)

Note that $I_v$ contains a subgroup $W_v$ 
(the \emph{wild inertia
subgroup}) which is a pro-$p$-group, such that the quotient $I_v/W_v$ is abelian
and pro-prime-to-$p$. See for instance \cite[Chapter~6]{ribenboim}.
\end{defn}

\begin{defn} \label{D:localizing subspace}
By a \emph{localizing subspace} of $P_K$ for $v$, we will mean
a subspace of $]z[_P$ defined by finitely many conditions, each of
the form $\epsilon \leq |g| < |h|$ for some $g,h \in \Gamma(P, \calO)$ 
and some $\epsilon
\in [0,1)$ such that the images
$\overline{g}, \overline{h} \in \Gamma(P_k, \calO)$ of $g,h$
satisfy $\overline{h} \neq 0$ and $\overline{g}/\overline{h} \in \gothm_v$.
Note that 
if we replace $g,h$ by some other
$g',h' \in \Gamma(P, \calO)$ with the same reductions, then the suprema
of $|g-g'|, |h-h'|$ over $P_K$ are strictly less than 1. For $\epsilon$  greater
than or equal to these suprema, 
the conditions $\epsilon \leq |g| < |h|$ and $\epsilon \leq |g'| < |h'|$
are equivalent.
\end{defn}

\begin{remark} \label{R:localizing1}
Suppose that $g,h \in \Gamma(P, \calO)$ are such that the images
$\overline{g}, \overline{h} \in \Gamma(P_k, \calO)$ of $g,h$
satisfy $\overline{h} \neq 0$ and $\overline{g}/\overline{h} \in \gotho_v$.
Then for any localizing subspace $A$,
we can always find another localizing subspace $A' \subseteq A$ on which 
we have $|g| \leq |h|$ everywhere. That is because we have assumed 
$v$ is minimal
and $k$ is algebraically closed, so we can find $\overline{c} \in k$ for which
$\overline{g}/\overline{h} - \overline{c} \in \gothm_v$. We can then achieve what
we want by lifting $\overline{c}$ to some $c \in \gotho_K$ and letting $A'$
be the subspace of $A$ on which $|g - ch| < |h|$.

On the other hand, if one were to construct local monodromy representations 
without the hypotheses that $v$ is minimal and $k$ is algebraically closed,
one would need to add conditions of the form $|g| \leq |h|$ to the definition
of a localizing subspace. We will not do this here.

Note also that when 
constructing localizing subspaces, we will sometimes find it convenient
to impose a condition of the form $\epsilon < |g| < |h|$,
rather than $\epsilon \leq |g| < |h|$.
This will be harmless for our purposes, because the space cut out by such a
condition contains the subspace cut out by the condition
$\epsilon' \leq |g| < h$ for any $\epsilon' \in (\epsilon,1)$.
\end{remark}

\begin{lemma} \label{L:lift flat}
Let $f: X_1 \to X$ be a separable alteration with $X_1$ irreducible.
Then there exist an open dense subscheme $T$ of $X$ and a
proper morphism $\tilde{f}: P_1 \to P$ of formal schemes
over $\Spf \gotho_K$,
such that $\tilde{f}$ is finite \'etale over $T$
and the induced map
$P_{1,k} \to X$ factors through a birational morphism $P_{1,k} \to X_1$.
(Beware that we do not guarantee that
$P_{1,k}$ is irreducible; birationality here means that the map
becomes an isomorphism over some open dense subscheme of $X_1$.)
\end{lemma}
\begin{proof}
Choose a primitive element $\alpha$ of $k(X_1)$ over $k(X)$
which is integral over $\Gamma(X, \calO)$. Let $\overline{Q}(t)$
be the minimal polynomial of $\alpha$ over $k(X)$, then 
lift $\overline{Q}$ to a monic polynomial $Q(t) \in \Gamma(P, \calO)[t]$
and take $P_0 = \Spf \Gamma(P, \calO)[t]/(Q(t))$.

By construction, there exists a rational map $P_{0,k} \dashrightarrow X_1$.
Let $X'_1$ be the Zariski closure of the graph of this rational map,
so that $X'_1$ projects onto both $P_{0,k}$ and $X_1$.
The map $X'_1 \to P_{0,k}$ is proper birational, 
so it can be written as the blowup along some ideal sheaf on $P_{0,k}$
\cite[Theorem~7.17]{hartshorne}. Since $P_{0,k}$ is affine,
we can lift global generators of this ideal sheaf to produce an ideal
sheaf on $P_0$, in which we blow up to obtain a proper morphism
$P_1 \to P_0$ such that $P_{1,k}$ admits a morphism to 
$X'_1$. This proves the claim. 
\end{proof}

\begin{lemma} \label{L:lift etale}
In Lemma~\ref{L:lift flat},
the map $\tilde{f}_K: P_{1,K} \to P_K$ is finite \'etale on some strict
neighborhood $V$ of $]T[_P$. Moreover, any automorphism $\tau$ of
$X_1$ over $X$ of finite order
lifts to an automorphism of $\tilde{f}_K^{-1}(V)$ over $V$,
for a suitable choice of $V$ (depending on $\tau$).
\end{lemma}
\begin{proof}
Set notation as in the proof of Lemma~\ref{L:lift flat}.
The map $P_0 \to P$ is finite, so we may argue as in
the proof of \cite[Theorem~2.6.3(1)]{tsuzuki-gysin}
to see that $P_{0,K} \to P_K$ is finite \'etale on some strict
neighborhood $V$ of $]T[_P$, and that a given automorphism $\tau$
of $X_1$ over $X$ of finite order lifts over $V$.
Let $T_0$ be the inverse image of $T$ in $P_{0,k}$,
and let $V_0$ be the inverse image of $V$ in $P_{0,K}$.
By Berthelot's strong fibration theorem
\cite[Th\'eor\`eme~1.3.5]{berthelot}, the map
$P_{1,K} \to P_{0,K}$ is an isomorphism over some strict neighborhood
$W$ of $]T_0[_{P_0}$. By shrinking $V$, we may ensure
that $V_0 \subseteq W$, and then the claim follows.
\end{proof}

\begin{lemma} \label{L:localizing}
In Lemma~\ref{L:lift flat},
let $S$ be the set of centers on $X_1$ of valuations on $k(X_1)$ extending $v$.
Then there exists a localizing subspace $A$ of $P_K$ for $v$
such that $\tilde{f}_K^{-1}(A)$ maps into $S$ under
$P_{1,K} \stackrel{\mathrm{sp}}{\to} P_{1,k} \to X_1$.
\end{lemma}
\begin{proof}
Apply Raynaud-Gruson flatification \cite[premi\`ere partie,
\S 5.2]{raynaud-gruson} to find a blowup $X' \to X$
under which the proper transform $X'_1 \to X'$ 
of $X_1 \to X$ is finite flat.
By blowing up $P_1$ as needed, we can force ourselves into the case
where $P_{1,k} \to X$ factors through $X'_1$.

Let $z'$ denote the center of $v$ on $X'$.
Let $U$ be an open dense affine subscheme of $X'$ containing $z'$.
Choose finitely many elements $g_i, h_i \in \Gamma(P, \calO)$
so that the ratios $\overline{g}_i/\overline{h}_i \in k(X)$
generate $\Gamma(U, \calO)$ as an algebra over $\Gamma(X, \calO)$.
By Remark~\ref{R:localizing1}, we can find a localizing subspace
$A_1$ of $P_K$ for $v$ on which $|g_i| \leq |h_i|$ for each $i$.

Next, choose finitely many elements $g_j, h_j \in \Gamma(P, \calO)$
so that the ratios $\overline{g}_j/\overline{h}_j \in k(X)$
belong to $\Gamma(U, \calO)$ and cut out a nonempty closed subscheme of $U$
supported at $z'$.
Let $A$ be the subspace of $A_1$ on which $|g_j| < |h_j|$ for each $j$;
then $A$ is a localizing subspace of $P_K$ for $v$.

By construction, any point of $\tilde{f}_K^{-1}(A)$ must map to a point of $P_{1,k}$
projecting onto $z'$ in $X'$. The image in $X'_1$ must then be 
the center of an extension of $v$, and similarly after projecting onto $X_1$.
This yields the claim.
\end{proof}

\begin{lemma} \label{L:unipotent point}
Suppose that $k$ is algebraically closed.
Let $(X,Z)$ be a smooth pair of $k$-varieties with $X$ affine irreducible.
Let $P$ be a smooth affine formal scheme over $\Spf \gotho_K$ with
$P_k \cong X$, and let $Q$ be a relative normal crossings divisor on $P$
with $Q_k \cong Z$.
Let $\calE$ be a convergent isocrystal on $(X,Z)$ with nilpotent residues, 
realized
as a log-$\nabla$-module on $(P_K, Q_K)$ with nilpotent residues.
Then for any closed point $x \in X$, the restriction of $\calE$
to $]x[_P$ is unipotent (i.e., it is a successive extension of constant
$\nabla$-modules).
\end{lemma}
\begin{proof}
Apply \cite[Proposition~3.6.9]{kedlaya-part1}.
\end{proof}

\begin{remark} \label{R:tube}
We warn the reader of a slight misuse of notation in 
Definition~\ref{D:mono rep} below. For $P$ a formal scheme,
the notation $]S[_P$ normally indicates the tube of a subspace $S$
of the special fibre $P_k$. We will also use it in settings where
$S$ is not a subspace of $P_k$ itself, but of another scheme $X$ to
which $P_k$ maps in a specified manner. We then mean to take the inverse
image of $S$ in $P_k$ before forming the tube.
\end{remark}

\begin{defn} \label{D:mono rep}
Let $\calE$ be an $F$-isocrystal on $X \setminus Z$ overconvergent
along $Z$ admitting local semistable reduction at $v$. 
Define $f, X_1, Z_1, U$ as in Definition~\ref{D:local semi}.
Choose an extension $v_1$ of $v$ to $k(X_1)$, let $z_1$ be the center
of $v_1$ on $X_1$, and let $U_1$ be an open
dense affine subscheme of $U$ containing $z_1$.
After possibly shrinking $U_1$, we can construct
a smooth affine 
formal scheme $Q_1$ over $\Spf \gotho_K$ with $Q_{1,k} \cong U_1$,
and a relative simple normal crossings divisor on $Q_1$ 
lifting $U_1 \cap Z_1$.

Let $g: X_2 \to X$ be the normalization of the Galois closure of some
separable alteration factoring through $f$;
we may factor $g = f \circ f_1$ for $f_1: X_2 \to X_1$ another separable
alteration. Put $G = \Aut(X_2/X)$.
Choose an extension $v_2$ of $v_1$ to $k(X_2)$, and let $z_2$ be the
center of $v_2$ on $X_2$; then $I_v$ surjects onto the stabilizer $H$ of
$z_2$ in $G$.

By Lemma~\ref{L:lift flat}, there exist an open dense affine subscheme
$T$ of $X$ and a proper morphism $\tilde{h}: P_2 \to P$ of formal schemes,
such that $h: P_{2,k} \to X$ factors through $X_2$,
and over $T$, $\tilde{h}$ is finite
\'etale and reduces to $g$.
We may then shrink $T$ to ensure that $T_1 = f^{-1}(T) \subseteq U_1$.
Put $T_2 = f_1^{-1}(T_1), U_2 = f_1^{-1}(U_1)$.

By Lemma~\ref{L:lift etale}, there exists a strict neighborhood $V$
of $]T[_P$ (see Remark~\ref{R:tube})
in $P_K$ over which $\tilde{h}_K$ is finite \'etale Galois
with group $G$.
Note that as we shrink $V$, $V_2 = \tilde{h}_K^{-1}(V)$ runs through a cofinal
set of strict neighborhoods of $]T_2[_{P_2}$ in $P_{2,K}$.
After shrinking $V$ suitably, 
we may realize $\calE|_T$ as a $\nabla$-module $\calE_V$ on
$V$, such that $\tilde{h}_K^* \calE_V$ realizes $(g^* \calE)|_{T_2}$
on $V_2$. 

On the other hand, we may realize $(f^* \calE)|_{T_1}$ as a $\nabla$-module 
$\calE_{W_1}$ on some strict neighborhood $W_1$ of $]T_1[_{Q_1}$ in $Q_{1,K}$. 
By \cite[Theorem~6.4.1]{kedlaya-part1}, 
after shrinking $W_1$ suitably, the $\nabla$-module
$(f^* \calE)|_{T_1}$ extends to a log-$\nabla$-module $\calF$
over all of $Q_{1,K}$.
By Lemma~\ref{L:unipotent point}, $\calF$
is unipotent on $]z_1[_{Q_1}$, hence also on $W_1 \cap ]z_1[_{Q_1}$.

Let $P'_2$ be the open formal subscheme of $P_2$ supported on
the inverse image of $U_1$ under the map $P_{2,k} \to X_2 \to X_1$.
Let $\Gamma$ denote the graph of the morphism $P'_{2,k} \to U_1$;
this coincides with the Zariski closure of 
$T_2$ in $P'_{2,k} \times U_1$ under the product
of the embedding $T_2 \to P'_{2,k}$ and the composition $T_2 \to T_1 \to U_1$.
In particular, it is isomorphic to $P'_{2,k}$.

We can construct two realizations of $(g^* \calE)|_{T_2}$
as $\nabla$-modules on a strict neighborhood $W_2$
of $]T_2[_{P'_2 \times Q_1}$ in $]\Gamma[_{P'_2 \times Q_1}$,
one by pulling back $\tilde{h}_K^* \calE_V$ along the first projection,
the other by pulling back $\calE_{W_1}$ along the second projection.
By the functoriality of rigid cohomology, for $W_2$ suitably small,
there is a distinguished isomorphism between these two realizations.
The second realization is unipotent on 
$W_2 \cap (P'_{2,K} \times ]z_1[_{Q_1})$, 
hence also on $W_2 \cap ]z_2[_{P'_2 \times Q_1}$. This implies in turn
that the the first realization is unipotent on $W_2 \cap ]z_2[_{P'_2 \times Q_1}$,
so $\tilde{h}_K^* \calE_V$ is unipotent on $V_2 \cap ]z_2[_{P_2}$.

By applying Lemma~\ref{L:localizing}, we obtain a localizing subspace $A$
of $P_K$ for $v$ such that $\tilde{h}_K^{-1}(A)$ is finite \'etale
Galois over $A$, and the restriction $\calE_A$ of $\calE$ to $A$
becomes unipotent upon pullback to 
$B = \tilde{h}_K^{-1}(A) \cap ]z_2[_{P_2}$.
The semisimplification $\calE_A^{\semis}$ then becomes constant 
upon pullback to $B$; let $\calE_B^{\semis}$ denote this pullback.

Note that $\tilde{h}_K^{-1}(A)$ is a disjoint union of copies of $B$
corresponding to points of $X_2$ in the $G$-orbit of $z_2$,
so $B$ is finite \'etale Galois over $A$ with group $H$.
We thus obtain an action of $H$ on $H^0(\calE_B^{\semis})$,
and by composition a homomorphism $\tau: I_v \to \GL(H^0(\calE_B^{\semis}))$.
Note that the construction does not depend on the choice of $g$:
if $g': X'_2 \to X$ is the normalization of the Galois closure of
another separable alteration factoring
through $g$, the action of $\Aut(X'_2/X)$ on 
local horizontal sections factors through $G$, so we end up with 
the same representation $\tau$ up to isomorphism.
We refer to $\tau$ as the \emph{semisimplified local monodromy
representation} associated to $\calE$ at $v$. 
\end{defn}

We will need the following Tannakian interpretation of the
construction of local monodromy representations.
\begin{lemma} \label{L:Tannakian}
With notation as in Definition~\ref{D:mono rep},
the Tannakian
category $[\calE_A^{\semis}]$
of $\nabla$-modules on $A$ generated by $\calE_A^{\semis}$ 
(not necessarily admitting Frobenius structures)
is equivalent
to the Tannakian category $[\tau]$ of representations of $I_v$ on 
finite-dimensional $K$-vector spaces
generated by $\tau$. 
Moreover, the latter category contains \emph{all} representations
of the group $\tau(I_v)$ on finite-dimensional $K$-vector spaces,
and $\tau$ is a $\otimes$-generator of this category;
in particular, both Tannakian categories are algebraic
with automorphism group $\tau(I_v)$.
\end{lemma}
\begin{proof}
We check that $\tau$ is a $\otimes$-generator of $[\tau]$
by an elementary argument (it may also be deduced using Tannaka duality).
Let $\chi: \tau(I_v) \to K$ be the character of $\tau$,
so that $\chi(1) = \dim(\tau)$ and $\chi(g) \neq \dim(\tau)$ for $g \in
\tau(I_v) \setminus \{1\}$. Since the $\chi(g)$ are algebraic integers,
we can find $P \in \ZZ[t]$ for which $P(\dim(\tau)) > 0$
but $P(\chi(g)) = 0$ for $g \in \tau(I_v) \setminus \{1\}$.
Thus the virtual character $P(\chi)$ 
is a positive multiple of the character of the regular representation,
in which each irreducible representation of $\tau(I_v)$ must appear
with positive multiplicity.
Write $P = Q - R$ with $Q,R \in \ZZ[t]$ having 
nonnegative coefficients; then $Q(\chi)+ R(\chi)$
is a true character in which each irreducible representation occurs
with positive multiplicity. This proves the claim.

On one hand, $[\tau]$ now equals the category
of representations of $\tau(I_v)$ on finite-dimensional
$K$-vector spaces. On the other hand, $[\tau]$ is algebraic.
Let $\omega$ be the forgetful functor from $[\tau]$ to finite-dimensional
$K$-vector spaces, and put $G = \Aut(\omega)$.
By Lemma~\ref{L:Tannaka1}(c),
$[\tau]$ is also equivalent to the category of representations of $G$
on finite-dimensional $K$-vector spaces. We obtain automorphisms
of $\omega$ from elements of $\tau(I_v)$, giving an inclusion
$\tau(I_v) \to G$. This inclusion must be an isomorphism, otherwise
we could construct nonisomorphic representations of $G$ whose restrictions
to $\tau(I_v)$ are isomorphic (e.g., the induction to $G$
of the trivial representation of $\tau(I_v)$, and a direct sum
of trivial representations of the same total dimension). Hence
$G = \tau(I_v)$.

Let $F: [\calE_A^{\semis}] \to [\tau]$
be the faithful functor taking a $\nabla$-module $\calF$ to
$H^0(\calF_B)$, where $\calF_B$ denotes the pullback to $B$.
We check that $F$ is fully faithful. Given $\calF_1, \calF_2 \in 
[\calE_A^{\semis}]$,
any $I_v$-equivariant morphism $H^0(\calF_{1,B}) \to H^0(\calF_{2,B})$
corresponds to an $I_v$-fixed horizontal section of
$\calF_{1,B}^\dual \otimes \calF_{2,B}$, which then descends to
a horizontal section of $\calF_1^\dual \otimes \calF_2$.
This in turn yields a morphism $\calF_1 \to \calF_2$, proving the claim.

We check that $F$ is essentially surjective.
From above, every irreducible $\rho \in [\tau]$ occurs as a subobject of
$\tau^{\otimes n}$ for some nonnegative integer $n$. Since $[\tau]$ 
is semisimple (Maschke's theorem),
we can find an endomorphism of $\tau^{\otimes n}$
which projects onto one copy of $\rho$. This endomorphism
corresponds to an $I_v$-fixed horizontal section of 
$(\calE^{\semis}_B)^{\otimes (-n)} \otimes (\calE_B^{\semis})^{\otimes n}$, 
thus to a horizontal section of 
$(\calE^{\semis}_A)^{\otimes (-n)} \otimes (\calE^{\semis}_A)^{\otimes n}$
and to an endomorphism of $(\calE^{\semis}_A)^{\otimes n}$.
The image of this last endomorphism is an element $\calF \in
[\calE_A^{\semis}]$ for which $F(\calF) \cong \rho$.
\end{proof}

\begin{prop} \label{P:refined log-extend}
Suppose that $k$ is algebraically closed.
Let $Z \hookrightarrow X$ be a closed immersion of $k$-varieties,
with $X$ irreducible affine and $X \setminus Z$ smooth.
Let $v$ be a minimal valuation on $k(X)$
centered on $Z$.
Let $f_i: X_i \to X$ be an exposing sequence for $v$.
Let $\calE$ be an $F$-isocrystal on $X \setminus Z$ overconvergent along 
$Z$ admitting local semistable reduction at $v$. If the semisimplified
local monodromy representation associated to $\calE$
is trivial, then there exists an index $i$ such that $f_i^* \calE$
is log-extendable.
\end{prop}
\begin{proof}
Set notation as in Definition~\ref{D:mono rep}, and suppose that
$\tau$ is trivial. By Lemma~\ref{L:Tannakian},
$\calE_A$ is unipotent.

Recall that $A$ was defined using finitely many conditions of the form
$\epsilon \leq |g_j| < |h_j|$ with $g_j,h_j \in \Gamma(P, \calO)$,
$\overline{h}_j \neq 0$, and
$v(\overline{g}_j/\overline{h}_j) > 0$. In particular, 
we can find a blowup $X'$ of $X$ such that the
$\overline{g}_j/\overline{h}_j$ belong to the local ring
of $X'$ at the center $z'$ of $v$.
Let $S$ be the subset of $S_{k(X)/k}$ consisting of valuations 
whose center on $X'$ is equal to $z'$; this is an open set for the patch
topology.

We can thus choose an index $i$ so that the divisorial valuations
of $k(X_i)$ corresponding to the components of $f_i^{-1}(Z)$
all belong to $S$.
By one direction of
\cite[Theorem~6.4.5]{kedlaya-part1}, $f_i^* \calE$ 
has unipotent local monodromy along each of these valuations.
By the other direction of
\cite[Theorem~6.4.5]{kedlaya-part1}, $f_i^* \calE$ is log-extendable,
as desired.
\end{proof}

\begin{remark} \label{R:isotypical}
It would simplify some arguments in this paper to construct a
local monodromy representation without semisimplification,
e.g., by following the procedure in the one-dimensional case given in
\cite[Theorem~4.45]{kedlaya-mono-over}.
Unfortunately, this is difficult in the higher-dimensional setting;
the main obstruction is the lack of a good comparison
between the de Rham cohomologies of the spaces $A$ and $B$ in 
Definition~\ref{D:mono rep}. (In the one-dimensional case, they are both 
open annuli.) We will not attempt to tackle this issue here.
\end{remark}

\section{Analysis of $p$-adic differential modules}
\label{sec:diffmod}

In this section, we recall some important facts from the theory
of ordinary $p$-adic differential equations. Our reference
for these is the book \cite{kedlaya-course}. We have not distinguished
here between results original to \cite{kedlaya-course} and results
of prior origin;
that is done amply in \cite{kedlaya-course}.

\setcounter{theorem}{0}
\begin{hypothesis}
Throughout this section,
let $F$ be a complete nonarchimedean field of characteristic $0$
and residue characteristic $p$. (It is not necessary for $F$ to
be discretely valued.)
Let $I$ be a subinterval of $[0, +\infty)$,
and let $\calE$ be a $\nabla$-module of rank $n>0$
on the disc/annulus $A_{F}(I)$
with coordinate $x$.
\end{hypothesis}

\begin{defn}
We say $\calE$ is \emph{constant} if it is spanned by horizontal sections,
and \emph{unipotent} if it is a successive extension
of constant objects.
\end{defn}

\subsection{Subsidiary radii of convergence}

The following definition is from \cite[Definition~9.4.4]{kedlaya-course}.
\begin{defn} \label{D:generic radius}
For $\rho \in I$ nonzero, 
let $F_\rho$ be the completion of $\Gamma(A_{F}(I), \calO)$
for the $\rho$-Gauss norm, and let $\calE_\rho$ be the base extension of
$\calE$ to $F_\rho$. Define the \emph{generic radius of convergence}
of $\calE_\rho$, denoted $R(\calE_\rho)$,
as $p^{-1/(p-1)}$ divided by the spectral norm of
$\frac{d}{dx}$ on $\calE_\rho$; this quantity is always positive
and never greater than $\rho$. Note also that
\[
R(\calE^\dual_\rho) = R(\calE_\rho)
\]
and that for $\calE'$ another nonzero $\nabla$-module on $A_F(I)$,
\[
R((\calE \otimes \calE')_\rho) \geq \min\{R(\calE_\rho), R(\calE'_\rho)\}
\]
with equality if the minimum is achieved only once 
\cite[Lemma~9.4.6]{kedlaya-course}.
\end{defn}

\begin{remark} \label{R:transfer}
Suppose that $0 \in I$. If $\rho' < R(\calE_\rho)$ for some $\rho, 
\rho' \in I$, 
then Dwork's transfer
theorem \cite[Theorem~9.6.1]{kedlaya-course} implies that 
the restriction of $\calE$ to $A_{F}[0, \rho']$ is constant, so
$R(\calE_{\rho'}) = \rho'$.
\end{remark}

The following definition is from \cite[Definition~9.8.1]{kedlaya-course}.
\begin{defn} \label{D:subsidiary}
For $\rho \in I$, let $\calE_{\rho,1}, \dots, \calE_{\rho,m}$ be the
Jordan-H\"older constituents of $\calE_\rho$. Define the \emph{subsidiary 
generic
radii of convergence}, or \emph{subsidiary radii}, of $\calE_\rho$
as the multiset consisting of, for $i=1,\dots,m$, the 
generic radius of convergence of $\calE_{\rho,i}$ with multiplicity
$\dim_{F_\rho}(\calE_{\rho,i})$. 
For $s \in -\log I$,
define $f_1(\calE,s) \geq \cdots \geq f_n(\calE,s)$ to be the real
numbers such that the subsidiary radii of $\calE_{e^{-s}}$ are
equal to $e^{-f_1(\calE,s)}, \dots, e^{-f_n(\calE,s)}$.
\end{defn}

\begin{theorem} \label{T:variation}
The functions $f_i(\calE,s)$ for $i=1,\dots,n$ have the following 
properties.
\begin{enumerate}
\item[(a)]
The function $f_i(\calE,s)$ is continuous and piecewise affine-linear,
with slopes in $\frac{1}{n!} \ZZ$.
\item[(b)]
If $0 \in I$ and $f_i(\calE,s_0) > s_0$ for some $s_0 \in -\log I$, 
then on each side of $s_0$ contained in $-\log I$,
the slope of $f_1(\calE,s) + \cdots + f_i(\calE,s)$
in a one-sided neighborhood of $s_0$ is nonpositive.
(By Remark~\ref{R:transfer}, we have $f_i(\calE,s) = s$ for $s$
sufficiently large.)
\item[(c)]
The function $f_1(\calE,s) + \cdots + f_i(\calE,s)$ is convex.
\end{enumerate}
\end{theorem}
\begin{proof}
See \cite[Theorem~11.3.2]{kedlaya-course}.
\end{proof}

\begin{cor}
Suppose $I = [\alpha, \beta]$ with $\alpha < \beta$. 
Then there exists $\gamma \in (\alpha,\beta]$
such that for $i=1,\dots,n$,
$f_i(\calE,s)$ is affine-linear on $[\alpha,\gamma]$.
\end{cor}
This corollary combines naturally with the following theorem.

\begin{theorem} \label{T:decomposition}
Suppose that the interval $I$ is open, and that the following
conditions are satisfied for some $i \in \{1,\dots,n-1\}$.
\begin{enumerate}
\item[(a)]
The function $f_1(\calE,s) + \cdots + f_i(\calE,s)$ is affine-linear
on $-\log I$.
\item[(b)]
We have $f_i(\calE,s) > f_{i+1}(\calE,s)$ for all $s \in -\log I$.
\end{enumerate}
Then there exists a unique direct sum decomposition $\calE_1 \oplus \calE_2$
of $\nabla$-modules on $A_F(I)$ such that $\rank(\calE_1) = i$,
$f_j(\calE,s) = f_j(\calE_1,s)$ for $j=1,\dots,i$, and
$f_j(\calE,s) = f_{j-i}(\calE_2,s)$ for $j=i+1,\dots,n$.
\end{theorem}
\begin{proof}
See \cite[Theorem~12.4.2]{kedlaya-course}.
\end{proof}

The $f_i(\calE,s)$ also satisfy the following subharmonicity property.
\begin{theorem} \label{T:subharmonicity}
Suppose that $I = [0, \beta]$ with $\beta > 1$. Suppose $z_1,
\dots, z_m \in \gotho_F$ have distinct images in $\kappa_F$.
For $j=1,\dots,m$, let $T_j: A_F(I) \to A_F(I)$ denote the substitution 
$x \mapsto x + z_j$. For $i=1,\dots,n$,
let $s'_{\infty,i}$ be the left slope
of $f_1(\calE, s) + \cdots + f_i(\calE,s)$ at $s=0$. For
$i=1,\dots,n$ and 
$j=1,\dots,m$, let $s'_{j,i}$ be the right slope of
$f_1(T_j^* \calE, s) + \cdots + f_i(T_j^* \calE,s)$ at $s=0$. Then
\[
s'_{\infty,i} \leq \sum_{j=1}^m s'_{j,i} \qquad (i=1,\dots,n).
\]
\end{theorem}
\begin{proof}
See \cite[Theorem~11.3.2(c,d)]{kedlaya-course}.
\end{proof}

We also use the following compatibility with tame base change.
\begin{lemma} \label{L:tame base change}
Let $m$ be a positive integer coprime to $p$,
and let $\calF$ be the pullback of $\calE$ along the substitution
$x \mapsto x^m$. Then for all $s \in -\log I$, $f_i(\calF, s/m)-s/m
= f_i(\calE,s)-s$. In particular, $f_i(\calF,s/m) = s/m$ if and only if
$f_i(\calE,s) = s$.
\end{lemma}
\begin{proof}
See \cite[Proposition~9.7.6]{kedlaya-course}.
\end{proof}

\subsection{The Robba condition}

\begin{notation}
In this subsection only, for $m$ a positive integer,
let $f_m: A_F(I^{1/m}) \to A_F(I)$ denote the
$F$-linear map for which $f_m^*(x) = x^m$.
\end{notation}

\begin{defn}
We say that $\calE$ satisfies the \emph{Robba condition} if 
$R(\calE_\rho) = \rho$ for all nonzero $\rho \in I$. For instance,
this holds if $\calE$ is unipotent.
\end{defn}

As noted in \cite{kedlaya-course},
the following results follow from the Christol-Mebkhout theory
of $p$-adic exponents. However, see Remark~\ref{R:no exponents} below.

\begin{theorem} \label{T:Robba condition Frob}
Suppose that $I$ is open and that $1 \in I$.
Let $q$ be a power of $p$.
Suppose that $\calE$ satisfies the Robba condition
and admits an isomorphism $(\phi_K \circ f_q)^* 
\calE \cong \calE$ over $A_F(I^{1/q})$ for some isometric endomorphism
$\phi_K: K \to K$.
Then there exists a positive integer $m$ coprime to $p$ such that
$f_m^* \calE$ is unipotent.
\end{theorem}
\begin{proof}
See \cite[Corollary~13.6.2]{kedlaya-course}.
\end{proof}

\begin{theorem} \label{T:extend Robba cond}
Suppose that $I$ is open, and that $I'$ is an open subinterval
of $I$. If $\calE$ satisfies the Robba condition, and
the restriction of $\calE$ to $A_F(I')$ is constant (resp.\ unipotent),
then so is $\calE$ itself.
\end{theorem}
\begin{proof}
See \cite[Corollary~13.6.4]{kedlaya-course}.
\end{proof}

\begin{prop} \label{P:seed unipotent}
Suppose that $I$ is open and that $\calE$ satisfies the Robba condition.
Suppose that there exist a closed subinterval  $J$ of $I$ with
nonempty interior and a finite \'etale cover $g: X \to A_F(J)$ such that
$g^* \calE$ is unipotent. Then 
there exists a positive integer $m$ coprime to $p$ such that
$f_m^* \calE$ is unipotent.
\end{prop}
\begin{proof}
(Compare \cite[Proposition~19.4.5]{kedlaya-course}, where 
a different proof is given.)
Since unipotence can be checked after enlarging $F$, we may assume that
$F$ contains a $p$-th root of unity $\zeta_p$. We may also assume that
there exists $\rho \in |F^*|$ in the interior of $J$;
by rescaling, we may then reduce to
the case $\rho = 1$.

Note that $\sigma = \sigma_K \circ f_p: A_F(J^{1/p}) \to A_F(J)$ 
extends uniquely to a map
$\tilde{\sigma}: X \times_{A_F(J)} A_F(J^{1/p}) \to X$,
whereas the inclusion $\iota: A_F(J^{1/p}) \to A_F(J)$ induces a second map
$\tilde{\iota}: X \times_{A_F(J)} A_F(J^{1/p}) \to X$.
Put $\calF = g_* \calO_X$.
Since the differential modules $\tilde{\sigma}^* \calO_X$ and 
$\tilde{\iota}^* \calO_X$ are both
trivial, they are isomorphic to each other;
we thus obtain an isomorphism $F: \sigma^* \calF \cong \iota^* \calF$.

Since $g^* \calE$ is unipotent, there exists a nonzero morphism
$g^* \calE \to \calO_X$ of $\nabla$-modules on $X$. By adjunction, 
there also exists a nonzero morphism $\calE \to \calF$ over $A_F(J)$.
Let $\calG$ be the minimal $\sigma$-stable 
$\nabla$-submodule of $\calF$ containing the
image of this morphism; then $\calG$ satisfies the Robba condition.
(More precisely, put $\calG_N = \oplus_{i=0}^N \sigma^{i*} \calE$
as a $\nabla$-module over $A_F(J^{1/p^N})$, and take $\calG = \calG_N$
for the minimal $N$ such that $\calG_N$ and $\calG_{N+1}$ have the same
restriction to $A_F(J^{1/p^{N+1}})$.)
By Theorem~\ref{T:Robba condition Frob}, there exist a positive integer $m$
coprime to $p$ and some $\epsilon > 0$ such that the restriction 
of $f_m^* \calG$ to $A_F(1-\epsilon,1+\epsilon)$ is unipotent.

Consequently, the restriction of $f_m^* \calE$ to $A_F(1-\epsilon,1+\epsilon)$ is 
either unipotent or reducible.
By induction on $\rank(\calE)$, we may deduce that for some $m$ coprime
to $p$ and some nonempty open interval $J'$ of $I$, the restriction of
$f_m^* \calE$ to $A_F(J')$ is unipotent.
By Theorem~\ref{T:extend Robba cond}, $f_m^* \calE$ itself is unipotent.
\end{proof}

\begin{remark} \label{R:no exponents}
The theory of indices of $p$-adic differential operators, developed
by Christol-Mebkhout, separates naturally into two aspects: the theory
of decompositions of solvable $\nabla$-modules, 
and the $p$-adic Fuchsian theory (i.e.,
the theory of $p$-adic exponents). Of these, the decomposition theory
is the more accessible; the version of it we developed 
in \cite{kedlaya-course} underlies most of the arguments in this paper.

By contrast, the theory of $p$-adic exponents seems to lie much deeper.
The original development by Christol-Mebkhout is quite difficult;
subsequent treatments (including one by Dwork, one that we gave in
\cite[Chapter~13]{kedlaya-course}, and one in an upcoming book of Christol)
achieve some simplifications, but still rely on some 
delicate $p$-adic analytic estimates. 

Consequently, it may be desirable
to avoid use of $p$-adic exponents for purely aesthetic reasons.
This can indeed be done, by virtue of 
the following observations.
\begin{itemize}
\item
Theorem~\ref{T:Robba condition Frob} is only used in the proof
of Proposition~\ref{P:seed unipotent}.
\item
Theorem~\ref{T:extend Robba cond} can be obtained using a transfer
theorem of Christol in place of the theory of $p$-adic exponents;
see \cite[Remark~13.7.3]{kedlaya-course}.
\item
Proposition~\ref{P:seed unipotent} can be obtained 
without using the theory of $p$-adic exponents;
see \cite[Proposition~19.4.5]{kedlaya-course}.
That argument uses the relationship
between generic radius of convergence and wild ramification of local
Galois representations in positive characteristic,
which can also be developed without $p$-adic exponents; 
see \cite[Theorem~5.23]{kedlaya-mono-over}.
\end{itemize}
\end{remark}

\subsection{Frobenius antecedents and descendants}

\begin{defn}
Let $\phi: A_F(I) \to A_F(I^p)$ be the $F$-linear substitution $x \mapsto x^p$.
A \emph{Frobenius antecedent} of $\calE$ is a $\nabla$-module
$\calF$ on $A_F(I^p)$ equipped with an isomorphism $\calE \cong \phi^* \calF$,
such that $R(\calF_\rho) >p^{-p/(p-1)}\rho$ for all $\rho \in I^p \setminus
\{0\}$.
\end{defn}

\begin{theorem} \label{T:antecedent}
Assume $I \neq [0,0]$.
\begin{enumerate}
\item[(a)]
A Frobenius antecedent of $\calE$ exists if and only if 
$R(\calE_\rho) > p^{-1/(p-1)} \rho$ for all $\rho \in I \setminus \{0\}$.
\item[(b)]
A Frobenius antecedent is unique if it exists.
\item[(c)]
If $\calF$ is the Frobenius antecedent of $\calE$, then
\[
f_i(\calF, ps) = p f_i(\calE, s) \qquad (i=1,\dots,n).
\]
\end{enumerate}
\end{theorem}
\begin{proof}
For the necessity of the condition in (a), see 
\cite[Lemma~10.3.2]{kedlaya-course}. For all of the other assertions,
see \cite[Theorems~10.4.2 and~10.4.4]{kedlaya-course}.
\end{proof}

There is also an ``off-centered'' variant.
\begin{defn} \label{D:off-centered}
Let $F \langle x \rangle$ denote the standard Tate algebra in the variable $x$
over $F$, i.e., the completion of $F[x]$ for the 1-Gauss norm.
For  $\eta \in F$ with $|\eta| = 1$,
let $\psi_\eta: F \langle x \rangle \to F \langle x \rangle$ 
be the substitution $x \mapsto (x+\eta)^p - \eta^p$.
For $I \subseteq (p^{-1/(p-1)}, 1]$, 
a \emph{Frobenius antecedent centered at $\eta$} 
of $\calE$ is a $\nabla$-module
$\calF$ on $A_F(I^p)$ equipped with an isomorphism $\calE \cong 
\psi_\eta^* \calF$,
such that $R(\calF_\rho) >p^{-p/(p-1)}$ for all $\rho \in I^p$.
\end{defn}

\begin{theorem} \label{T:antecedent2}
Retain notation as in Definition~\ref{D:off-centered}.
\begin{enumerate}
\item[(a)]
A Frobenius antecedent centered at $\eta$ of $\calE$ exists if and only if 
$R(\calE_\rho) > p^{-1/(p-1)}$ for all $\rho \in I$.
\item[(b)]
A Frobenius antecedent centered at $\eta$ is unique if it exists.
\item[(c)]
If $\calF$ is the Frobenius antecedent centered at $\eta$ of $\calE$, then
\[
f_i(\calF, ps) = p f_i(\calE, s) \qquad (i=1,\dots,n).
\]
\end{enumerate}
\end{theorem}
\begin{proof}
See \cite[Theorem~10.8.2]{kedlaya-course}.
\end{proof}

Note that Frobenius descendants also come in normal and off-centered
variants; we only use the off-centered version here.
\begin{defn}
For $I = [0,1)$,
the \emph{Frobenius descendant centered at $\eta$} of $\calE$ 
is the pushforward $(\psi_\eta)_* \calE$; it again may be naturally
viewed as a $\nabla$-module on $A_F[0,1)$, of rank $p \rank(\calE)$.
\end{defn}

\begin{theorem} \label{T:descendant}
Pick  $\rho \in (p^{-1/(p-1)},1)$, and let
$s_1,\dots,s_n$ be the subsidiary radii of $\calE_\rho$.
Then the subsidiary radii of $((\psi_\eta)_* \calE)_{\rho^p}$ 
consist of
\[
\bigcup_{i=1}^n \begin{cases} \{s_i^p, p^{-p/(p-1)} \mbox{ ($p-1$ times)}\} &
s_i > p^{-1/(p-1)} \\
\{p^{-1} s_i \mbox{ ($p$ times)}\} & s_i \leq p^{-1/(p-1)}.
\end{cases}
\]
\end{theorem}
\begin{proof}
See \cite[Theorem~10.8.3]{kedlaya-course}.
\end{proof}
\begin{cor} \label{C:descendant}
For $r \in (0, \frac{1}{p-1} \log p)$, if 
$f_i(\calE,r) < \frac{1}{p-1} \log p$, then
$f_{i+(p-1)n}((\psi_\eta)_* \calE,pr) = p f_i(\calE,r)$ and
\[
\sum_{j=1}^{i+(p-1)n} f_i((\psi_\eta)_* \calE, pr)
= pn \log p + p \sum_{j=1}^i f_i(\calE,r).
\]
\end{cor}

\subsection{Cyclic vectors}

\begin{defn}
A \emph{cyclic vector} of $\calE$ is an element $\bv \in \Gamma(A_F(I), \calE)$
such that $\bv, \frac{d}{dx} \bv, \cdots, \frac{d^{n-1}}{dx^{n-1}} \bv$
are linearly independent (but need not form a basis).
Such an element always exists; see 
\cite[Theorem~5.7.3]{kedlaya-course}.
\end{defn}

\begin{theorem} \label{T:cyclic}
Let $\bv$ be a cyclic vector of $\calE$, and choose $P_n, \dots, P_0 \in
\Gamma(A_F(I), \calO)$ not all zero such that
$P_n \frac{d^n}{dx^n} \bv + \cdots + P_0 \bv = 0$. 
Pick  $\rho \in I$, form the Newton polygon of the polynomial
$P = P_n T^n + \cdots + P_0$ measured using $|\cdot|_\rho$,
and let
$r_1 \leq \dots \leq r_m$ be the slopes of the polygon
which are less than $\log \rho$. Then
\[
f_i(\calE, - \log \rho) = \frac{1}{p-1} \log p - r_i \qquad (i=1,\dots,m);
\]
moreover, if $m < n$, then $f_{m+1}(\calE,-\log \rho) 
\leq \frac{1}{p-1} \log p - \log \rho$.
\end{theorem}
\begin{proof}
See \cite[Corollary~6.5.4]{kedlaya-course}.
\end{proof}

\subsection{Dwork modules}

We introduce an example of Dwork which we will use crucially.
\begin{lemma} \label{L:dwork}
Assume that $F$ contains an element $\pi$ with $\pi^{p-1} = -p$ (in other
words, $F$ contains a primitive $p$-th root of unity). Let $\calE$
be the $\nabla$-module on $A_F(I)$ which is free on one generator $\bv$
satisfying
\[
\nabla(\bv) = \pi r \bv \otimes dx
\]
for  $r = \sum_{j \in \ZZ} r_j x^j \in \Gamma(A_F(I), \calO)$. 
Suppose that $r_j = 0$ whenever
$j+1$ is divisible by $p$. Then
\begin{equation} \label{eq:Dwork radius}
R(\calE_\rho) = \min\{\rho, \min_j \{|r_j|^{-1} \rho^{-j}\}\} \qquad
(\rho \in I \setminus \{0\}).
\end{equation}
Moreover, if $\rho < |r_j|^{-1} \rho^{-j}$ for all $\rho \in I \setminus \{0\}$
and all $j \in \ZZ$, then $\calE$ is constant.
\end{lemma}
\begin{proof}
If $r = r_0$, it is straightforward to check that
$R(\calE_\rho) = \min\{\rho, |r_0|^{-1}\}$
from the convergence behavior of the exponential series
\cite[Example~9.3.5]{kedlaya-course}. 
Similarly, if $r = \sum_j r_j x^j$ with $\rho < |r_j|^{-1} \rho^{-j}$
for all $j$, then $\calE_\rho$ is constant because the
relevant exponential converges, so $R(\calE_\rho) = \rho$.

If $r = r_j x^j$ with $j+1$ not divisible
by $p$, we may infer $R(\calE_\rho) = \min\{\rho, |r_j|^{-1} \rho^{-j}\}$
by writing $\calE$ as the pullback along the map $x \mapsto x^{j+1}$ of the
$\nabla$-module on $A_F(I^{j+1})$ free on one generator $\bv$ satisfying
$\nabla(\bv) = (j+1)^{-1} \pi r_j \bv \otimes dx$,
and using a tame base change argument (Lemma~\ref{L:tame base change}).

In the general case, we may assume $I = [\rho,\rho]$.
We may write $\calE$ as a tensor product of modules,
one corresponding to each term $r_j x^j$ of $r$ for which
$\rho \geq |r_j|^{-1} \rho^{-j}$, and one corresponding to
the remaining terms.
By the last comment in 
Definition~\ref{D:generic radius}, we obtain
\eqref{eq:Dwork radius} for all $\rho$ for which no two of the
terms $|r_j|^{-1} \rho^{-j}$ coincide and are less than $\rho$. 
This excludes only a discrete
set of values of $\rho$, to which we may extend \eqref{eq:Dwork radius}
by the continuity aspect of Theorem~\ref{T:variation}.
\end{proof}

\section{Relative analysis}
\label{sec:relative}

In this section, we analyze differential modules on discs 
defined over rings which themselves carry
multiple norms. A typical such ring would be
the ring of functions on an annulus; however, for the basic analysis,
it will be more convenient to use a slightly more abstract class of rings.
The relevance of this analysis will become more clear when
we specialize to a more overtly geometric setting in the next section.

\subsection{Analytic rings}

We recall the setup of \cite[\S 2]{kedlaya-slope}, although with
a different purpose in mind.

\begin{defn} \label{D:analytic}
Let $L$ be a complete extension of $K$ with the same value group
as $K$ and having perfect residue field $\ell$.
Then 
for any uniformizer $\pi$ of $K$, each $a \in L$ has a unique representation
as a sum $\sum_{i=m}^\infty \pi^i [\overline{a_i}]$ with
$\overline{a_i} \in \ell$, where $[\cdot]$ denotes the Teichm\"uller lift.
For $v$ a real valuation on $\ell$ and $\rho \in (0,1)$, define
\[
v_{n,\rho}(a) = \sup_{i \leq n} \{ |\pi|^i \rho^{v(\overline{a_i})}\};
\]
this function does not depend on the choice of $\pi$.
For $\epsilon \in (0,1)$, let $L_{v,\epsilon}$ be the set of
$a = \sum_{i=m}^\infty \pi^i [\overline{a_i}] \in L$ such that
$|\pi|^i \epsilon^{v(\overline{a_i})} \to 0$ as $i \to \infty$.
For $a \in L_{v,\epsilon}$ and $\rho \in [\epsilon,1)$, define
\[
|a|_{\rho^v} = \lim_{n \to \infty} v_{n,\rho}(a) = 
\sup_{i: \overline{a_i} \neq 0} \{|\pi|^i \rho^{v(\overline{a_i})}\};
\]
this is a norm on $L_{v,\epsilon}$.
For $I \subseteq (0,1)$ an interval, pick any $\epsilon \in (0,1)$
such that $I \subseteq [\epsilon,1)$, and let $L_{v,I}$ be the Fr\'echet
completion of $L_{v,\epsilon}$ for the norms $|\cdot|_{\rho^v}$
for $\rho \in I$;
the result is independent of $\epsilon$.
Note that for $\rho \in I$ and $n \in \ZZ$,
the functions $v_{n,\rho}$ also extend continuously to $L_{v,I}$.
\end{defn}

\begin{remark} \label{R:principal}
For $\epsilon \in (0,1)$, the ring $L_{v, \epsilon}$ is a
principal ideal domain \cite[Proposition~2.6.5]{kedlaya-slope}.
For $I \subseteq (0,1)$ a closed interval, the ring $L_{v,I}$ is also
a principal ideal domain \cite[Proposition~2.6.8]{kedlaya-slope}.
Both of these rely crucially on the fact that $K$ is discretely valued.
It is shown in \cite[Theorem~2.9.6]{kedlaya-slope} for $I = [\epsilon,1)$,
and likely true for arbitrary $I$, that $L_{v,I}$ is a B\'ezout domain
(a domain in which every finitely generated ideal is principal); we will
not need this.
\end{remark}

\begin{remark}
Beware that, as per the errata to \cite{kedlaya-slope}
(refuting a claim in \cite[\S 2.5]{kedlaya-slope}),
elements of $L_{v,I}$ cannot necessarily be written as doubly infinite sums
$\sum_{i \in \ZZ} \pi^i [\overline{a_i}]$.
\end{remark}

\subsection{Adding a second direction}

We now add a second direction to the rings of the previous subsection.

\begin{defn}
For $I \subseteq (0,1)$ an interval, let $R_{v,I}$
be the ring of series $b = \sum_{i=0}^\infty b_i x^i \in
L_{v,I}\llbracket x \rrbracket$ such that
for each $\rho \in I$, $|b_i|_{\rho^v} \to 0$ as $i \to \infty$.
For $\rho \in I$ and $s \geq 0$,
define the norm $|\cdot|_{\rho^v,s}$ on $R_{v,I}$ by the formula
\[
|b|_{\rho^v,s} = \sup_i \{|b_i|_{\rho^v} \rho^{is}\};
\]
note that for fixed $\rho$, $|b|_{\rho^v,s}$ is a nonincreasing function
of $s$.
\end{defn} 

By computing termwise, we deduce the following lemma 
\cite[Lemma~3.1.6]{kedlaya-part1}.
\begin{lemma} \label{L:hadamard}
For $\rho_1, \rho_2 \in I$, $s_1, s_2 \geq 0$, and $c \in [0,1]$,
put $\rho = \rho_1^c \rho_2^{1-c}$ and $s = cs_1 + (1-c)s_2$. Then for
any $b \in R_{v,I}$,
\[
|b|_{\rho^v,s} \leq |b|_{\rho_1^v,s_1}^c |b|_{\rho_2^v,s_2}^{1-c}.
\]
\end{lemma}

\begin{defn}
For $I \subseteq (0,1)$ an interval, 
let $I_1 \subseteq I_2 \subseteq \cdots$ be an increasing sequence of
closed intervals with union $I$. For this definition only,
put $R_i = R_{v,I_i}$.
Let $C_{v,I}$ be the category whose
elements are sequences $M_1,M_2,\dots$ with $M_i$ a finite locally free
$\nabla$-module over $R_i$, 
equipped with isomorphisms $M_{i+1} \otimes_{R_{i+1}} R_i \cong M_i$.
Note that this category is canonically independent of the choice of the
$I_i$.
Let $C_{v,*}$ denote the direct limit of the categories $C_{[\epsilon,1)}$ as
$\epsilon \to 1^-$.
\end{defn}

\begin{defn} \label{D:cross section}
For $M \in C_{v,I}$ and $\rho \in I$, let $F_\rho$ be the completion
of $\Frac L_{v,[\rho,\rho]}$ under $|\cdot|_{\rho^v}$. We may
then base extend $M$ to obtain a $\nabla$-module $M_\rho$
over the Tate algebra $F_\rho\langle x \rangle$, or in other words,
a $\nabla$-module on the disc $A_{F_\rho}[0,1]$.
We call $M_\rho$ the \emph{cross section} of $M$
at $\rho$. For $s \geq 0$, put $R(M, \rho,s) = R((M_\rho)_{\rho^s})$ and
$f_i(M, r,s) = f_i(M_{e^{-r}}, rs)$
according to Definition~\ref{D:subsidiary}. 
\end{defn}

\subsection{Passage to the generic fibre}

\begin{defn}
Let $v'$ be the $0$-Gauss valuation on $\ell(x)$ with respect to
the valuation $v$ on $\ell$. Let $\ell'$ be the completion of $\ell(x)$
under $v'$. Put 
\[
\tilde{L}' = W((\ell')^{\perf}) \otimes_{W(k)} K.
\]
Let $L'$ be the closure of $L(x)$ inside $\tilde{L'}$ for the $p$-adic
topology; it is a complete discretely valued field with residue
field $\ell(x)$.
\end{defn}

\begin{remark}
Although $L'$ does not have perfect residue field, it fits into the
formalism of \cite[\S 2.2]{kedlaya-slope} using the Frobenius lift
$x \mapsto x^q$. We may thus define $L'_{v',I}$ as in 
Definition~\ref{D:analytic}, and for $I$ closed, the result is again
a principal ideal domain by \cite[Proposition~2.6.8]{kedlaya-slope}.
Moreover, given $M \in C_{v,I}$ for $I$ closed, we can form a base extension
$M'$ to $L'_{v',I}$, which will then be a finite free $L'_{v',I}$-module.
\end{remark}

We will limit our use of this construction to the following key results,
which are similar (and similarly proved) to Theorem~\ref{T:variation}
and Theorem~\ref{T:decomposition}, respectively.
\begin{lemma} \label{L:convex}
Take $M \in C_{v,I}$ of rank $n$.
For $i=1,\dots,n$, for $s \geq 0$ fixed, the function
\[
r \mapsto f_1(M,r,s) + \cdots + f_i(M,r,s)
\]
is convex. 
\end{lemma}
\begin{proof}
There is no harm in assuming $I$ is closed.
Also, we may enlarge $\ell$ to contain an element $\overline{a}$ with
$v(\overline{a}) = s$, and then pull back along the map $x \mapsto
x [\overline{a}]$ to reduce to the case $s=0$. We may now replace $M$
by its base extension $M'$ to $L'_{v',I}$; since this module is free, we
may imitate the proof of \cite[Theorem~11.3.2]{kedlaya-course}
to obtain the desired result.

Let us explain a bit more explicitly how this works.
(For $i=1$, one can instead proceed
as in \cite[Proposition~4.2.6]{kedlaya-part3}.)
It suffices to prove that each $c>0$, the function
$r \mapsto \sum_{j=1}^i \max\{f_j(M_{e^{-r}}, 0), c\}$
is convex. To check this, it suffices to do so for $r$ in a neighborhood of any
given $r_0$ in the interior of $-\log I$. By rescaling, we may 
assume $r_0 \in (0, \frac{1}{p-1} \log p)$; then
by replacing $M$ by a Frobenius descendant and using
Theorem~\ref{T:descendant}, we may reduce the case for a given
$c$ to the case for $pc$. Repeating this construction finitely
many times, we may force $c > \frac{1}{p-1} \log p$.

We may treat this case by following \cite[Lemma~11.5.1]{kedlaya-course}.
(For an even more similar argument, see \cite[Theorem~2.2.6(d)]{kedlaya-xiao}.)
Namely, the argument in the proof of \cite[Lemma~11.5.1]{kedlaya-course}
shows that there is a basis of $M'$
on which $\frac{d}{dx}$ acts via a matrix $N$ such that,
for $r$ in a neighborhood of $r_0$, the multiset of norms 
under $|\cdot|_{e^{-r}}$ of the eigenvalues
of $N$ coincides with the multiset of the $p^{-1/(p-1)} 
e^{f_j(M_{e^{-r}}, 0)}$ in
their values greater than $e^{c}$.
We can thus read off the desired convexity from the
Newton polygon of the characteristic polynomial of
$N$, i.e., by following \cite[Theorem~11.2.1]{kedlaya-course}
(and \cite[Remark~11.2.4]{kedlaya-course})  but using
Lemma~\ref{L:hadamard} in place of
\cite[Proposition~8.2.3(b)]{kedlaya-course}.
\end{proof}

\begin{lemma} \label{L:convex decomp}
Take $M \in C_{v,I}$ of rank $n$.
Suppose that for $i=1,\dots,n$, the function $f_i(M,r,0)$ is affine.
Then for any closed subinterval $I'$ of the interior of $I$,
there exists a direct sum decomposition 
$M' \otimes_{L'_{v,I}} L'_{v,I'} \cong \oplus_j  M'_j$
such that for each $j$, the functions
$f_i(M'_j,r,0)$ are the same for all $i$.
(This decomposition becomes unique if we insist that the
functions $f_i(M'_j,r,0)$ be distinct for distinct $j$.)
\end{lemma}
\begin{proof}
In the notation of the proof of Lemma~\ref{L:convex},
this follows as in the proof of \cite[Theorem~12.2.2]{kedlaya-course}
(or more exactly \cite[Theorem~2.3.5]{kedlaya-xiao}).
\end{proof}

\subsection{Frobenius structures on $\nabla$-modules}

\begin{defn}
Let $\sigma_L$ be the unique $q$-power Frobenius lift on $L$
extending $\sigma_K$; it acts by 
\[
\sigma_L(a) = \sum_{i=m}^\infty \sigma_K(\pi)^i [\overline{a_i}^q]
\]
and satisfies
\[
v_{n,\rho}(a) = v_{n,\rho^{1/q}}(\sigma_L(a)).
\]
One consequence is that 
$\sigma_L$ carries $L_{v,\epsilon}$ to $L_{v,\epsilon^{1/p}}$;
another is that for $a \in L_{v,\epsilon}$ and
$\rho \in [\epsilon,1)$, $|a|_{\rho^v} = |\sigma_L(a)|_{\rho^{v/q}}$.
(Here we write $\rho^{v/q}$ as shorthand for $(\rho^{1/q})^v$.)
As a result, $\sigma_L$ extends continuously to $L_{v,I}$ and 
carries that ring
into $L_{v,I^{1/q}}$.
\end{defn}

\begin{defn} \label{D:Frobenius lift}
By a \emph{Frobenius lift} on $R_{v,[\epsilon,1)}$ relative to $\sigma_L$,
we will mean a map $\sigma: R_{v,[\epsilon,1)} \to R_{v,[\epsilon^{1/q},1)}$ 
of the form
\[
\sigma\left(\sum_i b_i x^i\right) = \sum_i \sigma_L(b_i) \sigma(x)^i
\]
for which there exists $\lambda \in (0,1)$
such that $|\sigma(x) - x^q|_{\rho^{v/q},0} \leq \lambda$
for $\rho \in [\epsilon,1)$.
For $b \in R_{v,[\epsilon,1)}$ and $\rho \in [\epsilon,1)$,
we have $|b|_{\rho^v,0} = |\sigma(b)|_{\rho^{v/q},0}$.
(If we also have $\sigma(x) \equiv 0 \pmod{x^q}$, then
we also have
$|\sigma(x) - x^q|_{\rho^{v/q},s} \leq \lambda |x^q|_{\rho^{v/q},s}$
for all $s \geq 0$.)
We call $\sigma$ the \emph{standard Frobenius lift} if $\sigma(x) = x^q$.
\end{defn}

\begin{defn}
Let $\sigma: R_{v,[\epsilon,1)} \to R_{v,[\epsilon^{1/q},1)}$ 
be a Frobenius lift. For $M \in C_{v,[\epsilon,1)}$,
a \emph{Frobenius structure} on $M$ with respect to $\sigma$
is an isomorphism $F: \sigma^* M \cong M$ of objects 
in $C_{v,[\epsilon^{1/q},1)}$.
\end{defn}

\begin{lemma} \label{L:solvable}
For $M \in C_{v,[\epsilon,1)}$ admitting a
Frobenius structure, $R(M, \rho, 0) \to 1$ as $\rho \to 1^-$.
\end{lemma}
\begin{proof}
Put $I_m = [\epsilon^{1/q^m}, \epsilon^{1/q^{m+1}}]$ for $m=0,1,\dots$.
Choose a basis $\be_1,\dots,\be_n$ of $M \otimes R_{v,I_0}$,
on which $\frac{d}{dx}$ acts via the matrix $N$.
Then the action on $F(\be_1),\dots,F(\be_n)$ is via the matrix
$(d\sigma(x)/dx) \sigma(N)$. Hence for any $\delta > 0$,
we can find $m_0$ such that for $m \geq m_0$, 
the matrix of action $N_m$ of $\frac{d}{dx}$ on 
the basis $F^m(\be_1), \dots, F^m(\be_n)$ of $M \otimes R_{v,I_m}$
satisfies
$|N_m|_{\rho^v,0} < \delta$ for $\rho \in I_m$.
As in the proof of \cite[Lemma~6.4]{kedlaya-mono-over}, this implies 
the claim.
\end{proof}

\begin{lemma} \label{L:change Frob}
Let $\sigma_1$ and $\sigma_2$ be two Frobenius lifts on $R_{v,[\epsilon,1)}$, 
let $M \in C_{v,[\epsilon,1)}$,
and let $F_1: \sigma_1^* M \cong M$ be a Frobenius structure with respect
to $\sigma_1$.
Then there exists $\eta \in [\epsilon,1)$ such that for $I \subseteq [\eta,1)$
closed and $\bv \in M \otimes R_{v,I}$, the series
\[
F_2(\bv) = \sum_{i=0}^\infty \frac{(\sigma_2(x) - \sigma_1(x))^i}{i!}
F_1 \left(\frac{d^i}{dx^i}(\bv) \right)
\]
converges in $M \otimes R_{v,I^{1/q}}$. For such $\eta$,
$F_2$ is a Frobenius structure with respect to $\sigma_2$
on the restriction of $M$ to $C_{v,[\eta,1)}$.
\end{lemma}
\begin{proof}
By Lemma~\ref{L:solvable}, there exists $\eta \in (0,1)$ such that
$R(M, \rho, 0) > \lambda \geq |\sigma_2(x) - \sigma_1(x)|_{\rho^v,0}$ 
for $\rho \in 
[\eta,1)$. For such $\eta$, we obtain the desired convergence.
\end{proof}

\subsection{Behavior at the boundary}

\begin{lemma} \label{L:constant breaks1}
Suppose $M \in C_{v,[\epsilon,1)}$
is of rank $n$ and admits a Frobenius structure for the standard
Frobenius lift $\sigma$.
If $r_0 \in (0, -\log \epsilon]$ and $s \geq 0$
satisfy $f_1(M,r_0,s) < \frac{p}{p-1} \log p + r_0 s$, then
there exist constants $b_i(M,s)$ for $i=1,\dots,n$ such that
for $r \in (0, r_0]$, $f_i(M,r,s) = b_i(M,s) r$.
\end{lemma}
\begin{proof}
Define $b_i(M,s) = f_i(M,r_0,s)/r_0$.
By Theorem~\ref{T:antecedent}, for each nonnegative integer $m$,
$f_i(M,r_0/q^{m+1},s) = f_i(M,r_0/q^m,s)/q$ for $i=1,\dots,n$.
Consequently, the function $f_1(M,r,s) + \cdots + f_i(M,r,s)$, which is 
convex by Lemma~\ref{L:convex}, agrees with the linear function
$(b_1(M,s) + \cdots + b_i(M,s))r$ for 
$r = r_0/q^m$ for each nonnegative integer $m$.
On each interval $[r_0/q^{m+2}, r_0/q^m]$, we now have a
convex function which agrees with an affine
function at both endpoints plus one interior point,
so the two functions must agree identically.
That is, $f_1(M,r,s) + \cdots + f_i(M,r,s) =
(b_1(M,s) + \cdots + b_i(M,s))r$ for
$r_0/q^{m+2} \leq r \leq r_0/q^m$ for each nonnegative integer $m$.
This yields the desired result.
\end{proof}

\begin{cor} \label{C:constant breaks2}
Suppose $M \in C_{v,[\epsilon,1)}$ is of rank $n$ and admits
 a Frobenius structure. Then  there exist
$r_0 \in (0, -\log \epsilon]$ and constants $b_i(M,s)$ for $i=1,\dots,n$
such that
for $r \in (0, r_0]$ and $s \geq 0$, $f_i(M,r,s) = b_i(M,s) r$.
\end{cor}
\begin{proof}
By Lemma~\ref{L:change Frob}, we may change to the standard Frobenius lift 
$\sigma$, at the expense of possibly changing $\epsilon$.
By Lemma~\ref{L:solvable}, we have $f_1(M,r,0) \to 0$ as
$r \to 0^+$; thus for $r$ sufficiently small, we have
\[
\frac{p}{p-1} \log p + rs >
 f_1(M,r,0) + rs \geq f_1(M, r,s).
\]
(The second inequality above follows from 
Theorem~\ref{T:variation}, which implies that
$f_1(M,r,s)$ is convex and identically equal to $rs$
for $s$ large; hence the slope of any secant line of the graph
of $f_1(M,r,s)$ is at most $rs$.)
We may then apply Lemma~\ref{L:constant breaks1}.
\end{proof}

\begin{defn} \label{D:maintain by translate}
For $I \subseteq [\epsilon,1)$
and $h \in L_{v,\epsilon}$ with $|h|_{\rho^v} \leq 1$ for 
$\rho \in I$,
let $T_h: R_{v,I} \to R_{v,I}$ be the translation
$\sum_{i=0}^\infty b_i x^i \mapsto \sum_{i=0}^\infty b_i (x+h)^i$.
For any $b \in R_{v,I}$, $\rho \in I$, $s \geq 0$
such that $\rho^s \geq |h|_{\rho^v}$,
we have $|b|_{\rho^v,s} = |T_h(b)|_{\rho^v,s}$; in particular,
this always holds for $s=0$.
\end{defn}

\begin{lemma} \label{L:maintain by translate}
For $M \in C_{v,[\epsilon,1)}$ of rank $n$,
for $h \in L_{v,\epsilon}$ with $|h|_L \leq 1$ and $v(\overline{h}) \geq 0$,
and $s \in [0, v(\overline{h})]$,
we have $b_i(M,s) = b_i(T_h^* M, s)$ for $i=1,\dots,n$.
\end{lemma}
\begin{proof}
Suppose first that $\overline{h} \neq 0$, or equivalently $|h|_L = 1$.
We then have
$|h|_{\rho^v} = \rho^{v(\overline{h})}$ for $\rho$ sufficiently close to 1.
For such $\rho$, if $s \in [0, v(\overline{h})]$,
then $|b|_{\rho^v,s} = |T_h(b)|_{\rho^v,s}$ for all $b \in R_{v,[\rho,\rho]}$.
The claim follows at once.

Suppose next that $\overline{h} = 0$, or equivalently $|h|_L < 1$.
In this case, $|h|_{\rho^v}$ tends to $|h|_L$ as $\rho$ tends to 1.
Consequently, for each $s \geq 0$, for $\rho$ sufficiently close to 1,
we have $\rho^s \geq |h|_{\rho^v}$; for such $\rho$,
again $|b|_{\rho^v,s} = |T_h(b)|_{\rho^v,s}$ for all $b \in R_{v,[\rho,\rho]}$,
and the claim follows.
\end{proof}

\begin{remark} \label{R:enlarge ell}
The value of $b_i(M,s)$ is impervious to enlargement of $\ell$.
\end{remark}

\subsection{Variation along a path}

\begin{hypothesis} \label{H:variation}
Through the end of Subsection~\ref{subsec:Dwork},
suppose that $\ell$ is complete under $v$ and algebraically closed.
Let $\ell \langle x \rangle$ denote the standard Tate algebra in $x$
over $\ell$.
Let $M \in C_{v,[\epsilon,1)}$ be an object of rank $n$ admitting
a Frobenius structure.
\end{hypothesis}

\begin{defn}
Let $\alpha \in \DD_{\ell}$ be a point of 
type (ii) or (iii). By Proposition~\ref{P:classify points},
we may represent $\alpha = \alpha_{\overline{h}, r}$ for some
$\overline{h} \in \gotho_{\ell}$ and $r = r(\alpha) \in (0,1]$. 
Choose $h \in L_{v,\epsilon}$ with $|h|_L \leq 1$ lifting $\overline{h}$,
and define $b_i(M, \alpha) = b_i(T_h^* M, -\log r(\alpha))$ for $i=1,\dots,n$.
By Lemma~\ref{L:maintain by translate},
this definition does not depend on the choice either of $h$ or of
$\overline{h}$. (In particular, one may choose $h = [\overline{h}]$.)
Moreover, as in Remark~\ref{R:enlarge ell}, we may replace $\ell$ by any 
complete extension without changing the $b_i(M, \alpha)$.
In particular, by enlarging $\ell$, we may convert points of
type (iv) into points of type (ii) or (iii), and so we may define
$b_i(M,\alpha)$ for such $\alpha$ also.
\end{defn}

\begin{defn} \label{D:breaks}
For $\alpha \in \DD_{\ell}$, for $s \in [0, -\log r(\alpha)]$
(or $s \in [0,+\infty)$ if $r(\alpha)=0$)
and $i=1,\dots,n$,
define $b_i(M, \alpha, s)$ to be the value $b_i(M, \beta)$, for $\beta$
the point of radius $e^{-s}$ on the generic path to $\alpha$.
\end{defn}

\begin{prop} \label{P:variation1}
For fixed $M, \alpha$, the functions $b_i(M,\alpha,s)$ for $i=1,\dots,n$
have the following properties.
\begin{enumerate}
\item[(a)]
Each $b_i(M, \alpha,s)$ is continuous and
piecewise affine-linear on $[0, -\log r(\alpha)]$
(or on $[0,+\infty)$ if $r(\alpha)=0$)
with slopes in $\frac{1}{n!} \ZZ$.
\item[(b)]
The function $b_1(M,\alpha,s) + \cdots + b_i(M,\alpha,s)$ is convex,
and is nonincreasing
in a neighborhood of any point where $b_i(M, \alpha,s) > s$.
(If $b_1(M, \alpha,s_0) = s_0$, then $b_1(M, \alpha,s) = s$ for $s \geq s_0$
by Remark~\ref{R:transfer}.)
\end{enumerate}
\end{prop}
\begin{proof}
Again, we may assume that $\alpha$ is of type (ii) or (iii).
By Lemma~\ref{L:change Frob}, $M$ admits a Frobenius structure
for the standard Frobenius lift;
this induces a Frobenius structure on $T_{h}^* M$
for the Frobenius lift
\[
x \mapsto (x-h)^q + h^q.
\]
Thus by applying a suitable $T_h^*$, we may reduce to the case 
where $\alpha \geq \alpha_{0,0}$.
By Corollary~\ref{C:constant breaks2}, we 
can choose $r_0 > 0$ such that $f_i(M,r,s) = b_i(M,s)r$ for $r \in (0,r_0]$
and $s \geq 0$. We thus deduce the claims by applying
Theorem~\ref{T:variation} to the cross section $M_{e^{-{r_0}}}$.
\end{proof}

\begin{defn} \label{D:terminal slope}
If $r(\alpha) < 1$, then Proposition~\ref{P:variation1} implies
that for some $s_0 \in (0, -\log r(\alpha))$,
each function $b_i(M,\alpha,s)$ is affine-linear on $[s_0, -\log r(\alpha)]$.
We call the slope of $b_i(M,\alpha,s)$ in this range the \emph{terminal
slope} of $b_i(M,\alpha,s)$.
\end{defn}

\begin{remark} \label{R:radius 0}
For $\alpha$ of type (i) (i.e., $r(\alpha) = 0$), 
the function $b_1(M, \alpha,s) - s$ has
slope at most $-1$ as long as its value is positive, so it must at
some point $s_0$ achieve the value 0. By Remark~\ref{R:transfer}, 
we have $b_1(M,\alpha,s) = s$ for any $s > s_0$. Since
$b_1(M,\alpha,s) \geq b_i(M,\alpha,s) \geq s$ for all $s$, we also have
$b_i(M,\alpha,s) = s$ for $i=1,\dots,n$ and $s \geq s_0$.
That is, all of the terminal slopes are equal to 1.
\end{remark}

\subsection{More cross-sectional analysis}

Throughout this subsection, retain Hypothesis~\ref{H:variation}. 
We wish to use the cross section $M_\rho$ to make a closer analysis
of the $b_i(M,\alpha,s)$ that is peculiar to the case
where  $\alpha$ is of type (iv). 
To do so, we must verify that under certain explicit conditions, we can make
a uniform choice of $\rho$ such that $M_\rho$ can be used to 
recover the $b_i(M,\alpha,s)$ for all $s$. Lemma~\ref{L:constant 
breaks1} does this for $\alpha = \alpha_{0,0}$, but we cannot uniformly
translate into this case without converting $\alpha$ to a type other than
(iv).

We first show that in certain ranges, the norm 
$|T_{[\overline{h}]}(\cdot)|_{\rho^v,s}$ can be written in terms of
Teichm\"uller lifts.
\begin{lemma} \label{L:Frob stability}
Let $\sigma$ be the standard Frobenius lift.
Suppose $\overline{h} \in \gotho^*_\ell$, $\rho \in [\epsilon,1)$, and 
\[
0 \leq s \leq \frac{\log p}{(-\log \rho)}.
\]
For $b = \sum_{i=0}^\infty b_i x^i \in (\gotho_L \llbracket x \rrbracket 
\otimes_{\gotho_L} L) \cap F_\rho \langle x \rangle$,
write $b = \sum_{j=m}^\infty \pi^j [\overline{b_j}]$.
Then
\begin{equation} \label{eq:compare norms}
|T_{[\overline{h}]}(b)|_{\rho^v,s} = 
\sup_j \{|\pi|^j \rho^{v_s(\overline{b_j})}\},
\end{equation}
where $v_s$ is the $s$-Gauss valuation on 
$\ell \langle \overline{x}-\overline{h} \rangle$
(i.e., in the variable $\overline{x} - \overline{h}$).
\end{lemma}
Note that while $\sigma$ does not appear in the following proof,
the choice of $\sigma$ affects the expression of $b$ as a sum of
Teichm\"uller elements. In particular, making $\sigma$ standard
forces $x = [\overline{x}]$.
\begin{proof}
Let $|\cdot|_1$ and $|\cdot|_2$ be the norms defined by the left and right
sides of \eqref{eq:compare norms}, respectively. 
For any fixed $b$, both $|b|_1$ and $|b|_2$ vary continuously in $s$.
Consequently, we may assume $s < (\log p)/(-\log \rho)$ hereafter.

We first note that
\[
x-[\overline{h}] = [\overline{x}-\overline{h}] + \sum_{i=1}^\infty p^i 
P_i(x^{1/p^i}, [\overline{h}^{1/p^i}]),
\]
where $P_i$ is homogeneous of degree $p^i$ with coefficients in $\ZZ$. Since
$|x|_2, |[\overline{h}]|_2 \leq 1$ and 
$s < (\log p)/(-\log \rho)$, this implies
\[
|x - [\overline{h}] - [\overline{x}-\overline{h}]|_2 \leq p^{-1} < \rho^s 
= |[\overline{x}-\overline{h}]|_2 = |x- [\overline{h}]|_1.
\]
In particular, $|x-[\overline{h}]|_2 = |x - [\overline{h}]|_1 = \rho^s$.
For $i \geq 0$, we have
\begin{align*}
|(x - [\overline{h}])^i - [\overline{x} - \overline{h}]^i|_2
&= 
|x - [\overline{h}] - [\overline{x} - \overline{h}]|_2
\left|
\frac{(x - [\overline{h}])^i - [\overline{x} - \overline{h}]^i}{x - [\overline{h}] - [\overline{x} - \overline{h}]}
\right|_2
\\
&\leq p^{-1} \rho^{s(i-1)} < |(x - [\overline{h}])^i|_1.
\end{align*}

To deduce \eqref{eq:compare norms} for a given $b$, write
$b = \sum_{i=0}^\infty b_i (x - [\overline{h}])^i$ with $b_i \in L$.
For each $s \in [0, (\log p)/(-\log \rho))$ 
for which there is a unique value $i$ maximizing
$|b_i (x-[\overline{h}])^i|_1$, we have
\begin{align*}
|b|_1 &= |b_i (x-[\overline{h}])^i|_1 \\
&= |b_i [\overline{x} - \overline{h}]^i|_2 \\
|b - b_i [\overline{x} - \overline{h}]^i|_2 &< |b_i (x-[\overline{h}])^i|_1.
\end{align*}
Hence $|b|_1 = |b|_2$ for such $s$. 
Since this excludes only
countably many $s$, we may deduce the claim for all $s$
by continuity.
\end{proof}
\begin{cor} \label{C:Frob stability1}
For $\sigma, \overline{h}, \rho, s$ as in Lemma~\ref{L:Frob stability},
for any $b \in F_\rho \langle x \rangle$,
\[
|T_{[\overline{h}]}(b)|_{\rho^v,s} = 
|T_{[\overline{h}]}(\sigma(b))|_{\rho^{v/q},s}.
\]
\end{cor}
\begin{proof}
For $b \in L_{v,[\rho,\rho]} \langle x \rangle$,
this follows from Lemma~\ref{L:Frob stability} and the fact that
$(\gotho_L \llbracket x \rrbracket 
\otimes_{\gotho_L} L)
\cap F_\rho \langle x \rangle$ is dense in $L_{v,[\rho,\rho]} 
\langle x \rangle$. The claim for
$b \in F_\rho \langle x \rangle$ follows immediately thereafter.
\end{proof}
\begin{cor} \label{C:Frob stability2}
For $\sigma, \overline{h}, \rho, s$ as in Lemma~\ref{L:Frob stability},
put $r = -\log \rho$.
If $f_1(M, r, 0) < \frac{p}{p-1} \log p$, then
\[
f_i(T_{[\overline{h}]}^* M,r/q, s)
= f_i(T_{[\overline{h}]}^* M, r, s)/q 
 \qquad (i=1,\dots,n).
\]
Hence if $\alpha_{\overline{h},e^{-s}} \geq \alpha$, then
\[
f_i(T_{[\overline{h}]}^* M, r, s) = b_i(M, \alpha,s)r \qquad (i=1,\dots,n).
\]
\end{cor}
\begin{proof}
Before proceeding, we observe 
(as in Proposition~\ref{P:variation1})
that
the Frobenius structure on $M$ for the standard Frobenius lift
induces a Frobenius structure on $T_{[\overline{h}]}^* M$
for the Frobenius lift
\[
x \mapsto (x-[\overline{h}])^q + [\overline{h}]^q.
\]
Consequently, each Jordan-H\"older factor of $T_{[\overline{h}]}^* M$
admits a Frobenius structure for some power of this Frobenius lift.

We then note that first assertion follows 
from Corollary~\ref{C:Frob stability1} plus the off-centered
Frobenius antecedent theorem (Theorem~\ref{T:antecedent2}).
Given this, we can argue that the function
$r_1 \mapsto f_1(T_{[\overline{h}]}^* M,r_1, s) + 
\cdots + f_i(T_{[\overline{h}]}^* M,r_1, s)$
is convex (Lemma~\ref{L:convex}), tends to 0 as $r_1 \to 0^+$
(Lemma~\ref{L:solvable}, applicable by our first observation), 
and agrees with a certain linear function 
at the values $r_1 = r$ and $r_1 = r/q$. It must then agree with that
linear function for all $r_1 \in (0,r]$, yielding the second assertion.
\end{proof}

\begin{lemma} \label{L:type iv stable}
Let $\sigma$ be the standard Frobenius lift.
Suppose $\alpha \in \DD_\ell$, and put
$s_0 = -\log r(\alpha)$. 
Suppose $\rho \in [\epsilon,1)$ satisfies
\[
0 \leq s_0 < \frac{\log p}{(-\log \rho)}.
\]
For $s \in [0,s_0)$, choose $\overline{h} \in \ell$ 
with $\alpha_{\overline{h}, e^{-s}} \geq \alpha$,
and define $|\cdot|_{\rho^\alpha,s} = |T_{[\overline{h}]}(\cdot)|_{\rho^v,s}$.
Then the following hold.
\begin{enumerate}
\item[(a)]
The quantity $|\cdot|_{\rho^\alpha,s}$ does not depend on the choice of
$\overline{h}$.
\item[(b)]
Suppose that $\alpha$ is of type (iv).
Then for fixed $b \in F_\rho\langle x \rangle$, for $s$ sufficiently large,
$|b|_{\rho^\alpha,s}$ is independent of $s$.
\end{enumerate}
\end{lemma}
\begin{proof}
Suppose first that $b \in 
(\gotho_L \llbracket x \rrbracket \otimes_{\gotho_L} L)
\cap F_\rho \langle x \rangle$. Then both (a) and (b) will follow from
Lemma~\ref{L:Frob stability} once we check that 
for $\alpha$ of type (iv),
for any given $P \in \ell \langle x \rangle$, 
$v_s(P)$ is constant for $s$ sufficiently large.
To show this, find $Q \in \ell[x]$
such that $v_0(P-Q) > v_{s_0}(P)$; then $v_s(P) = v_s(Q)$ for all 
$s \in [0, s_0)$, and $v_s(Q)$ is constant for $s$ sufficiently large
because $\alpha$ is of type (iv).

We now allow general $b \in F_\rho \langle x \rangle$;
we may assume $b \neq 0$.
To check (a), we choose $\overline{h_1}, \overline{h_2}$
with $\alpha_{\overline{h_i}, e^{-s}} \geq \alpha$
for $i=1,2$.
We can then find $c \in L_{v,[\rho,\rho]}$ nonzero
and $b' \in 
(\gotho_L \llbracket x \rrbracket \otimes_{\gotho_L} L)
\cap F_\rho \langle x \rangle$ such that 
\[
|b'-cb|_{\rho^v,0} < |c|_{\rho^v} |T_{[\overline{h_i}]}(b)|_{\rho^v,s_0}
\qquad (i \in \{1,2\}).
\]
Recalling that $|\cdot|_{\rho^v,s}$ is a nonincreasing function of $s$,
we deduce
\begin{align*}
|T_{[\overline{h_i}]}(b'-cb)|_{\rho^v,s} 
&\leq |T_{[\overline{h_i}]}(b'-cb)|_{\rho^v,0}\\
&= |b'-cb|_{\rho^v,0} \\
&< |T_{[\overline{h_i}]}(cb)|_{\rho^v,s},
\end{align*}
so
$|T_{[\overline{h_i}]}(b')|_{\rho^v,s} =
|T_{[\overline{h_i}]}(cb)|_{\rho^v,s}$.
We may thus deduce (a) from the corresponding assertion for $b'$.

We may now unambiguously write $|b|_{\rho^\alpha,s}$.
We next verify that $|b|_{\rho^\alpha,s}$ is bounded away from $0$ for
$s \in [0,s_0)$. For this, we may temporarily enlarge $\ell$ to
include an element $\overline{h}$ such that 
$\alpha = \alpha_{[\overline{h}],e^{-s_0}}$;
we then have $|b|_{\rho^\alpha,s} \geq |b|_{\rho^\alpha,s_0} > 0$.

We now repeal our previous enlargement of $\ell$, in order to
check (b).
Suppose that $\alpha$ is of type (iv).
We may then find $c \in L_{v,[\rho,\rho]}$ nonzero
and $b' \in 
(\gotho_L \llbracket x \rrbracket \otimes_{\gotho_L} L)
\cap F_\rho \langle x \rangle$ such that 
\[
|b'-cb|_{\rho^v,0} < |c|_{\rho^v} \inf_{s \in [0,s_0)} \{ 
|b|_{\rho^\alpha,s}\}.
\]
Then 
$|b'-cb|_{\rho^\alpha,s} \leq 
|b'-cb|_{\rho^\alpha,0} = |b'-cb|_{\rho^v,0}
< |cb|_{\rho^\alpha,s}$, so
$|b'|_{\rho^\alpha,s} = |cb|_{\rho^\alpha,s}$
and we may deduce (b) from the corresponding assertion for $b'$.
\end{proof}

\begin{prop} \label{P:stable}
Suppose $\alpha \in \DD_\ell$ is of type (iv), and put $s_0 = -\log r(\alpha)$.
For $i=1,\dots,n$, if $b_i(M,\alpha) > s_0$, then $b_i(M,\alpha,s)$ is constant
for $s$ in some neighborhood of $s_0$.
\end{prop}
\begin{proof}
We proceed on induction by $i$. Assume that
$B = b_i(M, \alpha) - s_0$ is positive, and that
for $j = 1,\dots, i$, $b_j(M, \alpha,s)$ is constant for $s$ in some
neighborhood of $s_0$.
Note that there is no harm in replacing $M$ by $T^*_h M$
(since the latter carries a Frobenius structure, as in the
proof of Corollary~\ref{C:Frob stability2}),
or in rescaling $x$
as long as we remain on the generic path to $\alpha$; this has the
effect of reducing  $s_0$ to some smaller value $s_1$, and
shifting the graph of the function $s \mapsto b_i(M,\alpha,s) - s$ to
the left by $s_0-s_1$. Using such a reparametrization,
we may put ourselves in the situation where
each $b_j(M,\alpha,s)$ is affine-linear (by 
Proposition~\ref{P:variation1}); in particular, 
we have $b_i(M,\alpha,s) = B+s_0 + c(s_0-s)$ for some $c$.
We may also ensure that
\begin{equation} \label{eq:stable1}
ps_0 < b_i(M,\alpha, 0) + |c|s_0 < p(B-2|c|s_0)
\end{equation}
since we can make $s_0$ arbitrarily close to 0 and
$b_i(M,\alpha,0)$ arbitrarily close to $B$.
We may then replace $M$ by $T^*_h M$ if needed to ensure that $\alpha$ does
not belong to the open unit disc. 
We may then choose  $\epsilon$ so that $f_1(M,r,0) < 
\frac{p}{p-1} \log p$ for all $r \in (0, -\log \epsilon]$.

Choose $r \in (0, -\log \epsilon)$ for 
which there exists a nonnegative integer $m$ with
\begin{equation} \label{eq:stable4}
\frac{q^{-m}}{p-1} \log p + 2|c|s_0r < B r, \qquad
b_i(M,\alpha,0)r + |c|s_0r  < \frac{q^{-m} p}{p-1} \log p
\end{equation}
(this is possible thanks to the second inequality in \eqref{eq:stable1}).
From the second inequality in \eqref{eq:stable4}
and the first inequality in \eqref{eq:stable1}, we deduce
\begin{equation} \label{eq:stable3}
q^m r s_0 < \frac{1}{p-1} \log p.
\end{equation}
For any $s \in (0,s_0)$, choose $\overline{h}(s) \in \gotho_\ell$
with $\alpha_{\overline{h}(s),e^{-s}} \geq \alpha$;
by construction, this forces $v(\overline{h}(s)) = 0$.
Then by Corollary~\ref{C:Frob stability2},
\begin{equation} \label{eq:compare f}
f_j(T_{[\overline{h}(s)]}^* M,r,s)
= b_j(M,\alpha,s) r \qquad (j=1,\dots,n).
\end{equation}

By Proposition~\ref{P:variation1},
$b_1(M, \alpha,s) + \cdots + b_i(M, \alpha,s)$ is nonincreasing 
for $s \in (0, s_0)$. However, we have by the induction hypothesis that
$b_1(M, \alpha,s), \dots, b_{i-1}(M, \alpha,s)$ are constant 
(only \emph{a priori} in a neighborhood of $s_0$, 
but we already ensured that each
$b_j(M, \alpha,s)$ is affine-linear). Hence $b_i(M, \alpha,s)$
is nonincreasing; in particular, we have
$b_i(M,\alpha,s) r \leq b_i(M,\alpha,0)r$.
Combining this with \eqref{eq:compare f} and \eqref{eq:stable4}, we deduce
$f_i(T_{[\overline{h}(s)]}^* M,r,s) < \frac{q^{-m} p}{p-1} \log p$.

Let $\sigma$ be the standard Frobenius lift.
By Corollary~\ref{C:descendant} applied repeatedly,
\begin{equation} \label{eq:frob desc}
f_{i+(q^m-1)n}(T_{[\overline{h}(s)]}^* \sigma^m_* M,q^m r,s)
= q^m f_i(T_{[\overline{h}(s)]}^* M,r,s)
\end{equation}
and
\begin{equation} \label{eq:frob desc2}
\sum_{j=1}^{i + (q^m-1)n}
f_{j}(T_{[\overline{h}(s)]}^* \sigma^m_* M,q^m r,s)
= \frac{p(q^m-1)}{p-1} n \log p +
q^m \sum_{j=1}^i  f_j(T_{[\overline{h}(s)]}^* M,r,s).
\end{equation}
Put $\rho = e^{-r q^m}$.
Pick a cyclic vector of $\sigma^m_* M$ and define
the polynomial $P$ as in Theorem~\ref{T:cyclic}.
By Lemma~\ref{L:type iv stable}
(whose hypotheses are satisfied thanks to \eqref{eq:stable3}), for each $j$,
$|P_j|_{\rho^\alpha,s}$ becomes constant for $s$ sufficiently close to
$s_0$. There is thus a terminal shape of the Newton polygon of $P$.

By both inequalities of \eqref{eq:stable4},
\begin{align*}
\frac{q^{-m}}{p-1} \log p + rs &< Br - 2|c|rs_0 + rs \\
&\leq (B + s_0 + c(s_0-s))r \\
&= b_i(M, \alpha,s)r   \\
&= b_i(M, \alpha,0)r - crs \\
&\leq b_i(M, \alpha, 0)r + |c|rs  \\
&< \frac{q^{-m}p}{p-1} \log p + rs.
\end{align*}
By \eqref{eq:compare f} and \eqref{eq:frob desc},
\[
b_i(M, \alpha,s)r = f_i(T_{[\overline{h}(s)]}^* M,r,s)
= q^{-m} f_{i+(q^m-1)n}(T_{[\overline{h}(s)]}^* \sigma^m_* M,q^m r,s).
\]
We may thus deduce
\begin{equation} \label{eq:trap push}
f_{i+(q^m-1)n}(T_{[\overline{h}(s)]}^* \sigma^m_* M,q^m r,s)
\in \left(\frac{1}{p-1} \log p + q^m rs, \frac{p}{p-1} \log p + q^m rs\right).
\end{equation}
Thanks to \eqref{eq:trap push},
Theorem~\ref{T:cyclic} allows us to read off the values
of $f_{j}(T_{[\overline{h}(s)]}^* \sigma^m_* M,q^m r,s)$ for
$j=1,\dots,i+(q^m-1)n$ from the Newton polygon of $P$ 
under $|\cdot|_{\rho^\alpha,s}$.
By \eqref{eq:frob desc2}, we may deduce that
$b_1(M,\alpha,s) + \cdots + b_i(M, \alpha,s)$ is constant
for $s$ sufficiently large; by the induction hypothesis, this implies
that $b_i(M, \alpha,s)$ is constant for $s$ sufficiently large, as desired.
\end{proof}
\begin{remark}
In trying to understand the proof of Proposition~\ref{P:stable}, it may
help to think of $B$ as playing the role of an absolute measure, which dictates
the choice of the other auxiliary parameters $r,m$
in order to put the situation into the range where
cyclic vectors and Theorem~\ref{T:cyclic} become useful. This crucial role
means that we do not have a similar argument in case $B=0$.
\end{remark}

\begin{cor} \label{C:at most 1}
Suppose $\alpha \in \DD_\ell$ is minimal under domination (i.e.,
 is of type (i) or (iv)).
For $i=1,\dots,n$, the terminal slope of $b_i(M,\alpha,s)$ is at most $1$.
\end{cor}
\begin{proof}
Put $s_0 = -\log r(\alpha)$.
If $\alpha$ is of type (i), then all the terminal slopes are equal to 1 by
Remark~\ref{R:radius 0}, so we assume hereafter that $\alpha$ is of type (iv).
If $b_i(M,\alpha) > s_0$, then Proposition~\ref{P:stable} implies that
the terminal slope of $b_i(M,\alpha,s)$ is zero.
Otherwise, for  $s$ slightly less than $s_0$, the function $b_i(M,\alpha,s)-s$
is nonnegative and affine-linear with limit $0$ as $s \to s_0^-$, so its 
slope must be nonpositive. Hence the terminal slope of $b_i(M, \alpha,s)$ is at most
$1$.
\end{proof}

\begin{remark}
Note that we have not ruled out the possibility that the terminal slope is in
the range $(0,1)$, in the case where $\alpha$ is of type (iv) and
$b_i(M, \alpha) = s_0$.
However, in case the terminal slope is known to be an integer,
it is then forced to either equal 1 (in which case $b_i(M,\alpha,s) = s$ for
$s$ in a neighborhood of $s_0$) or to be nonpositive.
\end{remark}

\subsection{Dwork modules}
\label{subsec:Dwork}

Throughout this subsection, retain Hypothesis~\ref{H:variation}. 
We now make a more careful study of the limiting behavior of $b_1(M, \alpha,s)$
when $M$ is a Dwork module.

\begin{lemma} \label{L:union discs}
For any $P_1, \dots, P_m \in \ell[x]$ and any $r_1,\dots,r_m \geq 0$, the
subset
\[
\{\beta \in \DD_\ell: |P_i|_\beta \leq r_i \quad (i=1,\dots,m)\}
\]
is a (possibly empty) finite union of discs, each of which contains an
$\ell$-rational point. In particular, if this set is nonempty, then it contains
an $\ell$-rational point.
\end{lemma}
\begin{proof}
Note that the intersection of 
two discs $D_1, D_2$ is either $D_1$, $D_2$, or the empty set.
It thus suffices to check the case $i=1$, which we treat as follows.

Factor $P_1 = c (x- z_1) \cdots(x-z_l)$. If $l = 0$, then the subset is
either empty or all of $\DD_\ell$. Otherwise, we proceed by induction on $l$.
Put $r = \max_{i \neq j} \{|z_i - z_j|\}$.
If $r_1 \geq |c| r^l$, then the desired region is precisely the disc
with center $z_1$ and radius $(r_1/|c|)^{1/l}$. (This includes the case
where the $z_i$ are all equal, as then $r=0$.) Otherwise,
the region is contained in the union of the open discs of radius $r$
with centers $z_1,\dots,z_l$. The disc containing $z_i$ fails to contain some
$z_j$, so on that disc $|x-z_j|$ is identically equal to $|z_i-z_j|$.
Hence we reduce to $l$ separate problems involving polynomials of degree
less than $l$, so the induction hypothesis finishes the job.
\end{proof}

\begin{prop} \label{P:identically zero}
Suppose that $\alpha$ is of type (iv).
Suppose that $M \in C_{v,[\epsilon,1)}$ is the 
free module of rank $1$ generated by
$\bv$ satisfying
\[
\nabla(\bv) = \sum_{i \not\equiv 0 \pmod{p}} 
\pi i \tilde{u}_i x^{i-1} \bv \otimes dx
\]
for some $\tilde{u}_i \in \gotho_L$, all but finitely many of which are zero,
and some $\pi \in K$ satisfying $\pi^{p-1} = -p$.
Assume also that $b_1(M, \alpha,s)$ is affine-linear on $[0, s_0]$.
Then the slope of $b_1(M, \alpha,s)$ equals $0$ unless 
$b_1(M, \alpha,s) - s$ is identically zero.
\end{prop}
Note that the existence of a Frobenius structure on $M$ is a consequence of
the fact that the series $\exp(\pi x - \pi x^q)$ converges in a disc of radius 
strictly greater than $1$. Also note that if $b_1(M, \alpha) > s_0$, then
the claim follows from Proposition~\ref{P:stable}, although we will not use
this in the proof.

\begin{proof}
We may assume that
$b_1(M, \alpha,s)-s$ is not identically zero, so that
by Proposition~\ref{P:variation1}, 
the slope of $b_1(M, \alpha,s)$ is some nonpositive integer.
By Lemma~\ref{L:dwork},
\[
b_1(M, s) = \max\{s, \max_{i \neq 0} \{ -v(u_i) + (1-i)s\}\}
\]
where $u_i \in \ell$ is the reduction of $\tilde{u}_i$.
Put $u = \sum_i u_i x^i$.

For $P(x) = \sum_{i=0}^{m} P_i x^i \in \ell[x]$, define
\[
AS(P)(x) = \sum_{j \geq 0, j \not\equiv 0 \pmod{p}} 
x^j \sum_{i=0}^\infty P_{j p^i}^{1/p^i}.
\]
For $s \in [0,s_0)$, choose $\overline{h} \in \ell$ with
$\alpha_{\overline{h}, e^{-s}} \geq \alpha$.
Let $T_{\overline{h}}: \ell[x] \to \ell[x]$ denote the substitution
$x \mapsto x + \overline{h}$.
Let $u'_i$ be the coefficient of $x^i$ in $AS(T_{\overline{h}}(u))$. Then
\[
b_1(M, \alpha, s) = \max\{s, \max_i \{ -v(u'_i) + (1-i)s\}\}.
\]
Let $1-i_0 \leq 0$ be the slope of $b_1(M, \alpha,s)$; then the term
$i=i_0$ must dominate everywhere.

Now choose $\overline{h}$ in a complete extension $\ell'$ of $\ell$ so 
that in fact
$\alpha = \alpha_{\overline{h}, e^{-s_0}}$. Then
\begin{equation} \label{eq:dominate}
-v(u'_{i}) + (1-i) s \leq b_1(M, \alpha, s) \qquad (s \in \{0, s_0\}).
\end{equation}

We can choose a nonnegative integer $l$ so that
each $u'_i$ is a polynomial in $\overline{h}^{1/p^l}$.
That is, we may write $u'_i = \sum_j u'_{ij} \overline{h}^{j/p^l}$
with $u'_{ij} \in \ell$. Define
\[
U'_i = \sum_j u'_{ij} x^j \in \ell[x].
\]
We then define a subset $S$ of $\DD_\ell$ by putting
\[
S = \{\beta \in \DD_\ell: \log |U'_i|_\beta + (1-i) s \leq 
b_1(M, \alpha, s) \quad (s \in \{0, s_0\}; i \geq 0)\}.
\]
The set $S$ is nonempty
because it contains $\alpha_0 = \alpha_{\overline{h}^{1/p^l},0}|_{\ell[x]}$.
Thus by
Lemma~\ref{L:union discs}, $S$ must
be a nonempty union of discs each containing an $\ell$-rational point.

Consequently, we can find $(\overline{h}')^{1/p^l} \in \ell$ in the intersection of 
$S$ with the open unit disc containing $\alpha_0$. 
Put $\alpha' = \alpha_{\overline{h}', e^{-s_0}}$.
Since $\overline{h}'$ lies in the open unit disc containing
$\overline{h}$, and since $\alpha$ is of type (iv),
there is a greatest value $s_1 \in (0, s_0)$ where
$\alpha_{\overline{h}', e^{-s_1}} \geq \alpha$.
We now have that 
$b_1(M, \alpha,s) = b_1(M, \alpha',s)$ 
for $s \in [0, s_1]$, but
$b_1(M, \alpha,s) \geq b_1(M, \alpha', s)$ for $s \in [0, s_0]$
by \eqref{eq:dominate}. Since
the function $s \mapsto b_1(M, \alpha, s)$ is affine by hypothesis
while the function $s \mapsto b_1(M, \alpha',s)$ is convex by 
Proposition~\ref{P:variation1}, we conclude that
$b_1(M, \alpha, s) = b_1(M, \alpha', s)$ for $s \in [0, s_0]$.

By Theorem~\ref{T:subharmonicity} applied after a rescaling, the left
slope of $b_1(M, \alpha,s)$ at $s=s_1$ must be less than or equal to
the sum of the right slopes of $b_1(M, \alpha,s)$ and $b_1(M, \alpha',s)$
at $s = s_1$. Consequently, $(1-i_0) \leq 2(1-i_0)$, whence
$1-i_0 \geq 0$. Since $1 - i_0 \leq 0$ from earlier,
the only possibility is $i_0 = 1$,
so the slope of $b_1(M, \alpha,s)$ is $0$, as desired.
\end{proof}

\subsection{Shrinking the domain of definition}
\label{subsec:shrinking}

We will sometimes need a variation on the construction of the 
category $C_{v,I}$,
corresponding to working on an annulus rather than a disc in 
the $x$-direction.
We no longer retain Hypothesis~\ref{H:variation}.

\begin{defn}
For $I \subseteq (0,1)$ an interval and $J \subseteq [0, +\infty]$ another 
interval, let $R_{v,I,J}$ be the Fr\'echet completion of $L_{v,I}[x]$
(if $+\infty \in J$) or $L_{v,I}[x,x^{-1}]$ (if $+\infty \notin J$)
for the norms $|\cdot|_{\rho^v,s}$ for $\rho \in I$, $s \in J$.
\end{defn}

\begin{lemma} \label{L:intersect}
Let $J \subseteq J' \subseteq [0, +\infty]$ be intervals.
Suppose $f \in R_{v,I,J}$ has the property that for each $\rho \in I$,
the image of $f$ in $\Gamma(A_{F_\rho}(\rho^J), \calO)$ lifts to
$\Gamma(A_{F_\rho}(\rho^{J'}), \calO)$. Then $f$ lifts to
$R_{v,I,J'}$.
\end{lemma}
\begin{proof}
In $\Gamma(A_{F_\rho}(\rho^J),\calO)$, we can write
$f = \sum_{i \in \ZZ} f_i x^i$ with $f_i \in F_\rho$.
In order for $f$ to lift to $R_{v,I,J}$, we must have
$f_i \in L_{v,I}$ for each $i$; in order for $f$ to lift to
$\Gamma(A_{F_\rho}(\rho^{J'}),\calO)$, for each $s \in J'$ we must have
$|f_i|_{\rho^v} \rho^{si} \to 0$ as $i \to \pm \infty$. These conditions
together imply $f \in R_{v,I,J'}$.
\end{proof}

\begin{defn}
Let $I_1 \subseteq I_2 \subseteq \cdots$ be an increasing sequence of
closed intervals with union $I$,
and let $J_1 \subseteq J_2 \subseteq \cdots$ be an increasing sequence
of closed intervals with union $J$. (If $+\infty \in J$, we require
$+\infty \in J_i$ for all $i$.)
For this definition only,
put $R_{i} = R_{v,I_i,J_i}$.
Let $C_{v,I,J}$ be the category whose
elements are sequences $M_1,M_2,\dots$ with $M_i$ a finite locally free
$\nabla$-module over $R_{i}$, 
equipped with isomorphisms $M_{i+1} \otimes_{R_{i+1}} R_i \cong M_i$.
Note that this category is canonically independent of the choice of the
$I_i$ and $J_i$, and that $C_{v,I,[0,+\infty]} = C_{v,I}$.
Let $C_{v,*,J}$ be the direct limit of $C_{v,[\epsilon,1),J}$ as
$\epsilon \to 1^-$.
\end{defn}

\begin{defn}
For $M \in C_{v,I,J}$ and $\rho \in I$, we may
base extend $M$ to obtain a $\nabla$-module $M_\rho$
on the annulus $A_{F_\rho}(\rho^J)$.
As in Definition~\ref{D:cross section},
we call $M_\rho$ the \emph{cross section} of $M$
at $\rho$. For $s \in I$, put $R(M, \rho,s) = R((M_\rho)_{\rho^s})$ and
$f_i(M, r,s) = f_i(M_{e^{-r}}, rs)$
according to Definition~\ref{D:subsidiary}. 
\end{defn}

\begin{lemma} \label{L:intersect2}
Let $J \subseteq J' \subseteq [0, +\infty]$ be intervals.
Suppose $M \in C_{v,I,J'}$ and that $\bv$ is a section of the restriction
of $M$ to $C_{v,I,J}$, such that for each $\rho \in I$, the image
of $\bv$ in $M_\rho$ lifts from $A_{F_\rho}(\rho^J)$ to
$A_{F_\rho}(\rho^{J'})$. Then $\bv$ is a section of $M$ itself.
\end{lemma}
\begin{proof}
This follows from Lemma~\ref{L:intersect} using 
\cite[Lemma~2.3.1]{kedlaya-xiao}.
\end{proof}

As in Lemma~\ref{L:convex}, we obtain the following.
\begin{lemma}
Take $M \in C_{v,I,J}$ of rank $n$.
For $i=1,\dots,n$, for $s \in J \setminus \{+\infty\}$ fixed, the function
\[
r \mapsto f_1(M,r,s) + \cdots + f_i(M,r,s)
\]
is convex. 
\end{lemma}

\begin{defn}
Let $\sigma: R_{v,[\epsilon,1)} \to R_{v,[\epsilon^{1/q},1)}$ 
be a Frobenius lift. For $M \in C_{v,[\epsilon,1), J}$,
a \emph{Frobenius structure} on $M$ with respect to $\sigma$
is an isomorphism $F: \sigma^* M \cong M$ of objects 
in $C_{v,[\epsilon^{1/q},1), J}$.
\end{defn}

As in Corollary~\ref{C:constant breaks2}, we obtain the following.
\begin{lemma} \label{L:constant breaks3}
Suppose $J$ is closed, and $M \in C_{v,[\epsilon,1),J}$ is of rank $n$
and admits a Frobenius structure. Then there exist $r_0 \in (0, -\log 
\epsilon]$ and constants $b_i(M,s)$ for $i=1,\dots,n$ such that for
$r \in (0, r_0]$ and $s \in J \setminus \{+\infty\}$,
 $f_i(M,r,s) = b_i(M,s)r$ for $i=1,\dots,n$.
\end{lemma}

As in Proposition~\ref{P:variation1}, we obtain the following.
\begin{lemma} \label{L:convex2}
For $M \in C_{v,[\epsilon,1),J}$ of rank $n$, the functions
$b_i(M,s)$ for $i=1,\dots,n$ have the following properties.
\begin{enumerate}
\item[(a)]
Each $b_i(M,s)$ is continuous and piecewise affine-linear on $J 
\setminus \{+\infty\}$ with slopes
in $\frac{1}{n!}\ZZ$.
\item[(b)]
The function $b_1(M,s) + \cdots + b_i(M,s)$ is convex.
(Beware that we have no nonincreasing assertion here.)
\end{enumerate}
\end{lemma}

\begin{defn}
Assume Hypothesis~\ref{H:variation}.
Suppose $\alpha \in \DD_{\ell}$
and $J \subseteq [0, -\log r(\alpha)]$. If $\alpha$ is of type (iv),
assume further that $J \subseteq [0, s]$ for some $s< -\log r(\alpha)$.
We may then choose $\overline{h} \in \gotho_\ell$ such that
$\alpha_{\overline{h}, e^{-s}} \geq \alpha$ for each $s \in J$.
Choose $h \in L_{v,\epsilon}$ with $|h|_L \leq 1$
lifting $\overline{h}$.
Define the category
$C_{v,*,J,h}$ to be a copy of $C_{v,*,J}$.

Choose now $\overline{h}' \in \gotho_\ell$ such that
$\alpha_{\overline{h}', e^{-s}} \geq \alpha$ for each $s \in J$,
and $h' \in L_{v,\epsilon}$ with $|h'|_L \leq 1$
lifting $\overline{h}'$.
If $\overline{h} \neq \overline{h}'$,
for $\rho \in (0,1)$
sufficiently close to 1, we have 
$|h-h'|_{\rho^v} = \rho^{v(\overline{h}-\overline{h'})}$.
If $\overline{h} = \overline{h}'$, then
as $\rho \to 1^-$, $|h-h'|_{\rho^v}$ tends to a limit strictly less than 1.
For $s \in J$, we have $\alpha_{\overline{h},e^{-s}} \geq \alpha$
and $\alpha_{\overline{h}',e^{-s}} \geq \alpha$, so
$\alpha_{\overline{h},e^{-s}} = \alpha_{\overline{h}',e^{-s}}$
by Lemma~\ref{L:unique dominate}.
That is, $v(\overline{h} - \overline{h'}) \geq s$,
so for $\rho \in (0,1)$
sufficiently close to 1 (uniformly in $s$), 
we have $\rho^s \geq |h-h'|_{\rho^v}$.
For such $\rho$, the norm $|\cdot|_{\rho^v,s}$ is invariant
under $T_{h-h'}$ as in Definition~\ref{D:maintain by translate},
so $T_{h'-h}$ induces an endomorphism of $R_{v,[\epsilon,1),J}$
for $\epsilon \in (0,1)$ sufficiently close to 1.

We thus obtain an isomorphism
$T_{h'-h}^*: C_{v,*,J,h} \to C_{v,*,J,h'}$;
since these isomorphisms satisfy the cocycle condition,
we may use them to identify $C_{v,*,J,h}$ and $C_{v,*,J,h'}$
with a single category $C_{v,*,J,\alpha}$.
Note that the functions $b_i(\cdot,s)$ for $s \in J$
commute with these isomorphisms thanks to Lemma~\ref{L:maintain by translate};
thus they induce well-defined functions $b_i(\cdot,\alpha,s)$ on $C_{v,*,J,\alpha}$
(in a manner consistent with Definition~\ref{D:breaks}).

For $J' \subseteq J$, we have restriction functors
$C_{v,*,J,\alpha} \to C_{v,*,J',\alpha}$.
Using these, if $\alpha$ is of type (i) or (iv), we
define $C_{v,*,*,\alpha}$ to be the inverse limit of $C_{v,*,J,\alpha}$
over increasing intervals $J \subseteq [0, -\log r(\alpha))$.
Again, the function $b_i(M, \alpha,s)$ is
well-defined for $s \in [0, -\log r(\alpha))$.

For $s \in [0, -\log r(\alpha))$,
define $C_{v,*,s,\alpha}$ to be the inverse limit of $C_{v,*,J,\alpha}$
over increasing intervals $J \subseteq [s, -\log r(\alpha))$.
Then define $C_{v,*,**,\alpha}$ to be the direct limit of
$C_{v,*,s,\alpha}$ over $s \in [0, -\log r(\alpha))$.
\end{defn}

\begin{remark} \label{R:gen fib intersect}
We will use later the fact that if $\rho \in I$ and 
$0 \in J$, then within $L'_{v,[\rho,\rho]}$,
\[
L'_{v,I} \cap A_{F_\rho}(\rho^J) = R_{v,I,J}.
\]
This is most easily seen by representing
elements of $L'_{v,[\rho,\rho]}$ as Laurent series in $x$ with coefficients
in $L_{v,[\rho,\rho]}$.
\end{remark}

\section{Induction on transcendence defect}
\label{sec:corank}

In this section, we carry out an induction on transcendence defect 
in order to prove
local semistable reduction. Specifically, we prove this result, which
completes the proofs of local (Theorem~\ref{T:local semi})
and global (Theorem~\ref{T:global semi}) semistable reduction.
(See \S~\ref{subsec:completion} for the overall structure of the proof
of Theorem~\ref{T:induct local semi}.)

\setcounter{theorem}{0}
\begin{theorem} \label{T:induct local semi}
Assume that $k$ is algebraically closed.
Suppose that for some nonnegative integer  $n$,
local semistable reduction (Definition~\ref{D:local semi})
holds whenever the valuation $v$ is minimal of transcendence defect at most $n$.
Then local semistable reduction also holds when $v$ is minimal of transcendence
defect
$n+1$.
\end{theorem}

\begin{remark}\label{R:Galois descent}
Note that in the course of showing Theorem~\ref{T:induct local semi},
it will be convenient to replace $K$ by a finite extension. This
is harmless, because log-extendability of an $F$-isocrystal can be checked 
after a finite extension of the constant field, e.g., by Galois descent. 
(In fact, log-extendability can even
be checked after an arbitrary complete extension of the constant field,
but this requires a more robust argument using linear compactness.)
\end{remark}

\begin{notation}
We write $\cEnd(*)$ as shorthand for $*^\dual \otimes *$.
\end{notation}

\subsection{Setup: relative geometry}

We first do a little geometry over $k$ in order to set up a framework, using
fibrations in curves, in which we will be able to analyze a valuation of
positive transcendence defect.

\begin{hypothesis} \label{H:relative}
Throughout this section, assume that $k$ is algebraically closed.
Fix a nonnegative integer $n$, and assume that
local semistable reduction
holds whenever the
valuation $v$ is minimal of transcendence defect at most $n$.
Let $X$ be a smooth irreducible $k$-variety, let $Z$ be a closed
subvariety of $X$, let $v$ be a minimal valuation on $k(X)$ of transcendence
defect
$n+1$ centered on $X$, and let $\calE$ be an $F$-isocrystal on $X \setminus Z$
overconvergent along $Z$. (Note that by Abhyankar's inequality,
we must have $\dim(X) \geq \trdefect(v)+1 = n+2 \geq 2$.)
\end{hypothesis}

To prove Theorem~\ref{T:induct local semi}, it is enough to show that
under Hypothesis~\ref{H:relative}, $\calE$ admits local semistable reduction
at $v$. To do this, we first make some additional geometric arrangements.
\begin{lemma} \label{L:get relative}
Under Hypothesis~\ref{H:relative}, there exists a local alteration
$f: \tilde{X} \to X$ around $v$, a closed subscheme $\tilde{Z}$ of
$\tilde{X}$ containing $f^{-1}(Z)$, and an extension $\tilde{v}$ of
$v$ to $k(\tilde{X})$, such that if we replace $X,Z,v$ by
$\tilde{X}, \tilde{Z}, \tilde{v}$, then
Hypothesis~\ref{H:relative2} below can be satisfied.
\end{lemma}
Before proving this, we introduce the geometric hypothesis we are trying
to fulfill, and make some remarks about it.

\begin{hypothesis} \label{H:relative2}
For the remainder of this section (except the proof of 
Lemma~\ref{L:get relative}, and Remark~\ref{R:fixing base change}),
suppose that there exist $X^0, Z^0, \pi, v^0, z^0, z, x$ as follows.
\begin{itemize}
\item
The $k$-variety $X^0$ is smooth irreducible affine, and $Z^0$ is a closed
subvariety of $X^0$.
\item
The morphism $\pi: X \to X^0$ is smooth of relative dimension 1 with 
geometrically integral generic fibre, and satisfies $Z = \pi^{-1}(Z^0)$.
\item
The restriction $v^0$ of $v$ to $k(X^0)$ has transcendence defect $n$ and is 
exposed by the pair $(X^0, Z^0)$.
Note that this implies that $\ratrank(v) = \ratrank(v^0)$, and that
$v$ is exposed by $(X, Z)$: we obtain a good system of parameters
by pulling back a good system of parameters at the center of $v^0$, then 
adding one more parameter to complete a system of local coordinates.
\item
The point $z^0 \in X^0$ is the center of $v^0$, and the point $z \in X$
is the center of $v$.
\item
The fibre $\pi^{-1}(z^0)$ is an affine line,
and $x \in \Gamma(X, \calO)$ restricts to a coordinate of that line.
\end{itemize}
\end{hypothesis}

\begin{remark} \label{R:same completion}
Under Hypothesis~\ref{H:relative2}, 
$k(X)$ is a finite extension of $k(X^0)(x)$.
Moreover, the local ring $\calO_{X,z}$ is contained in the completion
of $k(X^0)[x]$ for the 0-Gauss valuation with respect to $v^0$.
Consequently, $k(X)$ is contained in the
completion of $k(X^0)(x)$ under any extension
of $v^0$ which on $k(X^0)[x]$ is greater than or equal to the Gauss valuation.
In particular, $v$ is uniquely determined by its restriction to $k(X^0)(x)$.
\end{remark}

\begin{defn} \label{D:restrict}
 Let $P^0$ be a smooth irreducible affine formal scheme over $\Spf \gotho_K$
with $P^0_k \cong X^0$. 
Let $\sigma^0$ be a $q$-power Frobenius lift on $P^0$.
We may restrict $\calE$ to an $(F, \nabla)$-module on a space of the form
$V \times A_{K,x}[0,1]$, where $V$ is the intersection of $]z^0[_{P^0_K}$
with a strict neighborhood of $]X^0 \setminus Z^0[_{P^0}$ in $P^0_K$.
\end{defn}

Once Hypothesis~\ref{H:relative2} has been enforced, we will proceed
to modify the geometric situation while preserving Hypothesis~\ref{H:relative2}.
Here are the operations we will use to accomplish this.

\begin{defn} \label{D:relative base change}
By a \emph{base change on Hypothesis~\ref{H:relative2}}, we will mean
an operation consisting of the following steps.
\begin{itemize}
\item
Replace $X^0$ by $\tilde{X}^0$ for some local alteration 
$f: \tilde{X}^0 \to X^0$ around $v^0$,
and replace $v^0$ by an extension $\tilde{v}^0$ to $k(\tilde{X}^0)$,
such that if we put $\tilde{Z}^0 = f^{-1}(Z^0)$,
then $\tilde{v}^0$ is exposed by $(\tilde{X}^0, \tilde{Z}^0)$.
\item
Replace $X$ by the Zariski closure $\tilde{X}$ of the graph of the 
rational map
$X \times_{X^0} \tilde{X}^0 \dashrightarrow \AAA^1_k$
given by $(x-h)/g$ for some 
$g \in \Gamma(\tilde{X}^0, \calO) 
\cap \Gamma(\tilde{X}^0 \setminus \tilde{Z}^0,
\calO^*)$ and $h \in \Gamma(\tilde{X}^0, \calO)$
such that for some extension $\tilde{v}$ of
$v$ to $k(\tilde{X})$ extending $v$ on $k(X)$ and $\tilde{v}^0$
on $k(\tilde{X}^0)$, we have $\tilde{v}(x-h) \geq \tilde{v}(g)$.
Let $\tilde{x} = (x-h)/g
\in \Gamma(\tilde{X}, \calO)$.
\end{itemize}
In case $g=1$, we call the operation
a \emph{scale-preserving base change}.
In case $h=0$, we call the operation a
\emph{center-preserving base change}.
If both of these hold (so the only effect is to change
$X^0$), we call the operation a \emph{simple base change}.
(See Remark~\ref{R:compatibility} for details on how we choose
the extended valuations.)
\end{defn}

\begin{remark} \label{R:fixing base change}
Note that it also makes sense to carry out the base change described
in Definition~\ref{D:relative base change} even if the fibre 
$\pi^{-1}(z_0)$ is not an affine line, 
as long as $x$ restricts to a local parameter for this fibre
at $z$. As long as $0 < \tilde{v}^0(g) 
\leq \tilde{v}(x-h)$, the result will then satisfy
Hypothesis~\ref{H:relative2}. (Such $g$ exists because
Hypothesis~\ref{H:relative} forces $\dim(X) \geq 2$ and hence
$\dim(X^0) \geq 1$.)
\end{remark}

\begin{defn}
By a \emph{tame alteration on Hypothesis~\ref{H:relative2}},
we will mean an operation consisting of the following steps.
\begin{itemize}
\item
Perform a simple base change after which the zero locus of $x$ on $X$
is smooth.
\item
Replace $X$ by the fibre product $\tilde{X}$
of the map $x: X \to \AAA^1_k$ with the 
$m$-th power map on $\AAA^1_k$, for some positive integer $m$ coprime to $p$.
Let $\tilde{x}$ be the coordinate on the source of the $m$-th power map,
so that we have the equality $x = \tilde{x}^m$ of maps $\tilde{X} \to \AAA^1$.
Then shrink if necessary to get something smooth on which an extension of $v$
is centered.
\end{itemize}
\end{defn}

\begin{defn}
By an \emph{Artin-Schreier alteration} on Hypothesis~\ref{H:relative2}, we will
mean an operation consisting of the following steps.
\begin{itemize}
\item
Perform a simple base change after which the zero locus of $x$ on $X$
is smooth.
\item
Replace $X$ by the finite flat cover $\tilde{X}$ with structure sheaf 
\[
\calO[\tilde{x}]/(\tilde{x}^p - g^{p-1} \tilde{x} - x - a_2 x^2 - \cdots - a_k x^k)
\]
for some $g, a_2, \dots, a_k \in \Gamma(X^0, \calO)$
with $a_2, \dots, a_k$ in the maximal ideal of $\calO_{X^0,z^0}$.
Then shrink if necessary to get something smooth on which an extension of $v$
is centered.
\end{itemize}
\end{defn}

To conclude the geometric discussion, 
we show that Hypothesis~\ref{H:relative2} can be realized,
by proving Lemma~\ref{L:get relative}.
\begin{proof}[Proof of Lemma~\ref{L:get relative}]
By Theorem~\ref{T:de jong}, we may reduce to the case where $(X,Z)$ is a smooth
pair. By Lemma~\ref{L:exposing alteration}
(at the possible expense of enlarging $Z$),
we may assume that $v$ is exposed by $(X,Z)$. Choose local parameters
$t_1,\dots,t_r$ for the components of $Z$ passing through the center $z$ of
$v$. Then extend this to a sequence of local parameters $t_1,\dots,t_r,
x_1,\dots,x_n,x$ for $X$ at $z$. Use these to construct 
a map $X \to \AAA^{r+n+1}_k$; by shrinking $X$, we can ensure
that this map is \'etale. 

By dropping the last coordinate, we get a smooth
morphism $\pi: X \to X^0 = \AAA^{r+n}_k$ of relative dimension 1. 
Let $Z^0$ be the zero locus of $t_1 \cdots t_r$ on $X^0$.
Let $v^0$ be the restriction of $v$ to $k(X^0)$, whose center $z^0$ 
lies on $Z^0$.

The morphism $\pi$ may not have geometrically
integral generic fibre. We fix this by replacing $X^0$ with
an exposing alteration $X^1$ for $v^0$, 
and $X$ with a connected component of $X \times_{X^0} X^1$.

We have now achieved Hypothesis~\ref{H:relative2} except that the fibre
of $X \to X^0$ over $z^0$ is not yet an affine line. However, 
as noted in Remark~\ref{R:fixing base change}, we may
perform a base change in the sense of Definition~\ref{D:relative base change}
to rectify this.
\end{proof}

\subsection{Setup: differential modules}

We now set up the differential modules we will use to carry out
the induction on transcendence defect, and sketch the proof of Theorem~\ref{T:induct local semi}.

\begin{defn} \label{D:regime}
We now enter the notational r\'egime of \S~\ref{sec:relative}.
Let $L_0$ be the completion of the direct system
\[
\Frac \Gamma(P^0_K, \calO) \stackrel{\sigma^0}{\to} 
\Frac \Gamma(P^0_K, \calO) \stackrel{\sigma^0}{\to} \cdots,
\]
and equip the residue field $\ell_0 = k(X^0)^{\perf}$ of $L_0$ with the unique
valuation extending $v^0$. 
Choose an extension of $v^0$ to $k(X^0)^{\alg}$ (again denoted $v^0$);
as in Remark~\ref{R:extension}, this extension preserves
height,
rational rank, residual transcendence degree, and transcendence defect.
Let $\ell$ be the completion of $k(X^0)^{\alg}$ under $v^0$
(which is again algebraically closed).
Put $L = L_0 \otimes_{W(\ell_0)} W(\ell)$.

We then obtain from $\calE$ an element
$M \in C_{v^0,[\epsilon,1)}$ for some $\epsilon \in (0,1)$,
admitting a Frobenius structure for the standard Frobenius lift $\sigma$.
More precisely, with notation as in 
Definition~\ref{D:restrict},
we have a continuous homomorphism $\Gamma(V \times A_{K,x}[0,1], \calO)
\to R_{v^0,[\epsilon,1)}$ for some $\epsilon \in (0,1)$.
The module $M$ is the base change of $\calE$ along this homomorphism.

By Lemma~\ref{L:base change1}, we can find a multiplicative seminorm
$\alpha$ on $\ell[x]$
which agrees with $e^{-v^0(\cdot)}$ on $\ell$ and with $e^{-v(\cdot)}$ 
on $k(X^0)[x]$. 
Then $\alpha$ corresponds to a point of $\DD_\ell$,
so in particular we 
may extend $\alpha$ to the completion $\ell \langle x \rangle$ of
$\ell[x]$ for the 0-Gauss valuation with respect to $v^0$.
Let $h(\alpha)$ denote the completion of
$\Frac (\ell \langle x \rangle/\ker(\alpha))$ for the valuation
induced by $\alpha$; note that $h(\alpha)$ contains $\ell(x)$
unless $\alpha$ is of type (i).
As in Remark~\ref{R:same completion}, the local ring
$\calO_{X,z}$ is contained in $\ell \langle x \rangle$;
therefore, at least when $\alpha$ is not of type (i),
$k(X)$ is a subfield of $h(\alpha)$, and the induced
valuation is precisely $v$. The same turns out to be true when
$\alpha$ is of type (i); see Remark~\ref{R:compatibility}.
\end{defn}

\begin{lemma}
In Definition~\ref{D:regime}, the point $\alpha \in \DD_\ell$ is of type (i)
or (iv).
\end{lemma}
\begin{proof}
Suppose to the contrary that $\alpha$ is of type (ii) or (iii).
Then $h(\alpha)$ is the completion
of an algebraic extension of $k(X)$ (namely the compositum with $\ell$
over $k$). 
By Remark~\ref{R:extension}, the valuation on $h(\alpha)$
has the same rational rank and residual transcendence degree as $v$.
However, by Lemma~\ref{L:same corank}, this implies that 
the sum of the rational rank and residual transcendence degree
on $\ell$ is one less than on $h(\alpha)$, implying
$\trdefect(v^0) = \trdefect(v)$. This contradicts
Hypothesis~\ref{H:relative2}.
\end{proof}

\begin{defn}
For $s_0 = -\log r(\alpha)$,
define $b_i(M, \alpha,s)$ for $s \in [0, s_0]$.
As observed in Definition~\ref{D:terminal slope}, the function $b_i(M, \alpha,s)$
is affine-linear on some neighborhood of $s_0$. We call the slope of
$b_i(M,\alpha,s)$ in this range the \emph{terminal slope} of $b_i(M, \alpha,s)$.
\end{defn}

\begin{remark} \label{R:compatibility}
From now on, when we perform a base change on Hypothesis~\ref{H:relative2},
we will need to maintain compatibility with Definition~\ref{D:regime}.
(We also need compatibility with tame and Artin-Schreier
alterations, but that is straightforward.)

To arrange the compatibility, in Definition~\ref{D:relative base change},
we will always identify $k(\tilde{X}^0)$ within $k(X^0)^{\alg}$ so that
$\tilde{v}^0$ agrees with the restriction of our chosen extension
of $v^0$. 
We then note that the map $\ell \langle x \rangle \to h(\alpha)$
cannot fail to be injective on $k(\tilde{X}^0)[x]$:
if $P(x)$ were a nonzero element of the kernel, then we could
find a polynomial $Q(x) \in k(X^0)[x]$ which becomes
divisible by $P(x)$ in $k(\tilde{X}^0)[x]$, and $Q(x)$ would also
be in the kernel. However, the map must be injective on
$k(X^0)[x]$ since $\alpha$ restricts to $e^{-v(\cdot)}$ there.
(Note that this argument implies that $k(X)$ is a subfield of $h(\alpha)$
even when $\alpha$ is of type (i), as noted in Definition~\ref{D:regime}.)

Since $\pi$ has geometrically integral generic fibre, we may identify
\[
k(\tilde{X}) =
k(X \times_{X^0} \tilde{X}^0) = k(X) \otimes_{k(X^0)} k(\tilde{X}^0)
=
k(X) \otimes_{k(X^0)(x)} k(\tilde{X^0})(x)
\]
with a subfield of $h(\alpha)$,
and then define the extension
$\tilde{v}$ of $v$ by restriction from $h(\alpha)$.

If the base change is not scale-preserving,
then it translates the domains of the functions $b_i(M, \alpha,s)-s$ to 
the left by the amount $\tilde{v}(g)$.
\end{remark}

\subsection{Artin-Schreier characters and Dwork modules}

In order to make further progress, we need to collect some information
about Artin-Schreier characters and their associated differential modules.

\begin{defn} \label{D:artin-schreier}
Let $R$ be a connected $\FF_p$-algebra,
and fix a geometric point $\overline{x}$ of $\Spec R$.
An \emph{Artin-Schreier character} of $R$ is a homomorphism
\[
\tau: \pi_1(\Spec R, \overline{x}) \to \ZZ/p\ZZ
\]
which is discrete (i.e., has open kernel).
By Artin-Schreier theory (see, e.g., \cite[1.4.5]{katz-local}),
the group of Artin-Schreier characters is isomorphic to the 
cokernel of the map $\phi-1: R \to R$, where $\phi: R \to R$
is the $p$-power Frobenius.
Explicitly, given $u \in R$,
the extension $R[z]/(z^p - z - u)$ carries the structure of a
$\ZZ/p\ZZ$-torsor by the formula
\[
g(z) = z+g \qquad (g \in \ZZ/p\ZZ).
\]
\end{defn}

\begin{defn} \label{D:prepared}
For $J \subseteq [0, +\infty]$ a closed subinterval
not containing $+\infty$ (resp.\ containing $+\infty$), 
let $R_J$ be the Fr\'echet completion of 
$\ell[x,x^{-1}]$ (resp.\ $\ell[x]$) for
the $s$-Gauss valuations $v_s$ for $s \in J$.
In particular, if $J = [0, +\infty]$, we have $R_J = \ell \langle x \rangle$.

We say that $u = \sum_{i \in \ZZ} u_i x^i \in R_J$ is \emph{prepared}
if $u\in k(X^0)[x,x^{-1}]$ and $u_i = 0$ whenever $i$ is divisible by $p$.
Note that prepared elements form a subgroup of the additive group of $R$,
and that a prepared element $u$ defines the trivial Artin-Schreier character
of $R$ if and only if $v_s(u) > 0$ for all $s \in J$.
(Beware however that 
preparedness is stable only under \emph{center-preserving} base changes on
Hypothesis~\ref{H:relative2}.)
\end{defn}

\begin{lemma} \label{L:prepare parameter}
Given $u \in R_J$, after a simple base change on 
Hypothesis~\ref{H:relative2}, we can find a prepared element $u' \in
R_J$ defining the same Artin-Schreier character as $u$.
\end{lemma}
\begin{proof}
Given $u = \sum_{i \in \ZZ} u_i x^i \in R_J$, 
we can make certain changes to $u$
without changing its image in $R_J/((\phi-1)R_J)$.
\begin{itemize}
\item
We may replace $u$ by $u + y$ for any $y \in R_J$ with $v_s(y) > 0$ for all
$s \in J$, because the series $z = y + y^p + y^{p^2} + \cdots$ converges
and satisfies $z - z^p = y$. 
Hence 
we can replace $u$ by an element of $k(X^0)^{\alg}[x, x^{-1}]$;
after a simple base change on Hypothesis~\ref{H:relative2},
we can force $u$ into $k(X^0)[x, x^{-1}]$.
\item
We may eliminate all terms with $i$ nonzero and divisible by $p$.
\item
After a simple base change
on Hypothesis~\ref{H:relative2}, we may also eliminate the term $u_0 x^0$.
\end{itemize}
The resulting Artin-Schreier parameter is prepared.
\end{proof}

\begin{defn} \label{D:dwork2}
Let $K^\sigma$ be the fixed field of
$K$ under $\sigma_K$. 
Suppose that $\FF_q \subseteq k$ and that $K^\sigma$ contains
a primitive $p$-th root of unity $\zeta_p$. Let $\pi$
be the unique solution of $\pi^{p-1} = -p$ satisfying
\[
\pi \equiv \zeta_p-1 \pmod{(\zeta_p-1)^2}.
\]
Let $J \subseteq [0, +\infty]$ be a closed subinterval,
and suppose $u = \sum_{i \in \ZZ} u_i x^i \in R_J$ is prepared.
Choose a lift $\tilde{u}_i \in \gotho_L$ of $u_i$ for each
$i$ not divisible by $p$, such that all but finitely many of the
$\tilde{u_i}$ are zero. 
Let $M_u \in C_{v^0,*,J}$ be the free (Dwork)
module generated
by a single generator $\bv$, on which $\nabla$ acts via
\[
\nabla(\bv) = -\sum_{i \in \ZZ} \pi i \tilde{u}_i x^{i-1} \bv \otimes dx.
\]
If $u$ corresponds to an Artin-Schreier character $\tau$
of $R_J$, we also write $M_\tau$ for $M_u$.
\end{defn}

\begin{lemma} \label{L:dwork2}
Set notation as in Definition~\ref{D:dwork2}.
\begin{enumerate}
\item[(a)]
The isomorphism class of $M_u$ depends only on $\tau$, not on the choices of
$u$ or $\tilde{u} = \sum_{i \in \ZZ} \tilde{u}_i x^i$.
\item[(b)]
For some $\epsilon \in (0,1)$,
if we extend $d$ to
\[
S = R_{v^0,[\epsilon,1),J}[z]/((1 + \pi z)^p - (1 - p \pi \tilde{u})),
\]
then regard $S$ as an object in $C_{v^0,*,J}$ carrying an action
of $\ZZ/p\ZZ$ for which $g \in \ZZ/p\ZZ$ carries
$1 + \pi z$ to $\zeta_p^g (1 + \pi z)$,
the eigenspace of $S$ for the character $g \mapsto \zeta_p^g$
is isomorphic to $M_u$.
\item[(c)]
There exists a Frobenius structure on $M_u$.
\item[(d)]
We have
\[
b_1(M_u, s) = \max\{s, \max_{i \neq 0} \{ -v^0(u_i)  + (1-i)s\}\}
\qquad (s \in J).
\]
\end{enumerate}
\end{lemma}
\begin{proof}
Before proceeding, we observe
(as in the proof of Lemma~\ref{L:maintain by translate})
that if $u_i \neq 0$, then $-\log |\tilde{u}_i|_\rho = v^0(u_i)(-\log \rho)$
for $\rho$ sufficiently close to 1. If on the other hand $u_i = 0$,
then $-\log |\tilde{u}_i|_\rho$ is bounded below by a positive constant
for $\rho$ sufficiently close to 1. 

To prove (a), 
note that if $u, u' \in R_J$ are prepared and define the same
Artin-Schreier character $\tau$, then $u-u'$ defines the trivial
Artin-Schreier character. It thus suffices to check that
if $\tau$ is trivial,
then $M_u$ is trivial for any choices of the $\tilde{u}_i$.
 
If $\tau$ is trivial, then $v^0(u_i) + is > 0$ for all $i$ and all
$s \in J$.
For each $i$ for which $u_i \neq 0$, for $\rho$ sufficiently
close to 1, $|\tilde{u}_i x^i|_{\rho^{v^0},s} = 
\rho^{v^0(u_i) + is}$ is bounded above by a constant less than 1 uniformly
for all $s \in J$.
For each $i$ for which $u_i = 0$, 
$|x^i|_{\rho^{v^0},s} = \rho^{is}$ tends to 1 uniformly in $s$ as $\rho$ tends to 1.
Consequently, 
for $\rho$ sufficiently close to 1, $|\tilde{u}_i x^i|_{\rho^{v^0},s}$
is also bounded above by a constant less than 1 uniformly for all $s \in J$.
(In both cases, the conclusion is still valid even if $+\infty \in J$,
because in that case we only consider $i>0$.)
In all cases, we may deduce
(as in the proof of the last assertion of Lemma~\ref{L:dwork})
that the exponential $\exp(\pi \sum_i \tilde{u}_i x^i)$ converges in
$R_{v^0, [\epsilon,1), J}$ for $\epsilon$ sufficiently close to 1,
and hence $M_u$ is trivial as an element of $C_{v^0,*,J}$.

To prove (b), note that the action of $d$ on $S$ satisfies
\[
\frac{d(1 + \pi z)}{1 + \pi z}
= \frac{1}{p} \frac{d(1 - p \pi \tilde{u})}{1 - p \pi \tilde{u}}
= -\frac{\pi d\tilde{u}}{1 - p \pi \tilde{u}}.
\]
If we can find $y \in R_{v^0,[\epsilon,1),J}$
such that
\[
\frac{dy}{y} = 
\frac{p \pi^2 \tilde{u} d\tilde{u}}{1 - p \pi \tilde{u}} =
-\pi d\tilde{u} + \frac{\pi d\tilde{u}}{1 - p \pi \tilde{u}},
\]
we may then conclude that $y(1 + \pi z)$
is a generator for
the eigenspace of $S$ for the character $g \mapsto \zeta_p^g$
which maps to $\bv$ under an isomorphism with $M_u$.
We would like to take $y = \exp(w)$
for $w$ formally defined as
\[
w = \int \frac{p \pi^2 \tilde{u} d\tilde{u}}{1 - p \pi \tilde{u}}
= \int \sum_{i=0}^\infty p^{i+1} \pi^{i+2} \tilde{u}^{i+1}
d\tilde{u}
= \sum_{i=0}^\infty \frac{p^{i+1} \pi^{i+2}}{i+2} \tilde{u}^{i+2}.
\]
This becomes valid once we observe that 
$|p^{i+1} \pi^{i+2}| \leq |\pi^2(i+2)|$ for each $i \geq 0$,
as then we may deduce that for $\rho$ sufficiently close to 1,
$|w|_{\rho^{v^0},s}$ converges to a limit less than $|\pi|$
uniformly for $s \in J$.

Assertion (c) holds by the discussion following the statement
of Proposition~\ref{P:identically zero}. It also follows from (b),
since the ring $S$ is finite \'etale over $R_{v^0,[\epsilon,1),J}$
and so admits an extension of any Frobenius lift (in a manner compatible
with $d$).

To prove (d), we again use the initial observation and then apply
Lemma~\ref{L:dwork}.
\end{proof}
\begin{cor} \label{C:AS nontrivial}
With notation as in Definition~\ref{D:dwork2}, 
suppose that $b_1(M_u,s)=s$ for $s \in J$.
Let $\tau$ be the Artin-Schreier character defined by $u$.
Then for any closed subinterval $J'$ of the interior of $J$, 
the restriction of $\tau$ to $\pi_1(\Spec R_{J'}, \overline{x})$
is trivial.
\end{cor}
\begin{proof}
Lemma~\ref{L:dwork2} implies that for each $i \neq 0$,
$-v^0(u_i) + (1-i)s \leq s$ for all $s \in J$.
Since the difference between the two sides is affine-linear in $s$
with nonzero slope, the equality must be strict for all $s \in J'$.
Hence as a parameter for an Artin-Schreier extension of $R_{J'}$,
we may replace $u$ by $0$ without changing the extension; this
proves the claim.
\end{proof}

We can use this computation to classify certain Dwork subobjects
of the module $M$ defined in Definition~\ref{D:regime}.
(For the next statement,
recall that $\alpha$ is a multiplicative seminorm on $\ell[x]$
which agrees with $e^{-v^0(\cdot)}$ on $\ell$ and with $e^{-v(\cdot)}$
on $k(X^0)[x]$, viewed as a point of $\DD_{\ell}$,
and that $s_0 = -\log r(\alpha)$.)
\begin{prop} \label{P:rank 1 analysis}
Suppose $N$ is a subquotient of $M$ in $C_{v^0,*,*, \alpha}$ of rank $1$.
Let $J \subseteq [0, s_0)$ be a closed subinterval
with nonempty interior, and suppose that $\alpha_{0,e^{-s}} \geq \alpha$
for all $s \in J$.
Let $u \in R_J$ be prepared, and suppose that
$N \cong M_u$ in $C_{v^0,*,J,\alpha} = C_{v^0,*,J}$.
Let $J'$ be any closed subinterval of the interior of $J$.
Then the following conditions hold.
\begin{enumerate}
\item[(a)]
The function $b_1(N, \alpha,s)$ is continuous, piecewise affine-linear, and 
convex on $[0, s_0]$.
\item[(b)]
There exists a prepared $u_+ \in \ell \langle x \rangle$
such that
such that $M_u \cong M_{u_+}$ in $C_{v^0,*,J',\alpha}$. 
\item[(c)]
Suppose that $0 \in J$ and 
$M_u \cong M_{u_+}$ in $C_{v^0,*,J,\alpha}$.
Then $N \cong M_{u_+}$ in $C_{v^0,*,*,\alpha}$.
\item[(d)]
The terminal slope of $b_1(N, \alpha, s)$ is equal to $0$ unless
$b_1(N, \alpha,s)=s$ identically in a neighborhood of $s_0$.
\end{enumerate}
\end{prop}
\begin{proof}
We first note that $b_1(N, \alpha,s)$ is continuous, convex, and piecewise
affine-linear on $[0, s_0)$ by Lemma~\ref{L:convex2}. 
By Proposition~\ref{P:variation1} applied to $M$,
for $s$ in a neighborhood of $s_0$, there is an index $i$ such that
$b_1(N, \alpha,s) = b_i(M,\alpha,s)$. In particular, $b_1(N, \alpha,s)$
is affine-linear in a neighborhood of $s_0$; this proves (a).

Put $u = \sum_{i \in \ZZ} u_i x^i$ and put $u_+ = \sum_{i > 0} u_i x^i$.
Since $M_{u_+}$ belongs to $C_{v^0,*}$, we may apply the argument of the first
paragraph to $M \otimes M_{u_+}^\dual$ to deduce that
for $P = N \otimes M_{u_+}^\dual$,
$b_1(P, \alpha,s)$ is continuous, convex, and
piecewise affine-linear on $[0, s_0]$.
By Corollary~\ref{C:at most 1} applied to $M \otimes M_{u_+}^\dual$,
the terminal slope of $b_1(P, \alpha,s)$
is at most 1.
Hence for $s \in J$,
the slopes of $b_1(P, \alpha,s) = 
b_1(P, s)$ 
are all less than or equal to 1. But by Lemma~\ref{L:dwork2}, 
we have
\[
b_1(P, s) = \max\{s, \max_{i <0}
\{-v^0(u_i) + (1-i)s\}\} \qquad (s \in J),
\]
which can only have slopes less than or equal to 1 if
$v_s(u_i x^i) \geq 0$ for all $s \in J$. 
As in Corollary~\ref{C:AS nontrivial},
this implies that $u - u_+$ defines the trivial Artin-Schreier character
of $\pi_1(\Spec R_{J'}, \overline{x})$; then (b) follows
as in Lemma~\ref{L:dwork2}(a).

Suppose that $0 \in J$ and 
$M_u \cong M_{u_+}$ in $C_{v^0,*,J,\alpha}$.
Then $P$ is trivial in $C_{v^0,*,J,\alpha}$,
so $b_1(P,\alpha,s)$ is identically equal to 
$s$ for $s \in J$.
However, as in the previous paragraph, 
this function is convex, and in a neighborhood of $s_0$
is affine-linear with slope at most 1. Hence we must have
$b_1(P,\alpha,s) = s$ identically for $s \in [0, s_0)$.
To check that $P$ is trivial in $C_{v^0,*,*,\alpha}$,
it suffices to check the triviality in $C_{v^0,*,[0,s_1], \alpha}$ for each
$s_1 \in [0, s_0)$. We may check this, at the expense of 
losing the preparedness of $u$ and $u_+$ (for the remainder of
this paragraph), by applying a scale-preserving  
base change on Hypothesis~\ref{H:relative2} to reduce to the case
where $\alpha_{0, e^{-s_1}} \geq \alpha$, so 
$C_{v^0,*,[0,s_1]} = C_{v^0,*,[0,s_1],\alpha}$.
We then have a basis $\bv_1,\dots,\bv_d$ of horizontal sections of
$P$ in
$C_{v^0,*,J}$, such that for each $\rho$ sufficiently close to 1,
the restriction of each $\bv_i$ to $P_\rho$ extends
from $A_{F_\rho}(\rho^J)$ to $A_{F_\rho}[\rho^{s_1}, 1]$
(by Theorem~\ref{T:extend Robba cond}).
By Lemma~\ref{L:intersect2}, each $\bv_i$ is a horizontal section of
$P$ in $C_{v^0,*,[0,s_1],\alpha}$.
As noted
above, this suffices to prove (c).

To prove (d), 
note that by (b),
there exists a prepared $u_+ \in \ell \langle x \rangle$
such that $M_u \cong M_{u_+}$ in $C_{v^0,*,J',\alpha}$. 
Since it suffices to check (d)
after a base change on Hypothesis~\ref{H:relative2},
we can force the same conclusion to hold, but with $J'$ 
replaced by an interval containing 0.
By (c), we have $N \cong M_{u_+}$ in $C_{v^0,*,*,\alpha}$,
so to check (d) it now suffices to check its analogue with $N$ replaced 
by $M_{u+}$.

If $\alpha$ is of type (i), then
$b_1(M_{u_+}, \alpha,s)=s$ identically in a neighborhood of $s_0$.
If $\alpha$ is of type (iv), 
Proposition~\ref{P:identically zero} 
(applied after a base change on Hypothesis~\ref{H:relative2})
implies that either
$b_1(M_{u_+}, \alpha,s)$ has terminal slope $0$, or
$b_1(M_{u_+}, \alpha,s)=s$ identically in a neighborhood of $s_0$.
This proves (d).
\end{proof}

\subsection{Monodromy representations}

Having given a functor from certain representations to differential modules,
we now go in the reverse direction. This step is where we first take input
from the induction hypothesis, by invoking
local semistable reduction for valuations of transcendence defect $n$.

\begin{hypothesis} \label{H:local mono}
Throughout this subsection,
let $s_1 > 0$ be a real number not in the divisible closure of the
value group $\Gamma_{v^0}$ of $v^0$.
Note that $\Gamma_{v^0} \neq \RR$ because $\Gamma_{v^0}$ has rational rank
at most $\dim(X)$ by Abhyankar's inequality, so it is possible to choose
such an $s_1$.
Write $v_1$ as shorthand for $v_{s_1}$, as a valuation on $\ell(x)$.
By Remark~\ref{R:same completion},
we may view $k(X)$ as a subfield of the completion $\ell(x)_{v_1}$,
and thus restrict $v_1$ to $k(X)$.
\end{hypothesis}

\begin{defn} \label{D:project local mono}
Let $I_1 = \pi_1(\Spec k(X)_{v_1}, \overline{x})$ 
denote the inertia group of $\pi_1(X, \overline{x})$
for the valuation $v_1$, for $\overline{x}$ some geometric point of 
$\Spec \ell(x)_{v_1}$.
By Lemma~\ref{L:same corank}, $v_1$ has transcendence defect $n$, 
so by Hypothesis~\ref{H:relative}, we are granted 
local semistable reduction at $v_1$.
We may thus let $\tau: I_1 \to \GL(V)$ denote the semisimplified local 
monodromy representation of $\calE$ at $v_1$, as in
Definition~\ref{D:mono rep}.
\end{defn}

\begin{remark} \label{R:shape of mono}
Retain notation as in Definition~\ref{D:project local mono}.
We can gain some information about $\tau$ by using the subgroup
structure on $I_1$ coming from Definition~\ref{D:inertia group}.
First, denote $I'_1 = \pi_1(\Spec \ell(x)_{v_1}, \overline{x})$
viewed as a subgroup of $I_1$. 
The subgroup $\tau^{-1}(\tau(I'_1)) \subseteq I_1$ is open, corresponding
to the \'etale fundamental group of a finite separable extension of $k(X)_{v_1}$
contained in $\ell(x)_{v_1}$.
Apply the primitive element theorem to write this extension as
$k(X)_{v_1}(\gamma)$ with $\gamma \in \ell(x)_{v_1}$, and let $P(T)$
be the minimal polynomial of $\gamma$. Since $k(X^0)^{\alg}(x)$ is dense in
$\ell(x)_{v_1}$, we can choose $\delta \in k(X^0)^{\alg}(x)$
such that $v_1(\gamma - \delta) > v_1(\gamma' - \gamma)$ for each
root  $\gamma' \neq \gamma$ of $P(T)$. After a suitable simple base change
on Hypothesis~\ref{H:relative2}, we can force $\delta \in k(X^0)(x)$.
Then Hensel's lemma implies that $\gamma \in k(X)_{v_1}$,
in which case $\tau(I_1) = \tau(I'_1)$.

Second, assuming $\tau(I_1) = \tau(I'_1)$,
let $W'_1$ denote the wild inertia subgroup of $I'_1$.
The subgroup $\tau^{-1}(\tau(W'_1)) \subseteq I_1$ 
is open, corresponding to the \'etale fundamental group of a finite
\emph{tamely ramified} extension of $k(X)_{v_1}$.
Thus after a suitable simple base change and tame alteration on 
Hypothesis~\ref{H:relative2},
we have $\tau(I_1) = \tau(W'_1)$. In this case, $\tau(I_1)$
becomes a $p$-group since $W'_1$ is a pro-$p$-group.
\end{remark}

\begin{lemma} \label{L:localizing sub}
After performing a suitable simple base change on
Hypothesis~\ref{H:relative2}, there exist:
\begin{itemize}
\item
some $t_1,\dots,t_r \in \Gamma(P^0, \calO)$ whose images
in $\Gamma(P^0_k, \calO)$ cut out the components of $Z^0$
passing through $z^0$;
\item
for $i=1,\dots,r$, some $a_i, b_i \in \ZZ_{(p)}$ 
with $a_i < v^0(t_i)/s_1 < b_i$;
\item
some $\epsilon \in (0,1)$;
\end{itemize}
such that when applying Definition~\ref{D:mono rep}
to $v_1$, we can take the localizing subspace
$A$ to have the form
\[
\{y \in ]z^0[_{P^0} \times A_{K,x}[0,1]:
|x(y)| \in (\epsilon,1), \quad
|t_i(y)| \in (|x(y)|^{b_i}, |x(y)|^{a_i})
 \quad (i=1,\dots,r)\}.
\]
\end{lemma}
\begin{proof}
We first verify that if we start with a subspace of the desired form,
then apply a simple base change on Hypothesis~\ref{H:relative2}, 
the inverse image of the original subspace
contains another subspace of the desired form.
Namely, if $\tilde{P}_0$ is a smooth irreducible affine formal scheme
over $\Spf \gotho_K$ with special fibre $\tilde{X}^0$,
and $\tilde{t}_1, \dots, \tilde{t}_r \in \Gamma(\tilde{P}^0, \calO)$
have images in $\Gamma(\tilde{P}^0_k, \calO)$ cutting 
out the components of $\tilde{Z}^0$
passing through $\tilde{z}^0$,
then there is a unique invertible $r \times r$ matrix $A$ over $\QQ$ 
such that
\[
\tilde{v}^0(\tilde{t}_j) = \sum_{i=1}^r A_{ij} v^0(t_i).
\]
Since $s_1$ is not in the $\QQ$-span of $\Gamma_{v^0}$,
the quantity $v^0(t_i)/s_1$ must belong not just to the closed
interval $[a_i, b_i]$ but also to the open interval $(a_i, b_i)$.
Applying the linear transformation on column vectors defined by $A$
to the product $\prod_{i=1}^r (a_i, b_i) \subseteq \RR^r$
gives an open subset of $\RR^r$.
Inside that open subset, we can find a product 
$\prod_{i=1}^r (a'_i, b'_i)$ with
$a'_i, b'_i \in \ZZ_{(p)}$ and 
$a'_i < \tilde{v}_0(t'_i)/s_1 < b'_i$.
Then the subset
\[
\{y \in \tilde{P}^0_k \times A_{K,x}[0,1]:
|x(y)| \in (\epsilon,1), \quad
|\tilde{t}_i(y)| \in (|x(y)|^{b'_i}, |x(y)|^{a'_i})
\quad (i=1,\dots,r)\}
\]
is contained in the inverse image of the original subspace.

Before continuing, we note that given any $e_1,\dots,e_r \in \QQ$
and $j \in \ZZ$ such that $v_1(t_1^{e_1} \cdots t_r^{e_r} x^j) > 0$,
we can adjust the choice of the $a_i$ and $b_i$ (without any base change
on Hypothesis~\ref{H:relative2}) to ensure that
$|t_1^{e_1} \cdots t_r^{e_r} x^j| < 1$ everywhere on the localizing subspace.
This follows by the argument using linear transformations from
the previous paragraph.

It remains to verify that we can shrink the localizing subspace so as to
satisfy an additional condition of the form in the definition of a localizing
subspace. We split this verification into two steps. We first consider
a condition of the form $\delta \leq |f| < 1$
for $f$ (defined on some smooth formal lift of $X$)
having reduction $\overline{f} \in \Gamma(X, \calO)$
with $v_1(\overline{f}) > 0$.
Expand $\overline{f} = \sum_{i=0}^\infty \overline{f}_i x^i$
in $\calO_{X^0,z^0}\llbracket x \rrbracket$.
Since $s_1$ is positive and not in $\Gamma_{v^0}$,
there must exist an index $h$ such that
$v^0(\overline{f}_i) + i s_1 > v^0(\overline{f}_h) + h s_1$ for all
$i \neq h$; there must also exist an index $N > h$ such that
$(N-h) s_1 > v^0(\overline{f}_h)$.

By performing a suitable simple base change on Hypothesis~\ref{H:relative2},
then adjusting $a_i$ and $b_i$ as above,
we can ensure that the following conditions hold.
\begin{enumerate}
\item[(a)]
For $i=0,\dots,N-1$, $\overline{f}_i$ can be written as the product of 
a unit $\overline{u}_i \in \Gamma(X^0, \calO)^\times$ with the monomial
$t_1^{e_{i1}} \cdots t_r^{e_{ir}}$ for some nonnegative integers
$e_{i1},\dots,e_{ir}$.
\item[(b)]
For $i = 0,\dots,N-1$ with $i \neq h$,
we have $|t_1^{e_{i1}-e_{h1}} \cdots t_r^{e_{ir}-e_{hr}} x^{i-h}| < 1$ 
everywhere on the localizing subspace.
\item[(c)]
We have $|t_1^{-e_{h1}} \cdots t_r^{-e_{hr}} x^{N-h}| < 1$ 
everywhere on the localizing subspace.
\end{enumerate}
Recall that $f$ is defined on some smooth formal lift $P$ of $X$,
for which we identify $]z[_P$ with $]z^0[_{P^0} \times A_{K,x}[0,1)$.
Lift $\overline{u}_0,\dots,\overline{u}_{N-1}$ and
$\overline{g} = x^{-N}(\overline{f} - \sum_{i=0}^{N-1} \overline{f}_i x^i)$
to $u_0, \dots, u_{N-1},g \in \Gamma(P, \calO)$.
For $i=0,\dots,N-1$, put $f_i = u_i t_1^{e_{i1}} \cdots t_r^{e_{ir}}$.
For 
some $c < 1$, we have $|f - g x^N - \sum_{i=0}^{N-1} f_i x^i| \leq c$
everywhere on $]z[_P$.
Thanks to (b) and (c), we have $|g x^N + \sum_{i=0}^{N-1} f_i x^i|
= |t_1^{e_{h1}} \cdots t_r^{e_{hr}} x^h|$
everywhere on the localizing subspace. By 
taking $\epsilon > \max\{c,\delta\}^{1/L}$ for 
$L = e_{h1} b_1 + \cdots + e_{hr} b_r + h$,
we may ensure that $\delta \leq 
|t_1^{e_{h1}} \cdots t_r^{e_{hr}} x^h| = |f| < 1$ everywhere on the
localizing subspace.

We next consider a condition of the form $|f_1| < |f_2|$
for $f_1, f_2$ (again defined on some smooth formal lift of $X$)
having reductions $\overline{f}_1, \overline{f}_2 \in \Gamma(X, \calO)$ 
with $\overline{f}_2 
\neq 0$ and 
$v_1(\overline{f}_1/\overline{f}_2) > 0$.
By approximating $f_1, f_2$ as in the previous
paragraph, 
after performing a suitable simple base change on Hypothesis~\ref{H:relative2}
and adjusting $a_i$ and $b_i$,
we can find 
$e_{i1},\dots,e_{ir},j_i \in \ZZ$
for $i \in \{1,2\}$ such that for some $\delta \in (0,1)$, 
$\delta < |t_1^{e_{i1}} \cdots t_r^{e_{ir}} x^{j_i}| = |f_i| \leq 1$
everywhere on the localizing subspace.
(We can only ensure strict inequality on the right in case
$v_1(\overline{f}_i) > 0$.)
Since $v_1(\overline{f}_1/\overline{f}_2) > 0$,
we have $v_1(t_1^{e_{11} - e_{21}} \cdots t_r^{e_{1r} - e_{2r}}
x^{j_1 - j_2}) > 0$,
and so $|f_1/f_2| = |t_1^{e_{11} - e_{21}} \cdots t_r^{e_{1r} - e_{2r}}
x^{j_1 - j_2}| < 1$ everywhere on the localizing subspace.

We have now shown that we can add conditions as in the definition of a 
localizing subspace, while maintaining the desired shape of our subspace.
This proves the claim.
\end{proof}

Note that Lemma~\ref{L:localizing sub} can equally well be stated using
closed intervals $[a_j,b_j]$ instead of open intervals $(a_j,b_j)$, 
provided that we modify the definition
of a localizing subspace as in Remark~\ref{R:localizing1}.
The proof of that modified statement includes the proof of a
purely characteristic $p$ assertion. For later convenience, we extract
this assertion explicitly.
\begin{construction} \label{C:localizing sub2}
Let $f_i: X_i^0 \to X^0$ be an exposing sequence for $v^0$.
We then obtain an exposing sequence for $v$ as follows.
For each index $i$, 
pick $t_1,\dots,t_r \in k(X_i^0)$
which, in some neighborhood $U$ of 
the center of an exposed extension of $v^0$ to $k(X_i^0)$,
are regular and cut out the components of $f_i^{-1}(Z^0)$ passing through
the center.
For $a_j, b_j \in \ZZ_{(p)}$ for $j=1,\dots,r$ such that
$a_j \leq v^0(t_j)/s_1 \leq b_j$,
form the minimal normal local modification 
of $U \times_{X^0} X$
on which the functions $x^{b_j m}/t^m$ and $t^m/x^{a_j m}$
become regular for all $m \in \ZZ$ with $a_j m, b_j m \in \ZZ$,
and let $\tilde{X}_i^0$ be a toroidal resolution of singularities of this
(as in \cite{kkms}).
Then take the exposing sequence to be the
maps $\tilde{X}_i^0 \to X$ 
indexed by integers $i$ and tuples $(a_1,b_1,\dots,a_r,b_r) \in \ZZ_{(p)}^{2r}$.
\end{construction}

\begin{defn} \label{D:Tannakian}
Recall that we obtained $M$ from $\calE$ by base change along
a continuous homomorphism 
$\Gamma(V \times A_{K,x}[0,1], \calO)
\to R_{v^0,[\epsilon,1)}$ for some $\epsilon \in (0,1)$,
in which $V$ is the intersection of $]z^0[_{P^0}$
with a strict neighborhood of $]X^0 \setminus Z^0[_{P^0}$.
For a suitable choice of the localizing subspace $A$ in
Lemma~\ref{L:localizing sub},
we have $A \subseteq V \times A_{K,x}[0,1]$;
for some closed interval $J$ containing $s_1$ in its interior
and some $\epsilon \in (0,1)$, we obtain
a commuting square
\[
\xymatrix{
\Gamma(V \times A_{K,x}[0,1], \calO) \ar[r] \ar[d] & R_{v^0,[\epsilon,1)} 
\ar[d]\\
\Gamma(A, \calO) \ar[r] & R_{v^0,[\epsilon,1),J}
}
\]
We thus obtain a base change from the Tannakian category
$[\calE_A]$ generated by $\calE_A$ to
the Tannakian subcategory $[M_J]$ of $C_{v^0,*, J}$
generated by the restriction $M_J$ of $M$ from
$C_{v^0,*}$ to $C_{v^0,*,J}$.
Note that this functor is faithful because
the map $\Gamma(A, \calO) \to R_{v^0,[\epsilon,1),J}$ is injective.
Similarly, we obtain a base change of semisimplified categories
from $[\calE^{\semis}_A]$  to $[M_J^{\semis}]$.
(Note that the notation $M_J^{\semis}$ is unambiguous:
the image of $\calE_A^{\semis}$ is already semisimple
since it generates an algebraic Tannakian category 
by Lemma~\ref{L:Tannakian}.
This image thus coincides with the semisimplification of $M_J$.)

Note that 
every object $N$ of $[M_J^{\semis}]$ is a subobject of an object 
generated from $M_J^{\semis}$ by some sequence of
direct sums, tensor products, and duals.
The latter can be obtained by base change from an object of
$[\calE_A^{\semis}]$, by performing the same sequence of operations
on $\calE_A^{\semis}$. Hence the base change functor
satisfies the hypothesis of Lemma~\ref{L:Tannaka2}(c),
so the automorphism group of $[M_J^{\semis}]$ (for any fibre functor)
may be viewed
as a subgroup of the automorphism group of $[\calE_A^{\semis}]$.
By similar reasoning,
shrinking $J$ does not 
increase the automorphism group of $[M_J^{\semis}]$; we may thus
assume hereafter that $J$ is chosen so that the automorphism group
$[M_J^{\semis}]$ does not decrease upon shrinking $J$ further.

With this assumption, we may assert that the automorphism group of 
$[M_J^{\semis}]$ is insensitive to a simple base change on 
Hypothesis~\ref{H:relative2}. 
By this observation plus Lemma~\ref{L:Tannakian},
we may identify the automorphism group of $[M_J^{\semis}]$ with a subgroup of
the image of $\pi_1(\Spec k(X)_{v_1} \times_{\Spec k(X^0)} \Spec k(X^0)^{\alg},
\overline{x})$ under the semisimplified local monodromy representation $\tau$
of $\calE$ at $v_1$. The ring $k(X)_{v_1} \otimes_{k(X^0)} k(X^0)^{\alg}$
is a union of complete fields, and is thus henselian; we thus do not
change the fundamental group by restricting to the component whose spectrum
contains $\overline{x}$ and then completing, to obtain
$R_{[s_1,s_1]}$.
Moreover, by Corollary~\ref{C:count components}, for $J$ sufficiently
small, the images of $\pi_1(\Spec R_J, \overline{x})$ and
$\pi_1(\Spec R_{[s_1,s_1]}, \overline{x})$ under $\tau$ coincide.

Let $\tilde{\tau}$ denote the restriction of $\tau$ to
$\pi_1(\Spec R_J,\overline{x})$. We now have an identification of
the automorphism group of $[M_J^{\semis}]$ with a subgroup of the image of
$\tilde{\tau}$. We will show momentarily that the $p$-parts of these
two groups
coincide for $J$ sufficiently small; see Lemma~\ref{L:Tannakian2}.
\end{defn}

We need the following compatibility between the previous construction
and the construction relating Artin-Schreier characters to Dwork modules.
\begin{lemma} \label{L:match Dwork}
With notation as in Definition~\ref{D:Tannakian},
suppose $\psi$ is a subrepresentation of $\tilde{\tau}$
isomorphic to the Artin-Schreier character defined
by a prepared parameter $u \in R_J$. Then the corresponding subquotient
of $M_J$ in $C_{v^0,*,J}$ is isomorphic to $M_u$.
\end{lemma}
\begin{proof}
Suppose first that $\psi$ is the trivial representation, in which case we must
check that the subquotient of $M_J$ in $C_{v^0,*,J}$
under consideration is trivial. This holds because the automorphism
group of $[M_J^{\semis}]$ is a subgroup of the image of $\tilde{\tau}$,
so any trivial subrepresentation of $\tilde{\tau}$ gives rise
to a trivial object in $[M_J^{\semis}]$.

Let $Z'$ be the union of $Z$ with the zero locus of $x$.
Suppose next that we can produce,
for each prepared Artin-Schreier parameter $u \in R_J$,
an $F$-isocrystal $\calE_u$ of rank 1 
on $X \setminus Z'$ overconvergent along $Z'$,
for which
$\tilde{\tau}$ is the Artin-Schreier character defined by $u$,
and $M_J \cong M_u$. 
Then we may apply the previous paragraph
with $\calE$ replaced by $\calE_u^{\dual} \otimes \calE$ to conclude.
(Note that although $\calE_u$ is not defined on the zero
locus of $x$, the construction of $M_J$ from Definition~\ref{D:Tannakian}
still makes sense for $\calE_u$.)

It thus remains to produce $\calE_u$. Since $u \in k(X^0)[x,x^{-1}]$,
we may perform a simple base change on Hypothesis~\ref{H:relative2}
to ensure that $u \in \Gamma(X^0 \setminus Z^0, \calO)[x,x^{-1}]$.
We may then form a finite normal cover $X'$ of $X$ 
with function field $k(X)[z]/(z^p - z - u)$. The pushforward
of the trivial isocrystal on $X'$ gives an 
overconvergent $F$-isocrystal $\calF$ on $X \setminus Z'$ 
carrying an action of $\ZZ/p\ZZ \cong \Gal(k(X')/k(X))$. 
Since $\calF$ becomes globally constant on $X'$,
the local monodromy representation of $\calF$ produced by
Definition~\ref{D:mono rep} coincides with the restriction of
the global monodromy representation produced by Crew
\cite[Theorem~2.1]{crew} (see also
\cite[Theorem~2.3.4]{kedlaya-swan2}).
The latter is the regular representation of $\Gal(k(X')/k(X))$,
which splits into one-dimensional characters; correspondingly,
$\calF$ splits into rank 1 isocrystals, one of which has
monodromy representation equal to the Artin-Schreier character
defined by the parameter $u$. We call this module $\calE_u$.

It remains to verify that for this choice of $\calE_u$, we have
$M_J = M_u$. This follows from the description of $M_u$ given in
Lemma~\ref{L:dwork2}(b),
together with the observations that
\[
\pi^{-p} ((1 + \pi z)^p - (1 - p \pi \tilde{u}))
\equiv z^p - z - u \pmod{\gothm_K}
\]
(this matches up the Artin-Schreier extensions)
and
\[
\pi^{-1}(\zeta_p(1 + \pi z) - (1 + \pi z)) \equiv 1 \pmod{\gothm_K}
\]
(this matches up the $\ZZ/p\ZZ$-torsor structures).
\end{proof}

\begin{lemma} \label{L:Tannakian2}
With notation as in Definition~\ref{D:Tannakian},
for $J$ sufficiently small, the
automorphism group of $[M_J^{\semis}]$ is a subgroup of $\tilde{\tau}$
having the same $p$-Sylow subgroup.
\end{lemma}
It is possible to show that the groups actually match, but we will not need
that stronger result.
\begin{proof}
As in Definition~\ref{D:Tannakian},
we take $J$ small enough so that 
neither the 
automorphism group of $[M_J^{\semis}]$ nor the image of
$\pi_1(\Spec R_J, \overline{x})$ under $\tilde{\tau}$ decreases upon shrinking $J$.
Let $G$ be the image of $\tilde{\tau}$, and let $H$ be the automorphism
group of $[M_J^{\semis}]$, identified with a subgroup of $G$.
By Definition~\ref{D:inertia group}, $G$ has a unique $p$-Sylow subgroup $P$,
and $G/P$ is abelian of order prime to $p$.

Note that it is harmless to make a tame alteration on
Hypothesis~\ref{H:relative2} before checking the claim: this tame
alteration can reduce $G$ and $H$ but not $P$,
so the difference between $P$ and $H \cap P$ cannot decrease.
We may thus reduce to the case where $G = P$ is a $p$-group.

Suppose now that $H$ is a proper subgroup of $G$. Then $H$ is contained
in a maximal proper subgroup of $G$; since $G$ is a $p$-group, any
maximal proper subgroup is normal of index $p$. Thus there is
an order $p$ abelian character of $G$ whose restriction to $H$ is trivial;
however, by Lemma~\ref{L:match Dwork}, this would produce a
prepared parameter $u$ which 
corresponds to a nontrivial Artin-Schreier
character of $R_J$, but for which $M_u$ is trivial.
As noted in Definition~\ref{D:prepared}, we must have
$v_s(u) \leq 0$ for some $s \in J$, and likewise after shrinking
$J$; we must thus have $v_{s_1}(u) \leq 0$. Since $s_1$ is not in the divisible
closure of $\Gamma_{v^0}$, we must in fact have $v_{s_1}(u) < 0$;
however, this contradicts Lemma~\ref{L:dwork2}(d). We thus have $H = G$,
as desired.
\end{proof}

\subsection{Terminal decompositions}

\begin{defn}
We say that $M$ is \emph{terminally presented} if the following
conditions hold.
\begin{enumerate}
\item[(a)]
For $i = 1, \dots, \rank(M)$,
$b_i(M,\alpha,s)$ is affine-linear on $[0,s_0]$ with slope
at most $1$. 
\item[(b)]
For $i=1,\dots, \rank(M)-1$, either
$b_i(M, \alpha,s) = b_{i+1}(M, \alpha,s)$ identically for $s \in [0, s_0)$,
or
$b_i(M, \alpha,s) > b_{i+1}(M, \alpha,s)$ identically for $s \in [0, s_0)$.
\end{enumerate}
This property is stable under 
performing a base change on Hypothesis~\ref{H:relative2}, because the
function $b_i(M, \alpha,s)-s$ is preserved up to a left shift of the domain.
\end{defn}

\begin{lemma} \label{L:make affine}
After performing a suitable base change on Hypothesis~\ref{H:relative2}
(depending on $M$), we can force $M$ to become terminally presented.
\end{lemma}
\begin{proof}
This holds by Proposition~\ref{P:variation1} and
Corollary~\ref{C:at most 1}.
\end{proof}

One might hope that having $\cEnd(M)$ being terminally presented would allow
us to decompose $M$ according to $b_i(M, \alpha, s)$ in the manner
of Theorem~\ref{T:decomposition}. It is the source of some complication in the
argument that only a weaker decomposition result can be established.

\begin{lemma} \label{L:lift trivial}
With notation as in Definition~\ref{D:Tannakian},
suppose that $\alpha_{0,e^{-s}} \geq \alpha$ for $s \in J$,
and that $M$ is terminally presented.
\begin{enumerate}
\item[(a)]
Each trivial subobject of $M_J$ induces
a (unique) trivial subobject of $M$ in $C_{v^0, *, **, \alpha}$.
\item[(b)]
For $J$ sufficiently small,
there is a direct summand of $M_J$ whose Jordan-H\"older
constituents are exactly the trivial Jordan-H\"older constituents
of $M_J$.
\end{enumerate}
\end{lemma}
\begin{proof}
Pick some $s_2 \in (0, s_0)$ such that $(0,s_2) \cap J \neq \emptyset$,
and make a scale-preserving base change on Hypothesis~\ref{H:relative2}
to ensure that $\alpha_{0, e^{-s_2}} \geq \alpha$.
By Theorem~\ref{T:decomposition}, for each $\rho$ sufficiently close to 1,
we may decompose $M_\rho$ over $A_{F_\rho}(\rho^{s_2},1)$ as
$M_{\rho,0} \oplus M_{\rho,1}$, such that for $r = -\log \rho$
and $s \in J$, we have $f_i(M_{\rho,0},rs)=rs$ and 
$f_i(M_{\rho,1},rs)>rs$ for all $i$.
From Definition~\ref{D:mono rep},
there exists a finite \'etale extension $B$ of $A$
over which $\calE_A$ becomes unipotent;
tensoring with $R_{v^0,[\epsilon,1),J}$,
we obtain a finite
\'etale cover of $A_{F_\rho}(\rho^J)$ on which $M_{\rho,0}$
becomes unipotent. By Proposition~\ref{P:seed unipotent}, 
$M_{\rho,0}$ becomes unipotent after pullback along the cover
$h: A_{F_\rho}(\rho^{s_2/m},1) \to A_{F_\rho}(\rho^{s_2},1)$ 
defined by $x \mapsto x^m$ for some positive integer $m$ coprime to $p$.

This implies (e.g., by \cite[Lemma~4.29]{kedlaya-mono-over})
that if we formally adjoin $\log(x)$ to the structure sheaf of
$A_{F_\rho}(\rho^{s_2/m},1)$, then $h^* M_{\rho,0}$ acquires a full basis
of horizontal sections. Let $V$ be the $F_\rho$-vector space spanned by 
these horizontal sections;
then $V$ carries an action of the Galois group $G$ of the cover $h$.
Splitting $V$ into isotypical spaces for the $G$-action induces
a decomposition of $M_{\rho,0}$; the isotypical space for the trivial 
representation induces a decomposition $M_{\rho,0} \cong
M'_{\rho,0} \oplus M''_{\rho,0}$ in which the Jordan-H\"older
constituents of $M'_{\rho,0}$ 
consist precisely of the trivial Jordan-H\"older constituents
of $M_{\rho,0}$. 
(Compare \cite[Proposition~4.37]{kedlaya-mono-over}.)

Let $\bv$ be a horizontal section of $M$ in $C_{v^0, *, J}$. 
The image of $\bv$ in $M_\rho$ over $A_{F_\rho}(\rho^J)$
must belong to $M'_{\rho,0}$; it must then extend all the way over
$A_{F_\rho}(\rho^{s_2},1)$.
By Lemma~\ref{L:intersect2}, $\bv$ lifts to a horizontal section of $M$
in $C_{v^0,*,(0,s_2) \cup J}$. We may deduce (a) from this by glueing
over all choices of $s_2$.

We proceed now to (b). By shrinking $J$, we may assume that its left endpoint
is in $v^0(\ell)$; let $J'$ be the translate of $J$ with left endpoint $0$.
By performing a center-preserving base change on Hypothesis~\ref{H:relative2}
that shifts $J$ to $J'$, we obtain an equivalence of categories
between $C_{v^0,*,J}$ and $C_{v^0, *, J'}$.
(More precisely, if we choose 
$\overline{h} \in \ell$ with $v^0(\ell)$ being the left endpoint of $J$,
then the substitution $x \mapsto x [\overline{h}]$
induces an isomorphism $R_{v^0,I,J} \to R_{v^0,I,J'}$ for any interval
$I \subseteq (0,1)$.)
Let $\tilde{M} \in C_{v^0,*,J'}$ be the image of $M_J$ under this equivalence,
so that we obtain a decomposition $\tilde{M}_{\rho}
\cong \tilde{M}_{\rho,0} \oplus \tilde{M}_{\rho,1}$.
We may then combine this decomposition with the
decomposition of $\tilde{M}$ given by Lemma~\ref{L:convex decomp}
(by invoking
\cite[Lemma~2.3.1]{kedlaya-xiao} and Remark~\ref{R:gen fib intersect});
the resulting decomposition of $\tilde{M}$ reflects back to
a decomposition of $M_J$ lifting the decomposition $M_\rho \cong
M_{\rho,0} \oplus M_{\rho,1}$ for each $\rho$ sufficiently close to 1.

Let $N$ be the direct summand of $M_J$ lifting $M_{\rho,0}$.
Suppose $P$ is a Jordan-H\"older constituent of $M_J$
which becomes trivial after pullback
along $x \mapsto x^m$ for some positive integer $m$ coprime to $p$.
By Lemma~\ref{L:tame base change}, 
$b_1(P, s) = s$ for all $s \in J$; consequently, $P$ must be
a constituent of $N$.

On the other hand, we claim that after pullback along 
$x \mapsto x^m$ for some positive integer $m$ coprime to $p$, 
$N$ becomes unipotent.
In other words, after a suitable tame alteration on Hypothesis~\ref{H:relative2},
$N$ becomes unipotent.
By invoking Remark~\ref{R:shape of mono},
we can choose the tame alteration so that 
$\tau(I_1) = \tau(W'_1)$ is a $p$-group.
By Lemma~\ref{L:Tannakian2},
by making $J$ sufficiently small, we can ensure that the images of $\tau$
and $\tilde{\tau}$ coincide, so the latter is also a $p$-group.
Suppose by way of contradiction that $N$ has still not become unipotent.
By Lemma~\ref{L:wild structure2},
after possibly replacing $K$ by a finite extension, we would find
a nontrivial one-dimensional subquotient $P$ of (the pullback along the 
alteration of) 
$N \oplus \cEnd(N)$
such that for some nonnegative integer $h$, $P^{\otimes p^h}$
corresponds to a nontrivial Artin-Schreier character of $I_1$ via
Lemma~\ref{L:match Dwork}.
But $b_1(N, \alpha,s) = s$ for all $s \in J$, so by 
Corollary~\ref{C:AS nontrivial},
after shrinking $J$ this character must become trivial,
a contradiction. (Compare the proof of Lemma~\ref{L:find nontrivial}.)

By splitting $N$ into isotypical components for the Galois group of
the cover $x \mapsto x^m$ (as above), we obtain a direct summand of $M_J$
whose Jordan-H\"older
constituents are exactly the trivial Jordan-H\"older constituents
of $M_J$. This proves (b).
\end{proof}
\begin{cor} \label{C:lift any}
With notation as in Definition~\ref{D:Tannakian},
suppose that $\alpha_{0,e^{-s}} \geq \alpha$ for $s \in J$,
and that $\cEnd(M)$ is terminally presented.
Then for $J$ sufficiently small, we have the following.
\begin{enumerate}
\item[(a)]
Any irreducible subquotient of $M$ in $C_{v^0, *, J}$
is isomorphic to an irreducible subobject.
\item[(b)]
Each irreducible subobject of $M$ in $C_{v^0, *, J}$
induces a (unique) subobject of $M$ in $C_{v^0, *, **, \alpha}$.
\end{enumerate}
\end{cor}
\begin{proof}
By Lemma~\ref{L:lift trivial}(b), for $J$ sufficiently small, 
$\cEnd(M_J)$ splits 
as a direct sum
$N_0 \oplus N_1$ in which the Jordan-H\"older
constituents of $N_0$ are exactly the trivial Jordan-H\"older constituents
of $\cEnd(M_J)$.
Let $P$ be any irreducible subquotient of $M_J$.
Then $P^\dual \otimes M_J$ occurs as a subquotient of $\cEnd(M_J)$,
so it also splits as a direct sum $Q_0 \oplus Q_1$ in which
the Jordan-H\"older
constituents of $Q_0$ are exactly the trivial Jordan-H\"older constituents
of $P^\dual \otimes M_J$. The image of $P \otimes Q_0$ under the trace map
$P \otimes P^\dual \otimes M_J \to M_J$
is a nonzero subobject of $M_J$ whose Jordan-H\"older constituents
are all isomorphic to $P$; this proves (a).

By a similar argument, $P$ also occurs as a nonzero quotient of $M_J$.
Composing a surjection $M_J \to P$ with an injection $P \to M_J$
gives an endomorphism of $M_J$ with image isomorphic to $P$.
This suffices to prove (b) because
Lemma~\ref{L:lift trivial} applied to $\cEnd(M)$ transfers this endomorphism
to $C_{v^0,*,**,\alpha}$.
\end{proof}

\subsection{Terminal unipotence} 
\label{subsec:terminal unip}

We are now equipped to attack Theorem~\ref{T:induct local semi} by reducing
to consideration of Artin-Schreier characters.

\begin{defn}
Let $M_{**}$ be the restriction of $M$ to $C_{v^0,*,**,\alpha}$.
We say $M$ is \emph{terminally unipotent} if 
$M_{**}$ is a successive extension of trivial objects in $C_{v^0,*,**,\alpha}$.
\end{defn}

\begin{lemma} \label{L:find nontrivial}
After a suitable sequence of base changes and tame alterations
on Hypothesis~\ref{H:relative2}, 
and after replacing $K$ by a finite extension,
one of the following conditions
becomes true.
\begin{enumerate}
\item[(a)]
$M$ is terminally unipotent.
\item[(b)]
We can find a nontrivial 
subobject $N$ of $M_{**} \oplus
\cEnd(M_{**})$ of rank $1$,
such that $N^{\otimes p^h}$ is trivial for some positive integer $h$.
\end{enumerate}
\end{lemma}
\begin{proof}
By Lemma~\ref{L:make affine}, after a 
base change on Hypothesis~\ref{H:relative2}, we may 
assume that
$\cEnd(M)$ is terminally presented.
Note that
this condition is preserved by further base changes and tame alterations
on Hypothesis~\ref{H:relative2}.

Pick $s_1 > 0$ not in the divisible closure of the value
group $\Gamma_{v^0}$ of $v^0$. After a scale-preserving base change on
Hypothesis~\ref{H:relative2}, we may assume that $\alpha_{0,e^{-s_2}} \geq
\alpha$ for some $s_2 \in (s_1, s_0)$.
We may then set notation as in Definition~\ref{D:Tannakian},
shrinking $J$ if necessary to force $J \subseteq [0, s_2]$.
Note that all of these conditions are preserved by further
\emph{simple} base changes and tame alterations.

By Remark~\ref{R:shape of mono}, after a simple base change
and tame alteration on Hypothesis~\ref{H:relative2},
we have $\tau(I_1) = \tau(W'_1)$, which is a $p$-group.
By shrinking $J$ again and invoking Lemma~\ref{L:Tannakian2}, 
we may ensure that $\tilde{\tau}(\pi_1(\Spec
R_J, \overline{x})) = \tau(I_1)$.
By Lemma~\ref{L:wild structure2}, 
after possibly replacing $K$ by a finite extension,
either $\tilde{\tau}$
is trivial, or $\tilde{\tau} \oplus \cEnd(\tilde{\tau})$ has
a nontrivial one-dimensional subrepresentation $\psi$
such that $\psi^{\otimes p^h}$ is trivial for some positive integer $h$.
In the former case, Corollary~\ref{C:lift any} ensures that
$M_{**}$ is unipotent. In the latter case, Corollary~\ref{C:lift any}
ensures firstly that $\psi$ corresponds to a subobject (not just a
subquotient) $N$ of $M_J$, secondly that $N$ lifts to a subobject
of $M_{**}$, and thirdly that $N^{\otimes p^h}$ is trivial as a subobject
of $M_{**}^{\otimes p^h}$.
\end{proof}

We now use what we have learned from the machinery of differential modules 
to show that we can always reduce consideration of an Artin-Schreier
extension to the very special shape permitted by an Artin-Schreier
alteration on Hypothesis~\ref{H:relative2}. Without this analysis,
we would be unable to simplify $\calE$ by passing up Artin-Schreier extensions
while maintaining smoothness of the map $X \to X^0$ at the center of $v$.

\begin{lemma} \label{L:kill AS}
Suppose that $N$ is a rank $1$ object of the Tannakian subcategory of
$C_{v^0,*,**,\alpha}$ generated by $M_{**}^{\semis}$,
such that $N^{\otimes p^h}$ is trivial for some nonnegative integer $h$.
Then after a suitable
sequence of base changes and Artin-Schreier
alterations on Hypothesis~\ref{H:relative2},
$N$ becomes trivial.
\end{lemma}
\begin{proof}
It suffices to treat the case $h=1$, as the general case follows by
induction on $h$ by applying the base case to $N^{\otimes p^{h-1}}$.
Reduce to the case where $M$ is terminally presented,
then set notation as in the proof of Lemma~\ref{L:find nontrivial}.
Then $N$ corresponds to an Artin-Schreier representation $\psi$
of $\pi_1(\Spec R_J, \overline{x})$ for some $J$,
which after a simple base change on Hypothesis~\ref{H:relative2}
we may assume is generated by a prepared parameter $u \in R_J$.
By Lemma~\ref{L:match Dwork}, $N$ is isomorphic as an element of
$C_{v^0,*,J}$ to the Dwork module $M_u$.

By Proposition~\ref{P:rank 1 analysis}, we may choose $u \in
\ell \langle x \rangle$ prepared so that $N \cong M_u$ in
$C_{v^0,*,**,\alpha}$. 
(This requires combining parts (b) and (c) of 
Proposition~\ref{P:rank 1 analysis}, as in the proof of
Proposition~\ref{P:rank 1 analysis}(d).)
Moreover, the terminal slope of $b_1(M_u, \alpha,s)$
is equal to 0 unless $b_1(M_u, \alpha,s)=s$ identically (only
\emph{a priori} in a neighborhood
of $s_0$, but we assumed that $M$ is terminally presented). 
In the latter case, by Corollary~\ref{C:AS nontrivial},
$\psi$ is trivial. In the former case,
write $u = \sum_{i=1}^h u_i x^i$; by Lemma~\ref{L:dwork2}(d),
we have $v^0(u_1) + s \leq v^0(u_i) + is$ for $i>1$ and $s \in J$.
We also have $v^0(u_1) + s \leq s$ for $s \in J$, so $v^0(u_1) \leq 0$.
By a center-preserving base change on Hypothesis~\ref{H:relative2},
we may force the previous inequality to hold for all $s \geq 0$,
while preserving the property that $\alpha_{0,e^{-s}} \geq \alpha$
for some $s > 0$.

Put $u' = u_1 x$; then $b_1(M_u,s) \geq b_1(M_u \otimes M_{u'}^\dual,
s)$ for $s \geq 0$. 
In this inequality, the left side
is constant and greater than or equal to $s$; 
by contrast, the right side has negative
slope until it equals $s$ and then is identically equal to $s$ thereafter.
It follows that for $s > 0$, we have 
$b_1(M_u,s) > b_1(M_u \otimes M_{u'}^\dual,s)$ unless both sides
are equal to $s$.

By a center-preserving base change on Hypothesis~\ref{H:relative2},
we can force $b_1(M_u,s) > b_1(M_u \otimes M_{u'}^\dual,s)$ for $s \geq 0$,
while preserving the property that $\alpha_{0,e^{-s}} \geq \alpha$
for some $s > 0$. Taking $s=0$, we find
$v^0(u_1) < v^0(u_i)$, or $v^0(u_i/u_1) > 0$.
By a simple base change on Hypothesis~\ref{H:relative2}, 
we can force $u_1$ to be a $p$-th power in $k(X^0)$, 
and we may force $u_i/u_1$ to belong to the maximal ideal of $\calO_{X^0,z^0}$
for $i > 1$. 
Since $v^0(u_1) \leq 0$, we may also force $u_1^{-1/p}$ to belong to 
$\Gamma(X^0, \calO)$.
Perform an Artin-Schreier alteration on 
Hypothesis~\ref{H:relative2} with $g = u_1^{-1/p}$ and $a_i = u_i/u_1$
for $i > 1$. 

At this point, the representation $\psi$ is now trivial.
By Lemma~\ref{L:match Dwork} again,
$N$ is trivial as an object in $C_{v^0, *, J} = C_{v^0,*,J,\alpha}$.
Perform a center-preserving base change on Hypothesis~\ref{H:relative2}
to force $0$ into $J$;
then by Proposition~\ref{P:rank 1 analysis}(c),
$N$ is also trivial in $C_{v^0,*,**,\alpha}$.
\end{proof}

Putting together the two previous arguments yields the following.

\begin{prop} \label{P:terminally unip}
After a sequence of base changes, tame alterations, and Artin-Schreier
alterations on Hypothesis~\ref{H:relative2},
and after possibly replacing $K$ by a finite extension,
$M$ becomes terminally unipotent.
\end{prop}
\begin{proof}
We proceed by descending induction on the number of
trivial Jordan-H\"older constituents of $M_{**} \oplus \cEnd(M_{**})$.
The maximum value of this number occurs when $M_{**} \oplus \cEnd(M_{**})$
is itself terminally unipotent, in which case $M$ is as well.

By Lemma~\ref{L:find nontrivial} (applied after possibly replacing
$K$ by a finite extension),
we can perform a sequence of base changes and tame alterations
on Hypothesis~\ref{H:relative2}, after which 
either $M$ is terminally unipotent, or
there exists a nontrivial subobject $N$ of $M_{**} \oplus
\cEnd(M_{**})$ of rank $1$,
such that $N^{\otimes p^h}$ is trivial for some positive integer $h$.
In the former case, we are done. In the latter case,
by Lemma~\ref{L:kill AS},
we can perform a sequence of
base changes and Artin-Schreier
alterations on Hypothesis~\ref{H:relative2} to make $N$ trivial.
Then we may apply the induction hypothesis to conclude.
\end{proof}

\subsection{Completion of the proof}
\label{subsec:completion}

To finish the proof of Theorem~\ref{T:induct local semi}, 
we now show that terminal unipotence of $M$
implies local semistable reduction at $v$.
Beware that we cannot apply Proposition~\ref{P:refined log-extend}
directly, since that result assumes local semistable reduction at $v$.

\begin{prop} \label{P:endgame}
Suppose that $M$ is terminally unipotent.
Then $\calE$ admits local semistable reduction at $v$.
\end{prop}
\begin{proof}
Since $M$ is terminally unipotent,
we may make a base change on Hypothesis~\ref{H:relative2}
to reduce to the case where the restriction of $M$ to
$C_{v^0,*,*,\alpha}$ (rather than $C_{v^0,*,**,\alpha}$)
is a successive extension of trivial objects.
Note that this condition is preserved by further
base changes on Hypothesis~\ref{H:relative2}.

We then set notation as in the proof of Lemma~\ref{L:find nontrivial}.
Namely, pick $s_1 > 0$ not in the divisible closure of the value
group $\Gamma_{v^0}$ of $v^0$. After a scale-preserving base change on
Hypothesis~\ref{H:relative2}, we may assume that $\alpha_{0,e^{-s_2}} \geq
\alpha$ for some $s_2 \in (s_1, s_0)$.
We may then set notation as in Definition~\ref{D:Tannakian},
shrinking $J$ if necessary to force $J \subseteq [0, s_2]$.
By a suitable sequence of base changes and tame alterations on 
Hypothesis~\ref{H:relative2}, we can force $\tau(I_1) = \tau(W'_1)$.
The latter group must be trivial, or else as in the
proof of Lemma~\ref{L:find nontrivial}, we could find a nontrivial subobject $N$
of $M_{**}$, contradicting our assumption that $M$ is terminally unipotent.

In other words, we are now in the case where the semisimplified local monodromy
representation of $\calE$ at $v_1$ is trivial.
We now apply Proposition~\ref{P:refined log-extend} to $v_1$,
using an exposing sequence generated by Construction~\ref{C:localizing sub2}.
This implies that after a simple base change on Hypothesis~\ref{H:relative2},
we can choose:
\begin{itemize}
\item
$t_1,\dots,t_r \in \Gamma(X^0, \calO)$
which cut out the components of $Z^0$ passing through
$z^0$;
\item
$a_i, b_i \in \ZZ_{(p)}$ for $i=1,\dots,r$ such that
$v^0(t_i)/b_i < s_1 < v^0(t_i)/a_i$;
\item
a positive integer $m$ with $a_i m, b_i m \in \ZZ$ for $i=1,\dots,r$;
\item
a desingularization $\tilde{Y}$ 
of the closure in $X \times_k \AAA^{2r}_k$
of the graph of the rational map
$X \dashrightarrow \AAA^{2r}_k$ defined by
$x^{b_i m}/t_i^m$ and $t_i^m/x^{a_i m}$ for $i=1,\dots,r$
(which exists by toroidal resolution of singularities
\cite{kkms});
\end{itemize}
such that
$\calE$ is log-extendable to $(\tilde{Y}, \tilde{W})$,
for $\tilde{W}$ the inverse image of $Z$ in $\tilde{Y}$.
By one direction of \cite[Theorem~6.4.5]{kedlaya-part1},
$\calE$ has unipotent local monodromy along every divisorial valuation
of $k(X)$ centered on $\tilde{Y}$.

To extract the desired conclusion, we argue as in \cite{kedlaya-part3}.
Fix some local coordinates of $X$ at $z$ including $x, t_1,\dots,t_r$,
and use them to define the derivations $\del_0 = \frac{\del}{\del x}$ 
and $\del_i = \frac{\del}{\del t_i}$ for $i=1,\dots,r$.
For $R = (\rho_0, \dots, \rho_r) \in (0,1)^{r+1}$, 
define $T(\calE, R)$ using the derivations $\del_0,\dots,\del_r$,
as in \cite[Definition~4.2.1]{kedlaya-part3}.
Put
\[
D = \{(u_0,\dots,u_r) \in \RR^{r+1}: u_0,\dots,u_r > 0, \quad u_1 + \cdots + u_r = 1\}.
\]
Define the function $h: D \cap \QQ^{r+1} \to [0, +\infty) \cap \QQ$
by requiring that for $U = (u_0,\dots,u_r) \in D \cap \QQ^{r+1}$,
for $\rho \in (0,1)$ sufficiently close to 1,
we have
$T(\calE, (\rho^{u_0},\dots,\rho^{u_r})) = \rho^{h(U)}$.
(The quantity $h(U)$ exists by \cite[Theorem~4.4.7]{kedlaya-part3}.)
The function $h$ then has the following properties.
\begin{itemize}
\item
(Monotonicity) The function $h(U)$ is nonincreasing in $u_0$
(as in \cite[\S 11.4]{kedlaya-course}).
\item
(Relation to unipotent monodromy) 
Let $v_U$ be the toric valuation
with $v_U(x) = u_0$ and $v_U(t_i) = u_i$ for $i=1,\dots,r$. Then $h(U) = 0$
if and only if $\calE$ acquires unipotent local monodromy along $v_U$
after pulling back along some tamely ramified cover branched over the zero loci of
$x, t_1,\dots,t_r$ \cite[Proposition~5.3.4]{kedlaya-part3}.
\end{itemize}
By the relation to unipotent monodromy,
we have $h(U) = 0$ for all $U \in D \cap \QQ^{r+1}$ such that
$b_i^{-1} u_i \leq u_0 \leq a_i^{-1} u_i$ for $i=1,\dots,r$.
By monotonicity, we also have $h(U) = 0$ for all $U \in D \cap \QQ^{r+1}$
such that $b_i^{-1} u_i \leq u_0$ for $i=1,\dots,r$. In particular, since
$s_1 \leq v(x)$, this includes a neighborhood of the point $(v(x), v(t_1),\dots,v(t_r))$.

By a suitable toroidal blowup in $x,t_1,\dots,t_r$, 
we obtain a local alteration $f_1: X_1 \to X$
around $v$, such that 
$(X_1, Z_1)$ is a smooth pair for 
$Z_1 = f_1^{-1}(Z)$, and each of the components of $Z_1$ corresponds to a
divisorial valuation of the form $v_U$ for some $U \in D$ with $h(U) = 0$.
For $m$ a positive integer prime to $p$, let
$X_{2,m}$ be the finite cover of $X_1$ which is tamely ramified of degree $m$ 
along each component of $Z_1$;
then $f_{2,m}: X_{2,m} \to X$ is again a local alteration around $v$ such that $(X_{2,m}, Z_{2,m})$ is a smooth pair for
$Z_{2,m} = f_{2,m}^{-1}(Z)$. By the relation to unipotent monodromy, for $m$
sufficiently divisible, $f_{2,m}^* \calE$ 
admits unipotent local monodromy along each component of $Z_{2,m}$.
By the other direction of \cite[Theorem~6.4.5]{kedlaya-part1}, $f_{2,m}^* \calE$ is log-extendable
to $(X_{2,m}, Z_{2,m})$ as desired.
\end{proof}

To conclude, we summarize how all of the ingredients in this section
come together to give a proof of Theorem~\ref{T:induct local semi}.
\begin{proof}[Proof of Theorem~\ref{T:induct local semi}]
It suffices to show that under Hypothesis~\ref{H:relative},
$\calE$ admits local semistable reduction at $v$.
During the course of the proof, we may at any time replace
$X$ by a (not necessarily separable)
local alteration of $X$ around $v$, pulling $\calE$
back accordingly. This mean on one hand that by Lemma~\ref{L:get relative},
we may reduce to the case where Hypothesis~\ref{H:relative2} holds,
and on the other hand, we may perform base changes, tame alterations,
and Artin-Schreier alterations on Hypothesis~\ref{H:relative2}.
Using such operations, by Proposition~\ref{P:terminally unip},
we may arrive at the situation where $M$ is terminally unipotent.
By Proposition~\ref{P:endgame}, $\calE$ admits local semistable reduction at $v$
as desired.
\end{proof}

\appendix

\section{Appendix}

In this appendix, we include two technical corrections,
one to \cite{kedlaya-part1} and one to \cite{kedlaya-part2}.
We also discuss the prospect of removing the hypothesis that the coefficient field
$K$ is discretely valued.

\subsection{A technical correction}

This subsection reports and corrects an error in \cite{kedlaya-part1},
kindly pointed out to us by Atsushi Shiho. To lighten notation, in
this subsection we make all
citations directly to \cite{kedlaya-part1} without saying so each time.

The error occurs in two places. In 
Proposition 3.5.3, the 
erroneous statements are the last two sentences of the proof: since
the strict neighborhood $V'$ may not have integral reduction, one cannot
necessarily embed it isometrically into a field.
A similar error occurs in Lemma~3.6.2,
because $X$ need not have integral reduction. 

In the case of Lemma~3.6.2,
this is most easily corrected by simply
requiring $X$ to have integral reduction,
and making the same requirement in Proposition~3.6.9,
This suffices because
the only invocation of Lemma~3.6.2 is in 
Proposition~3.6.9,
and all invocations of Proposition~3.6.9 occur
in cases where $X$ has integral reduction.

In the case of Proposition~3.5.3, the only invocation occurs in
Lemma~5.1.1, in which we may assume $b > 0$. Under this extra
hypothesis, we may correct the proof of Proposition~3.5.3 as follows.
(Thanks to Shiho for feedback on this correction.)

Retain notation from the proof of Proposition~3.5.3.
We have from the correct part of the proof that 
$H^0_{V_\lambda}(V_\lambda \times A^n_K[b,c], \mathcal{E}) \neq 0$.
Take the $\mathcal{O}$-span of this set; it is a submodule of 
$\mathcal{E}$ stable under $\nabla$. Restrict to 
$V_\lambda^0 \times A^n_K[b,c]$, for
$V_\lambda^0$ the nonlogarithmic locus of $V_\lambda$;
by Proposition~3.3.8, this restriction extends to a 
log-$\nabla$-submodule $\mathcal{F}$ of $\mathcal{E}$.
On $(V_\lambda^0 \cap ]X[) \times A^n_K[b,c]$, $\mathcal{F}$ is
constant relative to $V_\lambda^0 \cap ]X[$; 
we may infer the same conclusion on $]X[ \times A^n_K[b',c']$ 
for any closed subinterval $[b',c']$ of $(b,c)$, via
Corollary 3.4.5.

In other words, the whole proof reduces to the case when
$\mathcal{E}$ is constant, not just unipotent, on $]X[ \times A^n_K[b,c]$.
In this case, we must show that $\mathcal{E}$ is also constant
on $V_\lambda \times A^n_K[b,c]$.
We first prove that 
$g_\lambda: \pi_1^* H^0_{V_\lambda} (V_\lambda \times A^n_K[b,c],
\mathcal{E}) \to \mathcal{E}$ is surjective for some $\lambda$.
For $\mathbf{v} \in \Gamma(V \times A^n_K[d,e], \mathcal{E})$, 
the correct part of the proof of Proposition~3.5.3
shows that the sequence $D_l(\mathbf{v})$ converges on
$V_\lambda \times A^n_K[b,c]$ for some $\lambda$.
Moreover, as in Proposition~3.4.3, 
on $]X[ \times A^n_K[b,c]$ we have the identity
\[
\mathbf{v} = \sum_{J \in \mathbb{Z}^n} t_1^{j_1} \cdots t_n^{j_n}
f(t_1^{-j_1} \cdots t_n^{-j_n})
\]
(using the hypothesis that $b \neq 0$),
so the cokernel of $g_\lambda$ has no support on $]X[ \times A^n_K[b,c]$.
On $V_\lambda \times A^n_K[b,c]$, the support of this cokernel is
a closed analytic subspace not meeting $]X[ \times A^n_K[b,c]$;
by the maximum modulus principle, it also fails to meet
$V_{\lambda'} \times A^n_K[b,c]$ for some $\lambda' \in (1,\lambda]$.
Hence for suitable $\lambda$, $g_\lambda$ is surjective.

Over $V_\lambda^0 \times A^n_K[b,c]$,
the map $g_\lambda$ is automatically a morphism in 
$\mathrm{LNM}_{V_\lambda^0 \times A^n_K[b,c]}$ because there are no
logarithmic singularities; hence Proposition~3.2.20 implies that
the restriction of $\mathcal{E}$ to $V_{\lambda}^0 \times A^n_K[b,c]$
is constant over $V_\lambda^0$.

To finish the proof that $\mathcal{E}$ is constant, it suffices to do so
after extending scalars from $K$ to a finite extension. By doing so,
we may ensure that there exists a $K$-rational point $x \in A^n_K[b,c]$.
Let $\mathcal{E}_0$ be the restriction of $\mathcal{E}$ to
$V_\lambda \times \{x\}$, identified with $V_\lambda$. We then
have an isomorphism of $\pi_1^* \mathcal{E}_0$ with $\mathcal{E}$
over $V_\lambda^0 \times A^n_K[b,c]$ (because $\mathcal{E}$ is
constant over $V_\lambda^0$); by Proposition~3.3.8
(applied to the graph of the isomorphism inside of
$\pi_1^* \mathcal{E}_0 \oplus \mathcal{E}$), this extends to
an isomorphism $\pi_1^* \mathcal{E}_0 \cong \mathcal{E}$
over $V_\lambda \times A^n_K[b,c]$.

We also note that Shiho has generalized the results of 
\S 3 in \cite{shiho-log}, using somewhat different
arguments. Hence one can also avoid the aforementioned errors (and gain
some simplification in the proofs) simply by
invoking \cite{shiho-log} instead.

\subsection{Another technical correction}

The referee of the present paper, and independently Atsushi Shiho,
have pointed out some
problems with the handling of imperfect base fields in \cite{kedlaya-part2}.
We remedy these here.

We first note that \cite[Theorem~3.1.3]{kedlaya-part2} is only
a correct statement of de Jong's alterations theorem in case $k$
is perfect; otherwise, we only may conclude that $\overline{X}_1$
is regular, not necessarily smooth over $k$.

As a result, the statements of \cite[Conjecture~3.2.5]{kedlaya-part2}
and its variants are not tenable. The simplest way to correct the statement,
as we have done in \S~\ref{subsec:semistable} of this paper,
is to weaken the definition of semistable reduction 
\cite[Definitions~3.2.4 and~3.4.1]{kedlaya-part2} to 
require the quasiresolution and logarithmic extension only to be defined
over $k^{-q^n}$ for some nonnegative integer $n$
(with $K$ then replaced with $K^{\sigma_K^{-n}}$).

With this change, the fact that semistable reduction for a given
isocrystal over $k$ is equivalent to semistable reduction over $k^{\perf}$
becomes an immediate consequence of the fact that a pair $(X,Z)$
of $k$-varieties becomes 
smooth over $k^{\perf}$ if and only if it becomes smooth over
$k^{p^{-n}}$ for some $n$ (and hence for all sufficiently large $n$).
As a result, \cite[Proposition~3.2.6]{kedlaya-part2} remains correct
(using the new definition), with the given proof applied only for $k$
perfect.

We similarly modify the definition of local semistable reduction
\cite[Definitions~3.3.1 and~3.4.3]{kedlaya-part2} to allow replacing $k$ with
$k^{-q^n}$ for some positive integer $n$ before constructing the
quasiresolution. The proof of \cite[Proposition~3.3.4]{kedlaya-part2},
and by extension \cite[Proposition~3.4.5]{kedlaya-part2},
carry over with minor modifications; alternatively, one may simply
invoke the amended \cite[Proposition~3.2.6]{kedlaya-part2} to reduce to
the case of $k$ perfect.

The one other argument that is affected is the proof of
\cite[Lemma~4.3.1]{kedlaya-part2}. 
By the amended \cite[Proposition~3.2.6]{kedlaya-part2}, we may assume
$k$ is algebraically closed.
With this assumption, 
even though the quasiresolution $(f_1, j_1)$ is now only defined over 
$\ell^{q^{-n}}$ for some $n$, the recipe for construction the quasiresolution
$(f_2, j_2)$ carries over without change.

\subsection{Removing the discreteness hypothesis}

Throughout this series of papers, we have been forced in various places
to assume that $K$ is not an arbitrary complete nonarchimedean field,
but rather a complete discretely valued field. This may ultimately not
be necessary, as the approach of Andr\'e \cite{andre} and Mebkhout
\cite{mebkhout} to the $p$-adic local monodromy theorem can be extended
to the nondiscrete case. (It is unclear whether our approach using Frobenius
slope filtrations \cite{kedlaya-local} can be so extended.)

It is thus worth cataloging exactly where the discreteness hypothesis
is being used in our work, and pointing out what may potentially be done
to relax it.
\begin{itemize}
\item
In \cite{kedlaya-part1}, the only use of the discreteness hypothesis is to
exploit Shiho's theory of convergent log-isocrystals. One could make an
\emph{ad hoc} construction for the case of log-schemes associated to smooth
pairs, without relying on the discreteness hypothesis.
\item
In \cite{kedlaya-part2}, the discreteness hypothesis is used essentially
in Theorem~4.2.1, a result on the full faithfulness of restriction
from a category of overconvergent $F$-isocrystals to the corresponding
category of convergent $F$-isocrystals. The restriction ultimately
occurs in a local problem, whose solution depends on the
$p$-adic local monodromy theorem (see for instance
\cite[Theorem~20.3.5]{kedlaya-course}); it may be possible to use
Christol-Mebkhout decomposition theory instead.
\item
In \cite{kedlaya-part3}, discreteness is needed because the proof of local
semistable reduction at monomial valuations ultimately relies on our 
extension of the $p$-adic local monodromy theorem to fake annuli
\cite{kedlaya-fake}, which uses Frobenius slope filtrations.
However, it should be possible to argue using Mebkhout's approach
to the local monodromy theorem, by using Christol-Mebkhout decomposition
theory to reduce to the case of a connection of rank 1. 
(The reader may detect echoes of that strategy 
in this paper, especially in \S~\ref{subsec:terminal unip}.)
The role of decomposition theory in this case would be played by
the higher-dimensional decomposition theory we gave with
Xiao in \cite{kedlaya-xiao}.
\item
In this paper, the discreteness hypothesis is used
starting in Definition~\ref{D:analytic}, in order
to use the formalism of analytic rings from
\cite{kedlaya-slope}. It should be possible to circumvent
this hypothesis with an improved formalism.
\end{itemize}

\end{document}